\documentclass[preprint,3p,12pt]{elsarticle}

\newcommand{\pathLaTeX}{.}

\usepackage[utf8]{inputenc}
\usepackage[ngerman,english]{babel}
\usepackage[T1]{fontenc}
\usepackage{marvosym} 

\usepackage{lmodern} 
\usepackage{microtype}

\usepackage[numbers]{natbib}
\usepackage{graphicx}
\usepackage{caption}
\usepackage{subcaption}
\usepackage{xcolor}
\usepackage{tabularx}

\usepackage[final]{animate}

\usepackage{mathtools}
\usepackage{amssymb}
\usepackage{amsfonts}
\usepackage{amsbsy} 
\usepackage{amsopn} 
\usepackage{amstext} 
\usepackage{amsxtra} 
\usepackage{bbm} 

\usepackage{cancel} 
\usepackage{amscd} 
\usepackage[all]{xy} 
\usepackage{esvect} 
\usepackage{esint} 
\usepackage{empheq} 
\usepackage{eqnarray} 

\usepackage{lscape}
\usepackage{blindtext}
\usepackage{listings}
\usepackage{url}
\usepackage{textcomp} 
\usepackage{exscale} 
\usepackage{xspace} 
\usepackage{scalerel} 
\usepackage{stackengine} 
\usepackage{hhline}
\usepackage{lineno}

\PassOptionsToPackage{svgnames,dvipsnames}{xcolor}
\usepackage{tcolorbox}
\tcbuselibrary{skins,breakable}

\usepackage{amsthm} 

\theoremstyle{definition}
\newtheorem{theorem}{Theorem}
\newtheorem{lemma}{Lemma}
\newtheorem{corollary}{Corollary}
\theoremstyle{remark}
\newtheorem{remark}{Remark}
\newtheorem{example}{Example}

\addto\captionsenglish{}

\usepackage{varioref} 
\usepackage[colorlinks=true,citecolor=blue,urlcolor=blue]{hyperref} 
\usepackage[capitalise]{cleveref} 

\renewcommand{\l}{\ensuremath{\left(}} 
\renewcommand{\r}{\ensuremath{\right)}} 
\newcommand{\lb}{\ensuremath{\left[}}
\newcommand{\rb}{\ensuremath{\right]}}
\newcommand{\lc}{\ensuremath{\left\lbrace}}
\newcommand{\rc}{\ensuremath{\right\rbrace}}

\newcommand{\dx}{\ensuremath{\,\mathrm{d}\vec x}}
\newcommand{\dhx}{\ensuremath{\,\mathrm{d}\vec{\hat x}}}
\newcommand{\ds}{\ensuremath{\,\mathrm{d} s}}

\renewcommand{\d}{\,\mathrm{d}} 

\newcommand{\N}{\ensuremath{\mathbb{N}}}

\newcommand{\R}{\ensuremath{\mathbb{R}}}

\newcommand{\diag}{\mathrm{diag}}
\let\vphi\phi
\renewcommand{\phi}{\ensuremath{\varphi}}

\renewcommand{\epsilon}{\ensuremath{\varepsilon}}

\renewcommand{\theta}{\ensuremath{\vartheta}}
\let\vrho\rho
\renewcommand{\rho}{\ensuremath{\varrho}}
\newcommand*{\bin}[2]{\ensuremath{
\begin{pmatrix}#2\\#1\end{pmatrix}}}
\renewcommand{\vec}[1]{\ensuremath{\boldsymbol{#1}}}
\newcommand{\mat}[1]{\ensuremath{\mathbf{#1}}}

\definecolor{myred}{RGB}{182,59,87}
\definecolor{myblue}{RGB}{116,135,197}
\definecolor{mygreen}{RGB}{52,152,126}
\definecolor{myorange}{RGB}{227,149,87}

\newcommand*\reallywidehat[1]{%
\savestack{\tmpbox}{\stretchto{%
\scaleto{%
\scalerel*[\widthof{\ensuremath{#1}}]{\kern-.6pt\bigwedge\kern-.6pt}%
{\rule[-\textheight/2]{1ex}{\textheight}}
}{\textheight}%
}{0.5ex}}%
\stackon[1pt]{#1}{\tmpbox}%
}

\makeatletter
\newcommand*{\eqqcolon}{=\mathrel{\rlap{%
\raisebox{0.38ex}{$\m@th\cdot$}}%
\raisebox{-0.38ex}{$\m@th\cdot$}}%
}
\makeatother

\NewEnviron{alignS}{
\begin{align*}
\scalebox{0.75}{$\BODY$}
\end{align*}
}

\definecolor{keywordcolor}{rgb}{0,0.25,0.5}
\definecolor{commentcolor}{rgb}{0.2,0.5,0.2}
\definecolor{stringcolor}{rgb}{0.5,0.5,0.2}
\crefname{listing}{Algorithm}{Algorithms}

\journal{Computers \& Mathematics with Applications}

\begin{document}
\begin{frontmatter}

\title{Monolithic convex limiting in discontinuous Galerkin discretizations of hyperbolic conservation laws}

\author{Hennes Hajduk}
\ead{hennes.hajduk@math.tu-dortmund.de}

\address{Institute of Applied Mathematics (LS III), TU Dortmund
University,\\ Vogelpothsweg 87, D-44227 Dortmund, Germany}

\begin{abstract}
In this work we present a framework for enforcing discrete maximum principles in discontinuous Galerkin (DG) discretizations. 
The developed schemes are applicable to scalar conservation laws as well as hyperbolic systems.
Our methodology for limiting volume terms is similar to recently proposed methods for continuous Galerkin approximations, while DG flux terms require novel stabilization techniques.
Piecewise Bernstein polynomials are employed as shape functions for the DG spaces, thus  facilitating the use of very high order spatial approximations.
We discuss the design of a new, provably invariant domain preserving DG scheme that is then extended by state-of-the-art subcell flux limiters to obtain a high-order bound preserving approximation.
The limiting procedures can be formulated in the semi-discrete setting. Thus convergence to steady state solutions is not inhibited by the algorithm.
We present numerical results for a variety of benchmark problems. Conservation laws considered in this study are linear and nonlinear scalar problems, as well as the Euler equations of gas dynamics and the shallow water system.
\end{abstract}

\end{frontmatter}

\section{Introduction}

For many partial differential equations (PDE) the analytical solution is known to be member of a convex set, called \textit{invariant domain}, of physically admissible states \cite{berthon2005,guermond2016}.
Stabilized finite element methods for hyperbolic conservation laws often rely on limiters or other shock-capturing techniques to prevent the numerical approximation from dropping out of the invariant domain.
Computational methods which prohibit the occurrence of such nonphysical solutions are therefore called \textit{invariant domain preserving} (IDP) \cite{guermond2016,guermond2017,guermond2018}.
Many representatives of such schemes (e.g. \cite{badia2017,barrenechea2016,guermond2016,kuzmin2012a,lohmann2017}) use algebraic flux correction (AFC) techniques to
combine a (provably IDP) \textit{low order method} with a corresponding high-order \textit{target scheme}.

Low order methods for AFC are obtained by performing mass lumping and adding diffusive matrices to the semi-discrete formulation.
The minimal amount of such stabilization  for scalar problems is provided by \textit{discrete upwinding} \cite{kuzmin2001,kuzmin2002}.
For a general hyperbolic problem, the appropriate amount of \textit{graph viscosity} is proportional to the maximal wave speed. Guermond and Popov \cite{guermond2016} analyzed the so-defined \textit{Rusanov-type dissipation }\cite{abgrall2006,guermond2014} and proved the IDP property of the resulting low-order method using its representation
in terms of the so-called \textit{bar states}.
Lohmann et al. \cite{lohmann2017} improved discrete upwinding by preconditioning the convective matrices in a way that makes this low order scheme better applicable to high order finite element spaces.
Based on this upwinding strategy, various high order extensions were considered in \cite{hajduk2020a,kuzmin2020a,lohmann2017}.
The concept of residual distribution \cite{abgrall2006,abgrall2010} makes it possible to design matrix-free low order methods for scalar equations \cite{hajduk2020a,hajduk2020b}.
These approaches provide an efficient alternative to matrix-based upwinding but their generalization to systems is not straightforward.

Most of the limiters mentioned so far utilize the \textit{Flux-Corrected-Transport} (FCT) methodology \cite{boris1973,kuzmin2012,zalesak1979}.
This class of schemes is based on the fully discrete formulation and predictor-corrector algorithms which do not converge to steady-state solutions.
Monolithic limiting techniques were recently designed in \cite{hajduk2020b,kuzmin2012a,lohmann2019} to overcome this shortcoming of FCT and construct nonlinear problems with well-defined residuals.
Building on the analysis of Guermond and Popov \cite{guermond2016}, Kuzmin \cite{kuzmin2020} proposed a new monolithic limiter for enforcing the IDP property for general hyperbolic problems discretized using (multi-)linear continuous finite elements.
A generalization to higher order space discretizations can be found in \cite{kuzmin2020a}.

In this work, we extend these concepts to discontinuous Galerkin (DG) approximations of scalar equations and hyperbolic systems. 
Classical DG discretizations make use of total variation bounded (TVB) slope limiters or weighted essentially non-oscillatory (WENO) schemes (e.g. \mbox{\cite{barth1989,krivodonova2007,kuzmin2010,zhang2010}})
rather than algebraic limiting techniques.
However, the AFC methodology can be advantageous in the DG setting as well \cite{anderson2017,badia2017,hajduk2020a,pazner-preprint} since it guarantees the IDP property, in contrast to geometric limiters.
A~high order IDP-DG scheme for nonlinear hyperbolic conservation laws was recently developed by Pazner \cite{pazner-preprint} who constrained a collocated spectral element approximation using FCT.
In their recent work \cite{guermond2019}, Guermond et al. mentioned the possibility of unified algebraic limiting for fluxes and volume integrals of DG schemes.
In the present paper, we handle the DG fluxes separately, which enables us to sparsify the element matrices of the high order discrete gradient operator via the application of the local mass lumping preconditioner proposed by Lohmann et al. \cite{lohmann2017}.
All currently available AFC-DG methods are designed for linear scalar problems and/or equipped with FCT limiters which inhibit convergence to steady state solutions.
The methodology that we propose yields the first monolithic limiting framework for high order DG discretizations of hyperbolic systems.

The approach presented in this work employs Bernstein polynomials as basis functions, as has been done previously, e.g., in \cite{abgrall2010,ainsworth2011,anderson2017,kirby2011,lohmann2017}).
It is suitable for tensor-product elements (lines in 1D, quadrilaterals in 2D, hexahedra in 3D) -- which we conflate using the term \textit{box elements} -- as well as simplices and prisms.
Limiting for volume terms works similarly to the corresponding methods in \cite{guermond2016,kuzmin2020,kuzmin2020a}, while DG flux terms are handled by a generalization of our approach for the advection equation \cite{hajduk2020a}.
The concept of bar states, which was originally introduced for volume terms of continuous Galerkin discretizations \cite{guermond2016,guermond2018,kuzmin2020}, translates naturally to the DG context leading to simple proofs of the IDP property.
The monolithic limiting strategy and customized treatment of DG flux terms distinguish our approach from the 
{\it discretization-independent} AFC tools suggested in \cite{guermond2019}.
The presented methods are tested on benchmark configurations for the advection and Burgers problems, as well as for the systems of Euler and shallow water equations.

The structure of this article is as follows:
\cref{sec:dg} describes the standard DG discretization, followed by the derivation of the low order method in \cref{sec:low}. 
Limiting aspects are discussed in \cref{sec:lim}. Numerical examples are presented in \cref{sec:results} and conclusions are drawn in \cref{sec:conclusion}.
Further properties of the methods are detailed in the Appendix.

\section{Standard DG discretization}\label{sec:dg}
Let $U(\vec x,t) \in \R^m,~m \in \N$ be the vector of conserved unknowns at location $\vec x \in \R^d,~d\in \{1,2,3\}$ and time $t \geq 0$.
We consider the initial value problem
\begin{subequations}\label{eq:ivp}
\begin{align}
\frac{\partial U}{\partial t} + \nabla \cdot \mat F(U) =\;& 0\phantom{u_0} \qquad \mbox{in } \Omega \times \R_+, \label{eq:pde}\\
U(\cdot,0) =\;& U_0\phantom{0} \qquad \mbox{in }\Omega, \label{eq:ic}
\end{align}
\end{subequations}
where $\Omega \subseteq \R^d$ is a Lipschitz domain, 
$\mat F(U) \in \R^{m \times d}$ is the hyperbolic flux matrix, and $U_0(\vec x) \in \R^m$ denotes some initial datum.
A convex set $\mathcal G \subseteq \R^m$ such that $U_0(\vec x) \in \mathcal G$ for all $\vec x \in \Omega$, is called an invariant set of \eqref{eq:ivp} if $U(\vec x,t) \in \mathcal G$ for all $\vec x\in\Omega$ and $t\geq0$ \cite{guermond2016}.
To obtain a well-posed problem, \eqref{eq:ivp} has to be equipped with suitable boundary conditions on $\Gamma = \partial\Omega$.
In the scalar case ($m=1$), boundary values $U_{\textrm{in}}$ are imposed weakly on the inflow boundary $\Gamma_- = \{\vec x \in \Gamma:~\mat F^\prime(U) \cdot \vec n(\vec x) < 0\},$ where $\vec n$ is the unit outward normal to $\Gamma$.
The definition of the invariant set $\mathcal G=[U^{\min},U^{\max}]$ must be modified to include $U_{\textrm{in}}$ \cite{guermond2016,kuzmin2020}.
For systems of conservation laws ($m>1$) one has to distinguish between various boundary types, some of which are mentioned in \cref{sec:results}.
Alternatively, one can use periodic domains to avoid the issue of handling boundary conditions altogether.

To discretize \eqref{eq:pde}, we decompose $\Omega$ into elements $K^e$ of a triangulation $\mathcal T_h = \lc K^1,\hdots,K^{E} \rc$, such that $\overline{\Omega} = \cup_{e=1}^{E} K^e$.
Every $K^e \in \mathcal T_h$ is assumed to be the image of a certain reference element $\hat K \subseteq \R^d$ under a $C^1$-mapping $\vec F^e$, i.e., $K^e = \vec F^e (\hat K)$, and $\mat J^e(\vec x) \coloneqq \nabla \vec F^e(\vec x) \in \R^{d\times d}$ denotes the Jacobian of $\vec F^e$.
For $p \in \N_0$, we define standard polynomial spaces $\mathbb V_p(K^e)$ based on the geometry of elements
\begin{align*}
\mathbb V_p(K^e) = \begin{cases}
\mathbb P_p(K^e) & \mbox{if } K^e \mbox{ is a simplex,} \\
\mathbb Q_p(K^e) & \mbox{otherwise.}
\end{cases}
\end{align*}
For a given mesh and polynomial degree, the DG finite element spaces read
\begin{align}\label{eq:dgspace}
\mathbb V_p \l \mathcal T_h \r = \lc \phi_h \in L^2(\Omega):~ \phi_h \mid_{K^e} \in \mathbb V_p(K^e)~ \forall K^e \in \mathcal T_h \rc.
\end{align}
A set of basis functions for \eqref{eq:dgspace} can have support which is local to only one respective element.
Hence, for each element, we can choose basis functions $\lc \phi_i^e \rc_{i=1}^{N}$ of $\mathbb V_p(K^e)$, where $N = N(e) = \dim \mathbb V_p(K^e)$ is the local number of unknowns.
On $K^e$, we can thus write the solution as
\begin{align}\label{eq:u-elem}
U_h^e(\vec x,t) \coloneqq U_h(\vec x,t)\mid_{K^e} = 
\sum_{i=1}^N u_i^e(t) \, \phi_i^e(\vec x), \qquad u_i^e(t) \in \R^m.
\end{align}
Testing \eqref{eq:pde} with $V_h \in \mathbb V_p(K^e)^d$ and using integration by parts yields
\begin{align}\label{eq:int-parts}
\int_{K^e}  V_h \cdot \frac{\partial U_h^e}{\partial t}\dx
-\int_{K^e} \nabla V_h : \mat F(U_h^e)\dx 
+\int_{\partial K^e} V_h \cdot \mat F(U_h^e) \, \vec n^e \ds = 0,
\end{align}
where, $\cdot$ and $:$ denote the standard vector and matrix scalar products and $\vec n^e$ is the outward unit normal to $K^e$.
The boundary integral can be split into integrals over separate open \textit{faces} $\Gamma_k^e$, such that $\partial K^e = \cup_{k=1}^{n_f} \overline{\Gamma_k^e}$. 
In this context, we use the term face to describe the $(d-1)$-dimensional interfaces with other elements (points in 1D, lines in 2D,...).
By $n_f = n_f(e)$ we denote the cardinality of this partition (e.g. $n_f=2$ in 1D, $n_f=3$ for triangular elements, $n_f=4$ for quadrilaterals,...).
Due to the discontinuous space approximation \eqref{eq:dgspace}, $U_h^e$ is not uniquely defined in $\vec x \in \Gamma_k^e$; its one-sided limits are given by
\begin{align*}
U_h^e(\vec x,t) = U_h^{e,-}(\vec x,t) \coloneqq\;& \lim_{\epsilon \to 0_+}
U_h^e(\vec x - \epsilon \vec n^e, t),\\
\hat U_h^e(\vec x,t) = U_h^{e,+}(\vec x,t) \coloneqq\;& \lim_{\epsilon \to 0_+} U_h^e(\vec x + \epsilon \vec n^e, t),
\end{align*}
i.e., values from the interior are denoted by $U_h^e$ and exterior ones by $\hat U_h^e$.

To obtain local conservation, we replace $\mat F(U_h^e) \, \vec n^e$ in \eqref{eq:int-parts} by a numerical flux $\mathcal H(U_h^e, \hat U_h^e; \vec n^e)$, satisfying the following properties
\begin{subequations}\label{eq:riem-cond}
\begin{align}
\mathcal H(U,U; \vec n^e) = \mat F(U) \, \vec n^e \qquad&\mbox{(consistency)}, \\
\mathcal H(U,V;\vec n^e) + \mathcal H(V,U; -\vec n^e)= 0 \qquad&\mbox{(conservation)}.
\end{align}
\end{subequations}
In our numerical experiments, we use the local Lax-Friedrichs flux
\begin{align}\label{eq:LaxFriedrichs}
\mathcal H(U,V;\vec n^e) = \frac{1}{2}\l \mat F(U) + \mat F(V) \r 
\vec n^e + \frac{\lambda}{2} (U - V),
\end{align}
where $\lambda$ is an estimate on the speed of fastest wave propagating along $\vec n^e$
\begin{align}\label{eq:wave-speed}
\lambda=\lambda(U,V;\vec n^e) \geq \;& \max_{\omega \in [0,1]} \mathrm{spr}\l \mat A(\omega U + (1-\omega) V;\vec n^e) \r \\
\mat A(U;\vec n^e) =\;& \frac{\partial}{\partial U} \l \mat F(U) \, \vec n^e \r \in \R^{m\times m}, \notag
\end{align}
and $\mathrm{spr}(\cdot)$ denotes the spectral radius of a matrix.
For scalar conservation laws having a convex flux $\mat F$, the maximum in \eqref{eq:wave-speed} is attained at one of the two states (either $U$ or $V$).
Practical upper bounds for more general problems such as the Euler and shallow-water equations can be found in \cite{guermond2016a,guermond2018a}.

\begin{remark}
Our methodology is in no way restricted to using the Lax-Friedrichs flux in the target scheme.
In fact, we show that any flux satisfying \eqref{eq:riem-cond} fits into our framework.
However, the low order method presented in the next section, does indeed rely on a Lax-Friedrichs flux.
\end{remark}
Setting one component of $V_h$ equal to $\phi_i^e$ and the other components equal to zero, we test \eqref{eq:int-parts} with $V_h$ and arrive at the semi-discrete system of equations for the target scheme
\begin{align}\label{eq:semi-disc-dg}
\int_{K^e} \phi_i^e \, \frac{\partial U_h^e}{\partial t} \dx
= \int_{K^e} \mat F(U_h^e) \, \nabla \phi_i^e \dx
- \sum_{k=1}^{n_f} \int_{\Gamma_k^e} \phi_i^e \, \mathcal H(U_h^e,\hat U_h^e; \vec n^e)\ds,
\end{align}
where $i \in \{1,\dots,N\}$ and $e \in \{1,\dots,E\}$.


\section{Derivation of the low order method}\label{sec:low}

The standard DG discretization presented in \cref{sec:dg} is valid for any basis of $\mathbb V_p(\mathcal T_h)$.
However, the approach presented in this section relies on the use of Bernstein (also called B\'ezier) polynomials $b_i^p$. 
On $[0,1]$ these functions are defined as
\begin{align*}
b_i^p(x) = \bin{i}{p} (1-x)^{p-i} x^i, \quad i \in \{0,\dots,p\}.
\end{align*}
Definitions for other reference elements and further properties can be found in the literature, e.g., \cite{abgrall2010,ainsworth2011,kirby2011,lohmann2017}.
Bernstein polynomials possess a number of advantages compared to other bases.
In particular, they attain values in $[0,1]$ and form a partition of unity.
Therefore, a numerical approximation defined as a linear combination of Bernstein polynomials is bounded in the following way:
\begin{align*}
\min_{j\in \{1,\dots,N\}} u_j^e \leq \sum_{j=1}^N u_j^e \varphi_j^e(\vec x) \leq \max_{j\in \{1,\dots,N\}} u_j^e.
\end{align*}
The latter property makes the Bernstein basis a natural choice for constraining numerical solutions to satisfy discrete maximum principles which do not necessarily hold, e.g., for an interpolatory Lagrange basis even if the coefficients of the finite element approximation are bounded as desired.
Similarly to Lagrange basis polynomials, their Bernstein counterparts can be associated with certain \textit{nodes} $\vec x_i^e \in K^e$, which are uniformly distributed within the elements.
Due to the DG setting, we need to distinguish between element-local nodes $\vec x_i^e$ and physical locations $\vec x_I \in \lc \vec x_i^e \in \Omega : e \in \{1,\dots,E\},~i \in \{1,\dots,N\} \rc$ of these points, which are assigned unique global numbers $I$ and may be shared by boundary nodes of adjacent elements.

The semi-discrete formulation presented in \cref{sec:dg} satisfies the IDP property only when piecewise constant basis functions are employed.
To obtain a low order method that is IDP in general, we make a number of modifications to \eqref{eq:semi-disc-dg}.
A first simplification due to the use of Bernstein polynomials is that the boundary integral reduces to integration over the set of faces $\Gamma_k^e,\ k\in\mathcal F_i$ which contain the corresponding node.
Reversing integration by parts, we write \eqref{eq:semi-disc-dg} in the equivalent form
\begin{align*}
\int_{K^e} \phi_i^e \, \frac{\partial U_h^e}{\partial t} \dx
=\;& -\int_{K^e} \phi_i^e\, \nabla \cdot \mat F(U_h^e) \dx \\ &
- \sum_{k \in \mathcal F_i} \int_{\Gamma_k^e} \phi_i^e \lb \mathcal H(U_h^e,\hat U_h^e; \vec n^e) - \mat F(U_h^e)\,\vec n^e \rb \ds.
\end{align*}
Next, we linearize the volume term using the \textit{group finite element formulation}
\begin{align*}
\mat F_j^e \coloneqq \mat F(u_j^e),\quad j \in \{1,\dots,N\}, \quad \mat F_h^e \coloneqq \sum_{j=1}^N \mat F_j^e \,\phi_j^e \approx \mat F(U_h^e) = \mat F\l \sum_{j=1}^N u_j^e \phi_j^e\r.
\end{align*}
This approximation produces
\begin{align*}
& \sum_{j=1}^N m_{ij}^e \frac{\d u_j^e}{\d t}
=\; -\sum_{j=1}^N \mat F_j^e \, \vec c_{ij}^e
- \sum_{k \in \mathcal F_i} \int_{\Gamma_k^e} \phi_i^e \lb \mathcal H(U_h^e,\hat U_h^e; \vec n^e) - \mat F(U_h^e)\,\vec n^e \rb \ds, \\&
\mbox{where} \quad m_{ij}^e = \int_{K^e} \phi_i^e\,\phi_j^e\dx, \quad
\vec c_{ij}^e = \int_{K^e} \phi_i^e \nabla \phi_j^e\dx.
\end{align*}
Since Bernstein polynomials form a partition of unity, the matrices $\vec C^e =(C_1^e,\dots,C_d^e)^T= (\vec c_{ij}^e)_{i,j=1}^N$ have zero row sums, which implies $\sum_{j=1}^N \mat F_j^e \, \vec c_{ij}^e = \sum_{j=1}^N \l \mat F_j^e - \mat F_i^e \r \, \vec c_{ij}^e$.

The next step toward constructing an IDP low order method is mass lumping.
At this stage, the consistent mass matrix $M_C^e= (m_{ij}^e)_{i,j=1}^N$ is transformed into its lumped counterpart $M_L^e = \diag(m_i^e)_{i=1}^N$ by applying the \textit{preconditioning operator} $P^e = M_L^e(M_C^e)^{-1}$.
Following the approach proposed in \cite{kuzmin2020a,lohmann2017}, we apply $P^e$ to the element matrices stored in $\vec C^e$ as well, which yields
\begin{align}\label{eq:derive-low1}
m_i^e \frac{\d u_i^e}{\d t} =\;& -\sum_{j\in \tilde{\mathcal N}_i} \l \mat F_j^e - \mat F_i^e \r \tilde{\vec c}_{ij}^e 
-\sum_{k \in \mathcal F_i} \int_{\Gamma_k^e} \phi_i^e \lb \mathcal H(U_h^e,\hat U_h^e; \vec n^e) - \mat F(U_h^e)\,\vec n^e \rb \ds,  \\
\mbox{where}&\quad \tilde{\vec C}^e = (\tilde C_1^e,\dots,\tilde C_d^e)^T = \l\tilde{\vec c}_{ij}^e\r_{i,j=1}^N,
\quad \tilde C_k^e = M_L^e(M_C^e)^{-1}C_k^e.\notag
\end{align}
One can show that the element matrices $\tilde{C}_k^e$ also have zero row sums and their entries are non-zero only if nodes $\vec x_i^e$ and $\vec x_j^e$ are nearest neighbors within the element \cite{kuzmin2020a,lohmann2017}.
This fact allows us to restrict summation to the smaller index set $\tilde{\mathcal N}_i\subset \{1,\dots,N\}$ consisting of only the indices $j$ whose nodes are nearest neighbors to node $\vec x_i^e$.

\begin{remark}\label{rem:ill}
For high order elements, consistent mass matrices $M_C^e$ become extremely ill-conditioned and, therefore, one should not rely on their inversion to compute $\tilde C_k^e$.
In the appendix, we present formulas for the entries of $\tilde C_k^e$ on box elements, as well as a MATLAB code for computing $\tilde C_k^e$ on simplices.
\end{remark}

Next, we introduce dissipative matrices $D^e=(d_{ij}^e)_{i,j=1}^N$ defined by \cite{guermond2016,kuzmin2020a,kuzmin2002}
\begin{align}\label{eq:dMat}
&d_{ij}^e = \begin{cases}
\max\{\left|\tilde{\vec c}_{ij}^e\right| \lambda_{ij}^e,\left|\tilde{\vec c}_{ji}^e\right| \lambda_{ji}^e\} & \mbox{if } i\neq j, \\
-\sum_{k=1}^N d_{ik}^e & \mbox{otherwise},
\end{cases} \\
&\mbox{where} \quad \lambda_{ij}^e = \lambda( u_i^e, u_j^e; \vec n_{ij}^e), 
\quad \vec n_{ij}^e = \frac{\tilde{\vec c}_{ij}^e}{\left|\tilde{\vec c}_{ij}^e \right|}. \notag
\end{align}
To stabilize the volume integral, the term $\sum_{j\in\tilde{\mathcal N}_i} d_{ij}^e ( u_j^e - u_i^e)$ is added to \eqref{eq:derive-low1}.
Note that summation is over the compact stencil $\tilde{\mathcal N}_i$.
The sparsity of $D^e$ follows from \eqref{eq:dMat} and is the main advantage of applying the preconditioner $P^e$ to the element matrices of the discrete gradient operator.
A sparsified diffusion operator implies that $d_{ij}^e ( u_j^e - u_i^e)=0$ if $i$ and $j$ are far away nodes within a high order element.
Diffusive fluxes between such nodes are nonphysical and tend to be large.
The compact-stencil version produces nonvanishing fluxes only for pairs of nearest neighbor nodes.
The necessity of such sparsification for high order spaces has been recognized in \cite{hajduk2020a,kuzmin2020a,lohmann2017,pazner-preprint}.

\begin{remark}\label{rem:cross}
The analysis in the appendix of \cite{kuzmin2020a} proves the {\bf all}-nearest-neighbors sparsity pattern for box elements in more than one dimension.
In fact, the element matrices $C_k^e$ have more zero entries than this analysis suggests.
We prove in the appendix of this paper, that nodes which are closest neighbors in the diagonal direction, produce a zero entry in $\tilde C_k^e$.
On the one hand, this remarkable property implies lower computational cost since fewer pairs of nodes need to be considered (e.g., $2d$ neighbors for one interior node of an element instead of $3^d-1$). 
On the other hand, the low order method bears an interesting resemblance to the spectral element method in \cite{pazner-preprint}:
Dissipation is only added in Cartesian directions, not along diagonals.
\end{remark}

Having discussed the stabilization of volume terms, let us now modify the DG face integrals in order to obtain the low order IDP scheme.
To this end, we generalize our interface treatment for the advection equation \cite{hajduk2020a,hajduk2020b}.
In the 1D case, face contributions to the global system matrix already have the right signs and do not inhibit the IDP property.
Due to positivity of Bernstein polynomials and the use of upwind (Lax-Friedrichs) fluxes, off-diagonal entries corresponding to the 1D face terms (i.e. contributions from other elements) are nonnegative, while diagonal entries are nonpositive.
In multiple dimensions, the off-diagonal entries within the diagonal blocks of the matrix need to be ``lumped'' in a conservative manner.
This \textit{flux lumping} is accomplished by replacing interior and exterior contributions $U_h^e,\hat U_h^e$ by their nodal counterparts, i.e., solution coefficients $u_i^e,\hat u_i^e$, which yields
\begin{align*}
m_i^e\frac{\d u_i^e}{\d t} =\;& \sum_{j\in \tilde{\mathcal N}_i} \lb d_{ij}^e (u_j^e - u_i^e) - \l \mat F_j^e - \mat F_i^e \r \tilde{\vec c}_{ij}^e \rb \\&
- \sum_{k \in \mathcal F_i}  \int_{\Gamma_k^e} \phi_i^e \lb \mathcal H(u_i^e,\hat u_i^e; \vec n^e) - \mat F(u_i^e)\,\vec n^e \rb \ds.
\end{align*}
Note that for nodes $i$ located at a junction of two or more interfaces (e.g. a vertex in 2D) $\hat u_i^e$ implicitly depends on the index $k$ of the face (in 2D, there are two such neighbors, one for each face).
At this stage, the use of Lax-Friedrichs flux \eqref{eq:LaxFriedrichs} is beneficial since it allows us to write the face term as 
\begin{align*}
\mathcal H(u_i^e,\hat u_i^e; \vec n^e) - \mat F(u_i^e)\,\vec n^e =\,&
\frac{1}{2} \lb \l \mat F(\hat u_i^e) - \mat F(u_i^e) \r \vec n^e + \lambda_i^e\,(u_i^e - \hat u_i^e)\rb \\=\;&
\frac{1}{2} \lb \big( \hat{\mat F}_i^e - \mat F_i^e \big)\, \vec n^e + \lambda_i^e\,(u_i^e - \hat u_i^e)\rb,
\end{align*}
where we introduced the notation $\hat{\mat F}_i^e \coloneqq \mat F(\hat u_i^e)$, and the nodal wave speed $\lambda_i^e = \lambda(u_i^e,\hat u_i^e; \vec n^e)$.
In summary, the low order scheme reads
\begin{align}\label{eq:loworder}
m_i^e\frac{\d u_i^e}{\d t} = \;& \sum_{j\in \tilde{\mathcal N}_i} \lb d_{ij}^e (u_j^e - u_i^e) - \l \mat F_j^e - \mat F_i^e \r \tilde{\vec c}_{ij}^e \rb 
\notag \\& + \sum_{k \in \mathcal F_i} \frac 1 2\int_{\Gamma_k^e} \phi_i^e\ds \lb \big( \mat F_i^e - \hat{\mat F}_i^e \big)\, \vec n^e + \lambda_i^e\,(\hat u_i^e - u_i^e) \rb.
\end{align}
Equation \eqref{eq:loworder} constitutes the semi-discrete form of our low order method.
Its IDP property follows from its equivalent \textit{bar state} form \cite{guermond2016}.
Similarly to volume integrals that can be expressed in terms of two-node bar states $\bar u_{ij}^e$, DG fluxes between nodes corresponding to the same physical location in two different elements can be written in terms of interfacial bar states $\bar u_{i,k}^e, k \in \mathcal F_i$.
The volume and face bar states of our DG method are defined by
\begin{align}\label{eq:barstates}
\bar u_{ij}^e = \frac{u_i^e+u_j^e}{2} - \frac{\l\mat F_j^e - \mat F_i^e\r \tilde{\vec c}_{ij}^e}{2 d_{ij}^e}, \qquad
\bar u_{i,k}^e = \frac{u_i^e + \hat u_i^e}{2} -
\frac{\big(\hat{\mat F}_i^e - \mat F_i^e\big)\, \vec n^e}{2 \lambda_i^e},
\end{align}
respectively.
Writing \eqref{eq:loworder} in terms of \eqref{eq:barstates} produces
\begin{align}\label{eq:barstateform}
&m_i^e\frac{\d u_i^e}{\d t} =  \sum_{j\in \tilde{\mathcal N}_i} 2d_{ij}^e \l \bar u_{ij}^e - u_i^e \r + \sum_{k \in \mathcal F_i}
2 d_{i,k}^e \l \bar u_{i,k}^e - u_i^e \r, \\
&\mbox{where}\quad d_{i,k}^e = \frac{\lambda_i^e}{2} \int_{\Gamma_k^e} \phi_i^e\ds\notag
\end{align}
and $d_{i,k}^e> 0$ due to the positivity of the Bernstein basis.
\begin{remark}
By definition, the volumetric  bar states $\bar u_{ij}^e$ and $\bar u_{ji}^e$ coincide if 
$\tilde{\vec c}_{ij}^e+\tilde{\vec c}_{ji}^e=0$.
This skew-symmetry condition
holds if (i) no preconditioning is performed for the element
matrices ${\vec C}^e$
of the discrete gradient operator and (ii)  at least one of the two
nodes is located in the interior of $K^e$.
The sparse preconditioned
element matrices $\tilde{\vec C}^e$ are not skew-symmetric.
Therefore, $\bar u_{ij}^e$ and $\bar u_{ji}^e$
may differ even if $i$ or $j$ is an internal node.
However, an interfacial bar state $\bar u_{i,k}^e$ always equals the
bar state $\bar u_{i^\prime,k^\prime}^{e^\prime}$ of the node $\mathbf{x}_{i^\prime}^{e^\prime}$
occupying the same physical location $\mathbf{x}_I$
in an adjacent element $K^{e^\prime}$.
Indeed, we have
\begin{align}\label{eq:fluxbarstateeq}
\bar u_{i^\prime,k^\prime}^{e^\prime} = \frac{u_{i^\prime}^{e^\prime} + \hat u_{i^\prime}^{e^\prime}}{2}
- \frac{\big( \hat{\mat F}_{i^\prime}^{e^\prime} - \mat F_{i^\prime}^{e^\prime} \big) \, \vec n^{e^\prime}}{2 \lambda_{i^\prime}^{e^\prime}} = 
\frac{\hat u_i^e + u_i^e}{2} - \frac{\big(\hat{\mat F}_i^e - \mat F_i^e\big)\, \vec n^e}{2 \lambda_i^e} = \bar u_{i,k}^e.
\end{align}
\end{remark}
In \cite{guermond2016}, the authors prove that $\bar u_{ij}^e$ is in the invariant set $\mathcal G$.
Following their proof, we now show that the same property holds for $\bar u_{i,k}^e$.
\begin{theorem}
Assume that the one-dimensional Riemann problem 
\begin{align}\label{eq:riemann} 
\frac{\partial U}{\partial t} + \frac{\partial}{\partial x} \l \mat F(U)\, \vec n^e \r = 0, \quad \mbox{in } \R \times \R_+, \qquad U(x,0) =
\begin{cases}
u_i^e & \mbox{if } x < 0, \\ \hat u_i^e & \mbox{if } x > 0
\end{cases}
\end{align}
has a unique solution $U(\vec n^e, u_i^e, \hat u_i^e)$.
Then the flux bar states $\bar u_{i,k}^e$ defined by \eqref{eq:barstates} belong to the invariant set $\mathcal G$.
\end{theorem}
\begin{proof}
Consider the \textit{fake time} $\tau = \left| \vec n^e \right| / (2 \lambda_i^e) = 1 / (2\lambda_i^e)$.
By definition \eqref{eq:wave-speed}, $\lambda_i^e$ is a guaranteed upper bound for the maximum wave speed $\lambda_{\max}$ of \eqref{eq:riemann}, i.e., $\tau\lambda_{\max} \leq \frac{1}{2}.$
Hence, following the analysis of the 1D Riemann problem in
\cite{guermond2016},
we can rewrite the averaged solution to \eqref{eq:riemann} as
\begin{align*}
\bar U(\vec n^e, u_i^e, \hat u_i^e) \coloneqq \int_{-\frac{1}{2}}^{\frac{1}{2}}  U(\vec n^e, u_i^e, \hat u_i^e)(x,\tau)\,\d x = \frac{u_i^e + \hat u_i^e}{2} 
- \frac{\big( \hat{\mat F}_i^e - \mat F_i^e \big)\, \vec n^e}{2 \lambda_i^e} = \bar u_{i,k}^e.
\end{align*}
Since $\bar U(\vec n^e, u_i^e, \hat u_i^e) \in \mathcal G$, this representation of $\bar u_{i,k}^e$ implies the statement of the theorem.
\end{proof}

The forward Euler time discretization of \eqref{eq:barstateform}
preserves the IDP property if the time step $\Delta t$ is sufficiently small.
The updated solution
\begin{align}\label{eq:euler}
\tilde u_i^e = u_i^e - \frac{2\Delta t}{m_i^e} \l \sum_{j\in \tilde{\mathcal N}_i} d_{ij}^e + \sum_{k \in \mathcal F_i} d_{i,k}^e \r u_i^e
+ \frac{2\Delta t}{m_i^e} \l \sum_{j\in \tilde{\mathcal N}_i} d_{ij}^e \bar u_{ij}^e  + \sum_{k \in \mathcal F_i} d_{i,k}^e \bar u_{i,k}^e \r
\end{align}
is a convex combination of $u_i^e$ and the bar states $\bar u_{ij}^e,~\bar u_{i,k}^e$, provided that 
\begin{align}\label{eq:cfl}
\frac{2\Delta t}{m_i^e} \l \sum_{j\in \tilde{\mathcal N}_i} d_{ij}^e + \sum_{k \in \mathcal F_i} d_{i,k}^e \r \leq 1.
\end{align}
\begin{corollary}
If \eqref{eq:cfl} holds, then $\tilde u_i^e$ is a convex combination of $u_i^e$ and the bar states, which are members of $\mathcal G$.
This proves the IDP property of \eqref{eq:loworder}.
\end{corollary}

\begin{remark}
It is possible to write \eqref{eq:barstateform} using a unified notation for volumetric and interfacial bar states.
Disguising the different nature of underlying approximations, the resulting representation resembles the bar state forms of the low order continuous finite element methods considered in \cite{guermond2016,kuzmin2020a}.
Let $\mathcal E^e$ denote the DG element stencil, i.e., the integer set containing the numbers of element $K^e$ and its common face neighbors.
Furthermore, define
\begin{align*}
\mathcal N_i^{e,e^\prime} = \begin{cases}
\tilde {\mathcal N}_i & \mbox{if } e=e^\prime, \\
\lc j \in \{1,\dots,N\}:~ \vec x_j^{e^\prime} = \vec x_i^e \rc & \mbox{otherwise},
\end{cases}
\end{align*}
and
\begin{align*}
d_{ij}^{e,e^\prime} = \begin{cases}
d_{ij}^e & \mbox{if } e=e^\prime, \\
d_{i,k}^e & \Gamma_k^e = K^e\cap K^{e^\prime},
\end{cases} \qquad
\bar u_{ij}^{e,e^\prime} = \begin{cases}
\bar u_{ij}^e & \mbox{if } e=e^\prime, \\
\bar u_{i,k}^e & \Gamma_k^e = K^e\cap K^{e^\prime}.
\end{cases}
\end{align*}
Then \eqref{eq:barstateform} becomes
\begin{align}\label{eq:alternative}
m_i^e\frac{\d u_i^e}{\d t} = \sum_{e^\prime \in \mathcal E^e} \sum_{j\in \mathcal N_i^{e,e^\prime}} 2d_{ij}^{e,e^\prime} \l \bar u_{ij}^{e,e^\prime} - u_i^e \r.
\end{align}
Since the volume and flux terms are treated differently in practice, we avoid using the compressed single-sum representation \eqref{eq:alternative}. However, it is quite remarkable that the bar state form of the low order IDP method analyzed in \cite{guermond2016} can be generalized to the DG setting in such a way.
\end{remark}

\begin{remark}
Condition \eqref{eq:cfl} provides a valuable property for time step control.
If the solution of \eqref{eq:barstateform} is advanced in time using the
forward Euler method, the IDP property is guaranteed for time steps
$\Delta t$ satisfying
\begin{align}\label{eq:dt}
\Delta t \leq \min_{e \in \{1,\dots,E\}} \min_{i \in \{1,\dots,N\}} \frac{m_i^e}{2 \l \sum_{j\in \tilde{\mathcal N}_i} d_{ij}^e + \sum_{k \in \mathcal F_i} d_{i,k}^e \r}.
\end{align}
In higher order strong stability preserving (SSP) Runge-Kutta methods \cite{gottlieb2001}, each stage corresponds to an update of the form \eqref{eq:euler}. Hence the validity of \eqref{eq:dt} is sufficient to ensure the IDP property for each stage and for the final result which represents a convex combination of IDP approximations.
\end{remark}

\section{Monolithic limiting approach}\label{sec:lim}

This section describes all mathematical aspects of our limiter starting with the definition of raw antidiffusive fluxes.
It then summarizes the subcell limiting approach of \cite{kuzmin2020} and extends it to the interfacial bar states of flux-corrected DG approximations.
Next, we describe the general methodology for limiting hyperbolic systems and define admissible bounds for the inequality constraints of each step.
For the reader's convenience, we give a brief summary of the presented limiting techniques at the end of this section.

\subsection{Raw antidiffusive fluxes}\label{sec:raw}
In order to adapt the convex limiting strategy originally proposed in \cite{kuzmin2020,kuzmin2020a} to the DG setting, we must first define raw antidiffusive terms $f_i^e,\, f_{i,k}^e$ which transform the low order method
\eqref{eq:loworder} into the compact-stencil
form 
\begin{align}\label{eq:unlimited}
m_i^e\frac{\d u_i^e}{\d t} = \;& \sum_{j\in \tilde{\mathcal N}_i} \lb d_{ij}^e (u_j^e - u_i^e) - \l \mat F_j^e - \mat F_i^e \r \tilde{\vec c}_{ij}^e \rb+f_i^e
\notag \\& + \sum_{k \in \mathcal F_i} \l \frac 1 2\int_{\Gamma_k^e} \phi_i^e\ds \lb \big( \mat F_i^e - \hat{\mat F}_i^e \big)\, \vec n^e + \lambda_i^e\,(\hat u_i^e - u_i^e) \rb + f_{i,k}^e \r
\end{align}
of the target scheme \eqref{eq:semi-disc-dg}. These correction terms
are given by
\begin{align}
f_i^e =\;& \sum_{j=1}^N \l m_i^e \delta_{ij} - m_{ij}^e\r \dot u_j^e - \sum_{j\in \tilde{\mathcal N}_i} \lb d_{ij}^e (u_j^e - u_i^e) - \l \mat F_j^e - \mat F_i^e \r \tilde{\vec c}_{ij}^e \rb \notag \\&
+ \int_{K^e} \mat F(U_h^e) \, \nabla \phi_i^e \dx 
- \sum_{k \in \mathcal F_i} \int_{\Gamma_k^e} \phi_i^e\, \mat F_i^e\, \vec n^e \ds \label{eq:anti-diff-el}
\end{align}
and
\begin{align}
f_{i,k}^e =\;& \int_{\Gamma_k^e} \phi_i^e \lb \mathcal H\l u_i^e,\hat u_i^e; \vec n^e\r - \mathcal H\l U_h^e,\hat U_h^e; \vec n^e\r \rb \ds.
\label{eq:anti-diff-flux}
\end{align}
The vector $\dot u^e = (\dot u_j^e)_{j=1}^N$ of nodal time derivatives in element $K^e$ is obtained by solving \eqref{eq:semi-disc-dg}.

\begin{lemma}
The semi-discrete target scheme \eqref{eq:semi-disc-dg} is equivalent to \eqref{eq:unlimited}.
\end{lemma}
\begin{proof}
It is easy to verify that substitution of $f_i^e$ and $f_{i,k}^e$ into
\eqref{eq:unlimited}
reverses the effect of algebraic manipulations
performed in Section \ref{sec:low}.
Moving the time derivative term to the left hand side, we rewrite \eqref{eq:unlimited} as
\begin{align*}
&\sum_{j=1}^N m_{ij}^e \frac{\d u_j^e}{\d t} = 
\int_{K^e} \mat F(U_h^e) \, \nabla \phi_i^e \dx
- \sum_{k \in \mathcal F_i} \int_{\Gamma_k^e} \phi_i^e\, \mat F_i^e \, \vec n^e \ds \\&
+ \sum_{k \in \mathcal F_i} \l \frac 1 2\int_{\Gamma_k^e} \phi_i^e\ds \lb \big( \mat F_i^e - \hat{\mat F}_i^e \big)\, \vec n^e + \lambda_i^e\,(\hat u_i^e - u_i^e) \rb \r \\&
+ \sum_{k \in \mathcal F_i} \int_{\Gamma_k^e} \phi_i^e \lb \mathcal H\l u_i^e,\hat u_i^e; \vec n^e\r - \mathcal H\l U_h^e,\hat U_h^e; \vec n^e\r \rb \ds \\&
= \int_{K^e} \mat F(U_h^e) \, \nabla \phi_i^e \dx
- \sum_{k \in \mathcal F_i} \int_{\Gamma_k^e} \phi_i^e \, \mathcal H(U_h^e,\hat U_h^e; \vec n^e)\ds \\&
+ \sum_{k \in \mathcal F_i} \int_{\Gamma_k^e} \phi_i^e\ds
\underbrace{\l \mathcal H\l u_i^e,\hat u_i^e; \vec n^e\r
- \frac 1 2\lb \big( \mat F_i^e + \hat{\mat F}_i^e \big)\, \vec n^e + \lambda_i^e\,(u_i^e - \hat u_i^e) \rb \r}_{= 0}.
%
\end{align*}
This semi-discrete scheme is, indeed, equivalent to \eqref{eq:semi-disc-dg}.
\end{proof}
\begin{lemma}\label{lem:f-sum}
The antidiffusive fluxes $f_i^e$ defined by \eqref{eq:anti-diff-el} satisfy $\sum_{i=1}^N f_i^e = 0$ for all $e \in \{1,\dots,E\}$.
\end{lemma}
\begin{proof}
By definition, the contributions of mass and dissipative matrices sum to zero.
Using the zero row sum property of $\tilde C_k^e$, the partition of unity property of the Bernstein basis and integration by parts, we find that
\begin{align*}
&\sum_{i=1}^N f_i^e = \sum_{i=1}^N \lb \sum_{j\in \tilde{\mathcal N}_i}\l\mat F_j^e - \mat F_i^e\r \tilde{\vec c}_{ij}^e -\sum_{k \in \mathcal F_i} \int_{\Gamma_k^e} \phi_i^e\,\mat F_i^e\,\vec n^e \ds \rb \\& = 
\sum_{i=1}^N \sum_{j=1}^N \mat F_j^e \l \tilde{\vec c}_{ij}^e - \vec c_{ij}^e
+ \vec c_{ij}^e \r - \sum_{i=1}^N \int_{\partial K^e} \phi_i^e\,\mat F_i^e\,\vec n^e \ds \\ &=
\sum_{i,j=1}^N \lb \mat F_j^e \l \tilde{\vec c}_{ij}^e - \vec c_{ij}^e \r
- \mat F_j^e \vec c_{ji}^e + \mat F_j^e \int_{\partial K^e} \phi_i^e \phi_j^e \vec n^e \ds \rb
- \sum_{i=1}^N\int_{\partial K^e} \phi_i^e\,\mat F_i^e\,\vec n^e \ds \\ &=
\sum_{j=1}^N \mat F_j^e \lb \sum_{i=1}^N \l \tilde{\vec c}_{ij}^e - \vec c_{ij}^e \r - \sum_{i=1}^N \vec c_{ji}^e \rb +
\sum_{i,j=1}^N \l \mat F_j^e - \mat F_i^e \r \int_{\partial K^e} \phi_i^e \phi_j^e \vec n^e \ds = 0.
\end{align*}
The final step exploits the partition of unity property and the fact that the matrices $\tilde C_k^e - C_k^e$ have zero column sums, cf. \cite{kuzmin2020a}.
\end{proof}
\begin{remark}
The antidiffusive DG fluxes $f_{i,k}^e$ contain the difference between the low and high order numerical fluxes.
They constitute a straightforward generalization of the corresponding terms arising in limiters for DG discretizations of advection problems \cite{hajduk2020a,hajduk2020b}.
As was the case there, $f_{i,k}^e$ vanishes in 1D because face integrals correspond to point evaluations and only one Bernstein polynomial is non-zero at the end points of a 1D element.
\end{remark}
The next lemma is a direct consequence of \eqref{eq:riem-cond} and of the fact that local basis functions from different elements coincide along the shared face if the corresponding nodes are associated with the same physical location.
\begin{lemma}
Let $K^{e^\prime}$ be the neighbor element of $K^e$ along face $\Gamma_k^e = \Gamma_{k^\prime}^{e^\prime}$ and $\vec x_{i^\prime}^{e^\prime}$ be the node in $K^{e^\prime}$ with the same physical location as node $\vec x_i^e \in K^e$.
Then the antidiffusive DG fluxes satisfy $f_{i,k} = - f_{i^\prime,k^\prime}^{e^\prime}$.
\end{lemma}

\subsection{Limiting strategy for volumetric fluxes}\label{sec:vol}

This section summarizes the methods originally proposed in \cite{kuzmin2020,kuzmin2020a} for limiting antidiffusive element contributions in continuous Galerkin methods.
These approaches are also applicable to DG volume terms.
Due to \cref{lem:f-sum}, we can split $f_i^e$ into fluxes $f_{ij}^e$ between pairs of nodes as follows:
\begin{align}\label{eq:fluxcond}
f_i^e = \sum_{j \in \tilde{\mathcal N}_i \setminus \{i\}} f_{ij}^e,\qquad f_{ij}^e = -f_{ji}^e \quad \forall j \in \tilde{\mathcal N}_i \setminus \{i\}.
\end{align}
To preserve the compact sparsity pattern of the low order scheme, it is
worthwhile to use a flux decomposition such that $f_{ij}^e$ is nonzero
if and only if $\vec x_i^e$ and $\vec x_j^e$ are closest neighbors.
Following the design principle in \cite{kuzmin2020a}, 
we use a natural decomposition of the term
$\sum_{j\in \tilde{\mathcal N}_i} d_{ij}^e (u_i^e - u_j^e)$ into
 antidiffusive fluxes $d_{ij}^e (u_i^e - u_j^e)$ which
possess the compact stencil property by construction.
The remaining components
\begin{align*}
q_i^e = f_i^e + \sum_{j\in \tilde{\mathcal N}_i} d_{ij}^e (u_j^e - u_i^e)
\end{align*}
also sum to zero and are now distributed in the following manner.
Let $\tilde M_C^e = (\tilde m_{ij}^e)_{i,j=1}^N$ and $\tilde M_L^e$ be consistent and lumped element matrices of the continuous, piecewise $\mathbb V_1$ shape functions on B\'ezier subcells of $K^e$ (as defined, e.g., in \cite{hajduk2020a,kuzmin2020a}).
On simplices, we define auxiliary vectors $v^e \in \R^N$ as solutions to $(\tilde M_L^e - \tilde M_C^e) v^e = q^e$ and use them to construct the antidiffusive fluxes
\begin{align}\label{eq:fij}
f_{ij}^e = \tilde m_{ij}^e (v_i^e -v_j^e) 
+ d_{ij}^e (u_i^e - u_j^e)
\quad \forall j \in \tilde{\mathcal N}_i \setminus \{i\}
\end{align}
satisfying \eqref{eq:fluxcond}.
Just as in the low order method, fluxes of box elements should be zero for diagonal neighbors. 
Therefore, we modify the element matrices $\tilde M_C^e$ of such Bernstein elements by performing partial mass lumping which sets each off-diagonal entry $\tilde m_{ij}^e$ associated with a diagonal neighbor $j\in\tilde{\mathcal N}_i\backslash\{i\}$ of node $i$ to zero and adds it to the diagonal element $\tilde m_{ii}^e$.
As a consequence of this modification, the corresponding fluxes $\tilde m_{ij}^e (v_i^e -v_j^e)$ become equal to zero and definition \eqref{eq:fij} produces $f_{ij}^e=0$ whenever $d_{ij}^e=0$.

\begin{remark}
The sparse element matrix $\tilde M_L^e - \tilde M_C^e$ is a Poisson-type operator in the sense that it is symmetric positive semi-definite with zero row sums and negative off-diagonal entries.
The solution $v^e$ of the linear system involving this matrix is defined up to a constant which has no influence on the flux $f_{ij}^e$ defined by \eqref{eq:fij}.
We fix the value of this arbitrary constant by setting each entry in the last row of $\tilde M_L^e - \tilde M_C^e$ to one.
To avoid the relatively high cost of inverting the resulting non-singular matrix repeatedly, we define it on the reference element and store its inverse.
\end{remark}

In the process of limiting, we multiply $f_{ij}^e$ and $f_{ji}^e$ with correction factors $\alpha_{ij}^e = \alpha_{ji}^e \in [0,1]$, which retains the conservation property established in \cref{lem:f-sum}.
The limited fluxes $f_{ij}^{e,*}=\alpha_{ij}^ef_{ij}^e$ transform
the bar states $\bar u_{ij}^e \in\mathcal G$ of the low order scheme
\eqref{eq:barstateform}
into
\begin{align}\label{eq:limbarstate}
\bar u_{ij}^{e,*} = \bar u_{ij}^e + \frac{f_{ij}^{e,*}}{2 d_{ij}} ~\in \mathcal G_i^e \cap \mathcal G,
\end{align}
where $\mathcal G_i^e$ is a set of admissible states defined by the solution $U_h$ at the beginning of the current Runge-Kutta stage or iteration.
Let
\begin{align*}
\mathcal G_i^e = \lc u_i^e \in \R^m:~
u_i^{e,\min} \cdot e_l \leq u_i^e \cdot e_l \leq u_i^{e,\max} \cdot e_l,
\quad l=1,\ldots,m
\rc
\end{align*}
be defined using the unit vector $e_l$ of $\R^m$ to pick out the inequality constraints for the  $l$-th conserved variable.
We select numerically admissible bounds $u_i^{e,\min},u_i^{e,\max}$
for all steps of the limiting process in \cref{sec:bounds} after completing the general presentation of our method.
For now, let us assume that generic bounds $u_i^{e,\min}, u_i^{e,\max}$ to be imposed on $u_i^e$ are available.
Following Kuzmin \cite{kuzmin2020}, we define the `monolithically' limited fluxes
\begin{align}\label{eq:limelflux}
f_{ij}^{e,*} = \begin{cases}
\min\lc
f_{ij}^e, 2d_{ij}^e\min \lc u_i^{e,\max} - \bar u_{ij}^e, \bar u_{ji}^e - u_j^{e,\min} \rc \rc & \mbox{if } f_{ij}^e \geq 0, \\
\max\lc
f_{ij}^e, 2d_{ij}^e\max \lc u_i^{e,\min} - \bar u_{ij}^e, \bar u_{ji}^e - u_j^{e,\max} \rc \rc & \mbox{if } f_{ij}^e < 0
\end{cases}
\end{align}
directly instead of calculating a correction factor
$\alpha_{ij}^e$ and applying it to $f_{ij}^{e}$.

The semi-discrete bar state form utilizing the limited fluxes $f_{ij}^{e,*}$ is obtained by replacing $\bar u_{ij}^e$ in \eqref{eq:barstateform} with $\bar u_{ij}^{e,*}$.
This is equivalent to replacing $f_i^e$ in \eqref{eq:unlimited} by $f_i^{e,*} = \sum_{j \in \tilde{\mathcal N}_i\backslash\{i\}} f_{ij}^{e,*}$.
The latter form is better suited for practical implementation purposes.
As explained in \citep{kuzmin2020}, direct calculation of the products $2d_{ij}^e \bar u_{ij}^e$ and $2d_{ij}^e\bar u_{ij}^{e,*}$ makes it possible to avoid unnecessary division and multiplication by $d_{ij}^e$.
Hence, the limited fluxes $f_{ij}^{e,*}$ defined by \eqref{eq:limelflux}
and the contribution of each node pair to the (flux-corrected) bar state form can be calculated without using the formal definition of $\bar u_{ij}^e$ and $\bar u_{ij}^{e,*}$.

\subsection{Limiting strategy for interfacial fluxes}\label{sec:flux}

To preserve the conservation property $f_{i,k}^e = -f_{i^\prime,k^\prime}^{e^\prime}$ of the antidiffusive DG fluxes defined by \eqref{eq:anti-diff-flux}, we need to use the same (virtual) correction factor for $f_{i,k}^e$ and $f_{i^\prime,k^\prime}^{e^\prime}$.
This is achieved similarly to the flux limiter for $f_{ij}^{e}$.
Once again, we assume the availability of generic bounds satisfying 
\begin{align}\label{eq:cond-bounds}
u_i^{e,\min} = u_{i^\prime}^{e^\prime,\min}, \qquad u_i^{e,\max} = u_{i^\prime}^{e^\prime,\max}.
\end{align}
To enforce these local bounds in a way which guarantees that conservation is not inhibited, the limited interfacial fluxes are defined by 
\begin{align}\label{eq:limdgflux}
f_{i,k}^{e,*} = \begin{cases}
\min \lc f_{i,k}, 2 d_{i,k}^e \min \lc
u_i^{e,\max} - \bar u_{i,k}^e, \bar u_{i,k}^e - u_i^{e,\min} 
\rc\rc & \mbox{if } f_{i,k}^e \geq 0, \\
\max \lc f_{i,k}, 2 d_{i,k}^e \max \lc
u_i^{e,\min} - \bar u_{i,k}^e, \bar u_{i,k}^e - u_i^{e,\max} 
\rc\rc & \mbox{if } f_{i,k}^e < 0.
\end{cases}
\end{align}
Note that this formula resembles \eqref{eq:limelflux}.
However, we have simplified \eqref{eq:limdgflux} by exploiting \eqref{eq:fluxbarstateeq}.
The semi-discrete scheme including limited DG fluxes is obtained by
replacing $f_{i,k}^e$ in \eqref{eq:unlimited} with $f_{i,k}^{e,*}$.
Again, direct calculation of the products $2 d_{i,k}^e\bar u_{i,k}^e$ that
appear in the definition of $f_{i,k}^{e,*}$ is preferable because 
denominator-free implementations are less sensitive to rounding errors.
The addition of $f_{i,k}^{e,*}$ has the same effect as replacement of 
$\bar u_{i,k}^e$ with 
\begin{align}\label{eq:limfluxstate}
\bar u_{i,k}^{e,*} = \bar u_{i,k}^e + \frac{f_{i,k}^{e,*}}{2d_{i,k}}
\end{align}
in the bar state form \eqref{eq:barstateform} of the
semi-discrete scheme. We emphasize again
that such representations are useful for
analysis and verification purposes but the flux-corrected
version of \eqref{eq:unlimited} may be employed in practice.

\subsection{Sequential limiting for systems of conservation laws}\label{sec:seq}

The scalar case $m=1$ is mostly covered by the explanations in the previous sections.
Limiting for systems ($m>1$), however, is more involved due to the coupling of multiple variables.
Some conserved quantities can be expressed as products of a \textit{specific variable} (a ratio of two conserved quantities) and a scalar conserved variable which we call the \textit{main unknown} in what follows.
For instance, momentum is the product of velocity and density in gas dynamics, whereas the main unknown of shallow water models is the water height.
When it comes to limiting for such systems, it makes sense to impose bounds on the specific variables, rather than conserved products, i.e., limit velocity rather than momentum, in the above examples.
This task can be accomplished by evolving the main variable first and constraining the products to satisfy local maximum principles for the specific variables in subsequent steps.
Algorithms based on this approach are called \textit{sequential limiters} \cite{dobrev2017,hajduk2019}.

A sequential bar state limiter for piecewise (multi-)linear continuous Galerkin
discretizations of hyperbolic systems was introduced in \cite{kuzmin2020}.
In this work, we extend it to the DG setting and high order approximations.
We begin with the volumetric terms and use the generic notation $\vrho$
and $(\vrho\vphi)$ for the conserved quantities of system \eqref{eq:pde}.
The main variable associated with $(\vrho\vphi)$ is $\vrho$, whereas
$\vphi$ is the corresponding specific variable (amount of $\vrho\vphi$
per unit of $\vrho$). The bar states for such derived quantities can be
defined as \cite{kuzmin2020}
\begin{align*}
\bar \vphi_{ij}^e = \frac{\overline{(\vrho\vphi)}_{ij}^e + \overline{(\vrho\vphi)}_{ji}^e}{\bar\vrho_{ij}^e + \bar\vrho_{ji}^e} = \bar \vphi_{ji}^e.
\end{align*}
The previous section details the steps necessary to produce limited bar states $\bar\vrho_{ij}^{e,*} = \bar \vrho_{ij}^e + \frac{f_{\vrho,ij}^{e,*}}{2d_{ij}^e}$ for the scalar main variable $\vrho$ at the first step of the sequential approach.
To limit $(\vrho\vphi)$ in a manner consistent with the product rule of calculus, we split the antidiffusive fluxes $f_{\vrho\vphi,ij}^e$ of the unknown $(\vrho\vphi)$ into a flux that leaves $\bar \vphi_{ij}^e$ unchanged and a remainder $g_{\vrho\vphi;\,ij}^{e}$ which requires additional limiting.
This product rule splitting is defined by
\begin{align}\label{eq:seq-splitting}
f_{\vrho\vphi;\,ij}^{e,*} =\;& 2d_{ij}^e \l\bar\vrho_{ij}^{e,*} \bar \vphi_{ij}^e - \overline{(\vrho\vphi)}_{ij}^e \r + g_{\vrho\vphi;\,ij}^{e,*}, 
\\
g_{\vrho\vphi;\,ij}^{e} =\;& f_{\vrho\vphi;\,ij}^{e} - 2d_{ij}^e \l\bar\vrho_{ij}^{e,*} \bar \vphi_{ij}^e - \overline{(\vrho\vphi)}_{ij}^e \r.\notag
\end{align}
The setting $g_{\vrho\vphi;\,ij}^{e,*} = g_{\vrho\vphi;\,ij}^{e}$ in \eqref{eq:seq-splitting} recovers the unlimited version of the scheme, while the setting $g_{\vrho\vphi;\,ij}^{e,*}=0$ yields a low order approximation with the bound preserving bar states $\bar\vrho_{ij}^{e,*} \bar \vphi_{ij}^e$.
As explained in \cite{kuzmin2020}, the bar states 
$\overline{(\vrho\vphi)}_{ij}^{e,*}$ of the flux-corrected high order scheme
should be constrained to stay in the range $[\bar\vrho_{ij}^{e,*}
\vphi_i^{e,\min},\bar \vrho_{ij}^{e,*}\vphi_i^{e,\max}]$, where
$\vphi_i^{e,\min},\vphi_i^{e,\max}$ are suitably defined local
bounds of numerical admissibility conditions. 
It follows
that the limited fluxes $g_{\vrho\vphi;\,ij}^{e,*}$ should satisfy
the inequality constraints 
\begin{align*}
g_{\vrho\vphi;\,ij}^{e,\min} \coloneqq 2d_{ij}^e 
\bar\vrho_{ij}^{e,*} \l \vphi_i^{e,\min} - \bar\vphi_{ij}^e \r \leq 
g_{\vrho\vphi;\,ij}^{e,*} \leq 2d_{ij}^e
\bar\vrho_{ij}^{e,*} \l \vphi_i^{e,\max} - \bar\vphi_{ij}^e \r \eqqcolon
g_{\vrho\vphi;\,ij}^{e,\max}
\end{align*}
and the equality constraint $g_{\vrho\vphi;\,ij}^{e,*} + g_{\vrho\vphi;\,ji}^{e,*}=0$
which ensures conservation.
These requirements can be met, similarly to \eqref{eq:limelflux},\eqref{eq:limdgflux}, by setting
\begin{align}\label{eq:seqlimel}
g_{\vrho\vphi;\,ij}^{e,*} = 
\begin{cases}
\min \lc g_{\vrho\vphi;\,ij}^{e}, \min \lc g_{\vrho\vphi;\,ij}^{e,\max}, 
-g_{\vrho\vphi;\,ji}^{e,\min}\rc \rc & 
\mbox{if } g_{\vrho\vphi;\,ij}^{e} \geq 0, \\
\max \lc g_{\vrho\vphi;\,ij}^{e}, \max \lc g_{\vrho\vphi;\,ij}^{e,\min},
-g_{\vrho\vphi;\,ji}^{e,\max} \rc \rc &
\mbox{if } g_{\vrho\vphi;\,ij}^{e} < 0.
\end{cases}
\end{align}
The fluxes $f_{\vrho\vphi;\,ij}^{e,*}$ are obtained by substituting
\eqref{eq:seqlimel} into \eqref{eq:seq-splitting} and the flux-corrected scheme is defined by replacing the specific components of $f_i^e$ in \eqref{eq:unlimited} with the sums $f_{\vrho\vphi;\,i}^{e,*} = \sum_{j \in \tilde{\mathcal N}_i\backslash\{i\}} f_{\vrho\vphi;\,ij}^{e,*}$.

The sequential limiter for interfacial DG fluxes operates in virtually the same fashion as its volumetric counterpart.
After computing the limited bar states $\bar\vrho_{i,k}^{e,*} = \bar \vrho_{i,k}^{e} + \frac{f_{\vrho;i,k}^e}{2d_{i,k}}$ for $\vrho$ and the bar states 
\begin{align}\label{eq:derived}
\bar \vphi_{i,k}^e = \bar \vphi_{i^\prime,k^\prime}^{e^\prime} =
\frac{\overline{(\vrho\vphi)}_{i,k}^e + \overline{(\vrho\vphi)}_{i^\prime,k^\prime}^{e^\prime}}{\bar\vrho_{i,k}^e + \bar\vrho_{i^\prime,k^\prime}^{e^\prime}} = \frac{\overline{(\vrho\vphi)}_{i,k}^e}{\bar\vrho_{i,k}^e}
\end{align}
 for $\vphi$, 
the limited antidiffusive fluxes are calculated using the splitting
\begin{align}\label{eq:seqantidiffflux}
f_{\vrho\vphi;\,i,k}^{e,*} =\;& 2d_{i,k}^e \l\bar\vrho_{i,k}^{e,*} \bar \vphi_{i,k}^e - \overline{(\vrho\vphi)}_{i,k}^e \r + g_{\vrho\vphi;\,i,k}^{e,*}, \\
g_{\vrho\vphi;\,i,k}^{e} =\;& f_{\vrho\vphi;\,i,k}^{e} - 2d_{i,k}^e \l\bar\vrho_{i,k}^{e,*} \bar \vphi_{i,k}^e - \overline{(\vrho\vphi)}_{i,k}^e \r.\notag
\end{align}
Again, the flux-corrected counterpart $\bar\vphi_{i,k}^{e,*}$
of the interfacial bar state $\bar \vphi_{i,k}^e$
stays in the range determined by suitable generic bounds $\vphi_i^{e,\min},\vphi_i^{e,\max}$ if
\begin{align*}
g_{\vrho\vphi;\,i,k}^{e,\min} \coloneqq 2d_{i,k}^e 
\bar\vrho_{i,k}^{e,*} \l \vphi_i^{e,\min} - \bar\vphi_{i,k}^e \r \leq 
g_{\vrho\vphi;\,i,k}^{e,*} \leq 2d_{i,k}^e
\bar\vrho_{i,k}^{e,*} \l \vphi_i^{e,\max} - \bar\vphi_{i,k}^e \r \eqqcolon
g_{\vrho\vphi;\,i,k}^{e,\max}.
\end{align*}
We enforce these interfacial flux constraints by setting
\begin{align}\label{eq:seqlimflux}
g_{\vrho\vphi;\,i,k}^{e,*} = 
\begin{cases}
\min \lc g_{\vrho\vphi;\,i,k}^{e}, \min \lc g_{\vrho\vphi;\,i,k}^{e,\max}, 
-g_{\vrho\vphi;\,i,k}^{e,\min}\rc \rc & 
\mbox{if } g_{\vrho\vphi;\,i,k}^{e} \geq 0, \\
\max \lc g_{\vrho\vphi;\,i,k}^{e}, \max \lc g_{\vrho\vphi;\,i,k}^{e,\min},
-g_{\vrho\vphi;\,i,k}^{e,\max} \rc \rc &
\mbox{if } g_{\vrho\vphi;\,i,k}^{e} < 0
\end{cases}
\end{align}
and replace the components of $f_{i,k}^e$ in \eqref{eq:unlimited} by
the limited fluxes \eqref{eq:seqantidiffflux}.
Similarly to other limiters considered in this work, the sequential
limiting of interfacial fluxes supports the possibility of
a denominator-free implementation.

\subsection{Definition of admissible bounds}\label{sec:bounds}

So far, we have left unspecified how to choose the bounds for limiting in \eqref{eq:limelflux},\eqref{eq:limdgflux},\eqref{eq:seqlimel}, and \eqref{eq:seqlimflux}.
Let us begin with the scalar case.
The mean value theorem implies that the bar states of a scalar variable $u$ satisfy
\begin{align*}
\min\lc u_i^e, u_j^e \rc \leq \bar u_{ij}^e \leq \max\lc u_i^e, u_j^e \rc, \quad
\min\lc u_i^e, u_{i^\prime}^{e^\prime} \rc \leq \bar u_{i,k}^e \leq \max\lc u_i^e, u_{i^\prime}^{e^\prime} \rc.
\end{align*}
The presented limiting strategy guarantees that the flux-corrected bar
states $\bar u_{ij}^{e,*}$ and $\bar u_{i,k}^{e,*}$ stay in a range
$[u_i^{e,\min},u_i^{e,\max}]$ of admissible values
which must contain $\bar u_{ij}^e$ for all $j \in \tilde{\mathcal N}_i\setminus \{i\}$ and $\bar u_{i,k}^e$ for all $k \in \mathcal F_i$ by definition.
To ensure that the same local bounds $u_i^{e,\min}=u_I^{\min}$ and
$u_i^{e,\max}=u_I^{\max}$ can be used for the Bernstein coefficients of all
nodes $\vec x_i^e$ having the same physical location $\vec x_I$, we define
\begin{align}\label{eq:bounds}
u_I^{\min} \coloneqq \min_{(e,i):\,\vec x_i^e = \vec x_I}\min_{j \in \tilde{\mathcal N}_i} u_j^e,\qquad
u_I^{\max} \coloneqq \max_{(e,i):\,\vec x_i^e = \vec x_I}\max_{j \in \tilde{\mathcal N}_i} u_j^e.
\end{align}
Note that \eqref{eq:bounds} satisfies condition \eqref{eq:cond-bounds}.
In contrast to the DG approximation $U_h$, the Bernstein finite element
interpolants $U_h^{\min}$ and $U_h^{\max}$ of the bounds defined by
\eqref{eq:bounds} are globally continuous.
In our experience \cite{hajduk2020a,hajduk2020b}, the use of
continuous bounding functions in limiters for DG schemes leads
to more reliable algorithms than discontinuous alternatives (as
considered, for instance, in \cite{anderson2017}).
The compact stencil bounds defined \eqref{eq:bounds} are generally
more restrictive than the full stencil bounds that we
used in \cite{hajduk2020a,hajduk2020b}. While the latter choice
was found to produce good results in the context of limited DG
schemes for the linear advection equation, its use for nonlinear
problems may give rise to spurious oscillations within high order elements. 
To avoid such ripples, we construct the nodal bounds \eqref{eq:bounds} by
taking maxima/minima over the union of subcell stencils $\tilde{\mathcal N}_i$  
rather than full element stencils (cf. also \cite{kuzmin2020a}).

For a scalar conservation law, the bar states $\bar u_{ij}^e$ and
$\bar u_{i,k}^e$ are bounded by pairs of Bernstein coefficients
belonging to $\mathcal G_I=[u_I^{\min},u_I^{\max}]$, where $u_I^{\min}$
and $u_I^{\max}$ are defined by \eqref{eq:bounds}. The main variable
$u=\vrho$ of a hyperbolic system can be limited as a scalar
quantity but the mean value theorem is no longer applicable
and, therefore, the bar states are not provably bounded by
$u_I^{\min}$ and $u_I^{\max}$. To rectify this, we include them in
the definition 
\begin{subequations}\label{eq:bounds-sys}
\begin{align}
u_i^{e,\min} = \min\lc u_I^{\min}, \bar u_{ij}^e \rc, \quad
u_i^{e,\max} = \max\lc u_I^{\max}, \bar u_{ij}^e \rc\qquad &\mbox{in \eqref{eq:limelflux}} \\
u_i^{e,\min} = \min\lc u_I^{\min}, \bar u_{i,k}^e \rc, \quad
u_i^{e,\max} = \max\lc u_I^{\max}, \bar u_{i,k}^e \rc \qquad &\mbox{in \eqref{eq:limdgflux}}
\end{align}
\end{subequations}
of the local bounds for the volumetric and interfacial terms, respectively.
These bounds can be seen as a generalization of \eqref{eq:bounds} to the system case.

An appropriate choice of bounds for specific variables $\vphi$ of conserved products $(\vrho\vphi)$ appears to be crucial to successful prevention of spurious ripples in sequential limiters for hyperbolic systems.
We have observed violations of maximum principles if these bounds were chosen too wide.
For volumetric terms, we obtained the best results using the maxima and minima of all states contributing to $\vphi_i^e$ to define the nodal bounds (cf. \cite{kuzmin2020})
\begin{subequations}\label{eq:seqbounds}
\begin{align}
&\vphi_I^{\min} = \min_{(e,i):\,\vec x_i^e = \vec x_I}
\min\lc \frac{(\vrho\vphi)_i^e}{\vrho_i^e}, \min\lc \min_{j \in \tilde{\mathcal N}_i} \bar \vphi_{ij}^e,~
\min_{k \in \mathcal F_i} \bar \vphi_{i,k}^e \rc \rc, \\
&\vphi_i^{\max} = \max_{(e,i):\,\vec x_i^e = \vec x_I}
\max\lc \frac{(\vrho\vphi)_i^e}{\vrho_i^e}, \max\lc \max_{j \in \tilde{\mathcal N}_i} \bar \vphi_{ij}^e,~
\max_{k \in \mathcal F_i} \bar \vphi_{i,k}^e \rc \rc.
\end{align}
\end{subequations}
In contrast to \eqref{eq:bounds}, we are not using any nodal states other than $\frac{(\vrho\vphi)_i^e}{\vrho_i^e}$ here but all low order bar states associated with the given node are included.
The corresponding bounding functions $\vphi_h^{\min}$ and $\vphi_h^{\max}$ are continuous again.

The different nature of volumetric terms and DG fluxes may be taken 
into account when it comes to constructing the property-preserving
bounds $\vphi_k^{e,\min}$ and $\vphi_k^{e,\max}$ for the
interfacial flux limiter \eqref{eq:seqlimflux}.
Sampling the auxiliary bar states \eqref{eq:derived} of nodes $\vec x_i^e$ belonging to the face component $\Gamma_k^e$, we define
\begin{align}\label{eq:fluxbounds}
\vphi_k^{e,\min} = \min_{\vec x_i^e \in \Gamma_k^e} \bar \vphi_{i,k}^e, \qquad \vphi_k^{e,\max} = \max_{\vec x_i^e \in \Gamma_k^e} \bar \vphi_{i,k}^e.
\end{align}
Although this definition is inconsistent with \eqref{eq:seqbounds}, it is motivated by the fact that the high-order flux $\mathcal H \l U_h^e,\hat U_h^e;\vec n^e \r$ across $\Gamma_k^e$ depends only on the Bernstein coefficients of nodes that are physically located on $\Gamma_k^e$.
To preserve this domain of dependence in the process of flux correction, we construct the bounds \eqref{eq:fluxbounds} of the limiting formula \eqref{eq:seqlimflux} individually for each face.

\subsection{Summary of the limiting procedure}
We end the description of our method with a brief summary of the limiting procedure.
The semi-discrete bound preserving formulation reads
\begin{align}\label{eq:limited}
m_i^e\frac{\d u_i^e}{\d t} = \;& \sum_{j\in \tilde{\mathcal N}_i} \lb d_{ij}^e (u_j^e - u_i^e) - \l \mat F_j^e - \mat F_i^e \r \tilde{\vec c}_{ij}^e \rb +
\sum_{j\in \tilde{\mathcal N}_i\backslash\{i\}}f_{ij}^{e,*}
\notag \\& + \sum_{k \in \mathcal F_i} \l \frac 1 2\int_{\Gamma_k^e} \phi_i^e\ds \lb \big( \mat F_i^e - \hat{\mat F}_i^e \big)\, \vec n^e + \lambda_i^e\,(\hat u_i^e - u_i^e) \rb + f_{i,k}^{e,*} \r.
\end{align}
Note that by omitting $f_{ij}^{e,*},\,f_{i,k}^{e,*}$ in \eqref{eq:limited} altogether, we reproduce the low order scheme \eqref{eq:loworder}.
For the scalar case, $f_{ij}^{e,*}$ and $f_{i,k}^{e,*}$ are defined by \eqref{eq:limelflux} and \eqref{eq:limdgflux}, respectively.
The bounds for both limiters are defined by \eqref{eq:bounds}.

For the system case, the fluxes for main unknowns are also limited using \eqref{eq:limelflux} and \eqref{eq:limdgflux} with bounds \eqref{eq:bounds-sys}.
Products of the main variables and ratios of conserved quantities are constrained using the limited fluxes \eqref{eq:seq-splitting}, \eqref{eq:seqantidiffflux} based on \eqref{eq:seqlimel}, \eqref{eq:seqlimflux} with respective bounds given by \eqref{eq:seqbounds}, \eqref{eq:fluxbounds}.

\section{Numerical results}\label{sec:results}

Having completed the description of our methodology, we now present results for classical benchmarks of various conservation laws.
We begin with tests for scalar linear and nonlinear transport models.
Further numerical examples demonstrate the feasibility of applying our method to the Euler and shallow water equations.
The implementation of all algorithms presented in this article is based on the open source C++ library MFEM \cite{mfem,anderson2019}.

For piecewise polynomial approximations of degree $p\in \N_0$, we abbreviate the methods under investigation using
the suffix $\mathbb V_p \in \{\mathbb P_p,\,\mathbb Q_p\}$ to indicate whether simplicial or box element meshes are used.
In our numerical study, we compare the results for the following spatial discretizations:
\begin{itemize}
\item DG-$\mathbb V_p$: Standard DG scheme, as defined by \eqref{eq:semi-disc-dg},
\item LO-$\mathbb V_p$: Low order method, as  defined by \eqref{eq:loworder},
\item MCL-$\mathbb V_p$: Monolithic convex limiting, as defined by \eqref{eq:limited}.
\end{itemize}
In this section, we occasionally compare piecewise constant approximations to limited solutions.
In such instances, $p=0$ always refers to DG-$\mathbb V_0$ approximations (rather than LO or MCL) because they do not require limiting.
As in \cite{hajduk2020a,hajduk2020b}, we often compare results for approximations of increasingly high order on a sequence of successively coarsened meshes.
The number of mesh elements $E=E(p)$ for each simulation is fitted to the  number of nodes per element $N=N(p)$ in such a way that the total number of degrees of freedom per variable (\#DOF = $EN$) remains constant.
For instance, the finite element spaces $\mathbb Q_2$ and $\mathbb Q_5$ have the same dimensions if the mesh for $\mathbb Q_2$ corresponds to one level of uniform refinement of the mesh for $\mathbb Q_5$.

All numerical solutions are evolved or marched to a steady state using explicit SSP Runge-Kutta methods \cite{gottlieb2001}.
In all time-dependent benchmarks, we employ the third order explicit time integrator with three stages, whereas forward Euler time stepping is sufficient for problems with steady state solutions.

\subsection{Linear transport in 1D}

Advective transport of a density field $u:\Omega \times \R_+\rightarrow \R$ in a 1D domain $\Omega \subseteq \R$ with unit velocity is described by the conservation law
\begin{align*}
\frac{\partial u}{\partial t} + \frac{\partial u}{\partial x} = 0 \qquad \mbox{in } \Omega \times \R_+.
\end{align*}
The one-dimensional unit normal $\vec n$ has as single component $n_x\in\{-1,1\}$ and the numerical flux \eqref{eq:LaxFriedrichs} simplifies to the upwind-sided trace
\begin{align*}
\mathcal H(u,v,\vec n) = \frac 1 2 \l (u+v)n_x + |n_x| (u-v) \r = \begin{cases}
u & \mbox{if } n_x = 1, \\
v & \mbox{if } n_x = -1.
\end{cases}
\end{align*}

We test the ability of the proposed methods to capture discontinuous and smooth initial profiles without introducing spurious oscillations.
To this end, we equip $\Omega = (0,1)$ with periodic boundary conditions.
Hence the analytical solution at any time $t \in \N$ coincides with the initial condition $u_0$.
In this setting, the invariant set is $\mathcal G = \lb \min_{\bar \Omega} u_0,\,\max_{\bar \Omega} u_0 \rb$.
The initial profile 
\begin{align}\label{eq:advection}
u_0(x) = \begin{cases}
1 & \mbox{if } 0.2 \leq x \leq 0.4, \\
\exp(10)\exp\l\frac{1}{0.5-x}\r\exp\l \frac{1}{x-0.9}\r & \mbox{if } 0.5 < x < 0.9, \\
0 & \mbox{otherwise}
\end{cases}
\end{align}
features a discontinuous part and an infinitely differentiable part.
We begin by advecting this profile up to the final time $t=1$.

For these simulations, we use time step $\Delta t$\,=\,$10^{-3}$, \#DOF\,=\,192
and consider orders $p \in \{0,2,5,11,23\}.$
The results for DG, LO and MCL are displayed in \cref{fig:advection}.
As expected, the standard DG method produces spurious ripples close to the discontinuities.
The LO scheme yields very inaccurate approximations which are even slightly more diffusive than the DG-$\mathbb Q_0$ result employing the same \#DOF.
However, it is remarkable that the accuracy of LO solutions does not deteriorate significantly as $p$ is increased while the mesh is coarsened. 
This property of the method is due to the use of preconditioned gradient matrix $\tilde C_1^e$ in definition \eqref{eq:dMat} of the artificial viscosity coefficients. 

The accuracy of numerical solutions obtained with the subcell flux limiting approach is also very similar for different orders and, in contrast to the pure DG results, the presented MCL approximations satisfy maximum principles.
This behavior indicates optimality of the method for a given number of \#DOF (cf. \cite{hajduk2020a}).
The slight peak clipping effects at the top of the exponential hill can be avoided by using smoothness indicators \cite{hajduk2020b,lohmann2017}.

\begin{figure}[ht!]
\centering
\begin{subfigure}[b]{0.325\textwidth}
\caption{DG solutions.}
\fbox{\includegraphics[scale=0.43]{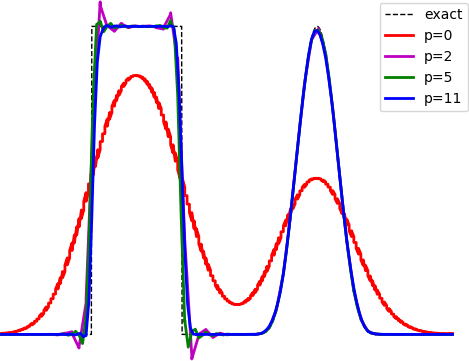}}
\end{subfigure}
\begin{subfigure}[b]{0.325\textwidth}
\caption{LO solutions.}
\fbox{\includegraphics[scale=0.43]{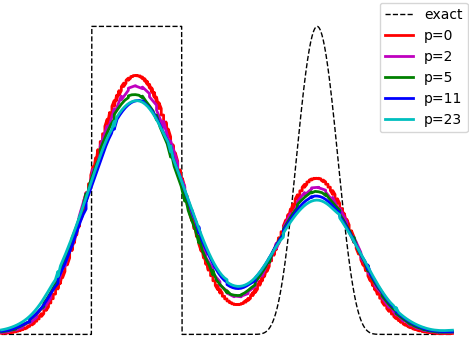}}
\end{subfigure}
\begin{subfigure}[b]{0.325\textwidth}
\caption{MCL solution.}
\fbox{\includegraphics[scale=0.43]{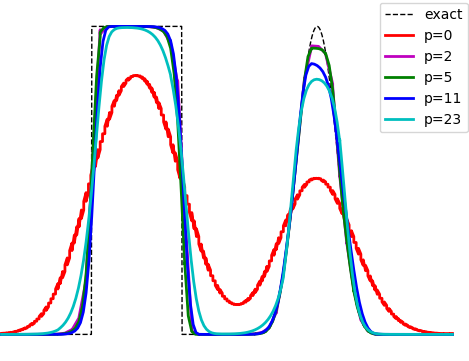}}
\end{subfigure}
\caption{Numerical approximations to the 1D advection equation with initial condition \eqref{eq:advection} obtained using $p \in \{0,2,5,11,23\}$.}\label{fig:advection}
\end{figure}


Next, we impose periodic boundary conditions at the endpoints of the domain $\Omega=(-1,1)$
and advect the smooth initial profile
\begin{align}\label{eq:advection-smooth}
u_0(x) = \exp\l-25x^2\r
\end{align}
up to the time instant $t=2$, at which we compute the $L^1(\Omega)$ errors and corresponding estimated orders of convergence (EOC) for DG, LO and MCL approximations using $p \in \{1,\dots,4\}$.
The time step size $\Delta t\,=\,10^{-4}$ is chosen to be small enough for the time discretization errors to be negligible compared to the spatial discretization error on the finest mesh.
The results presented in \cref{tab:dg-advection,tab:lo-advection,tab:mcl-advection} confirm optimal rates of convergence.
DG approximations converge with order $p+1$, the low order method converges at least with order $\frac 1 2$, and limited solutions are at least second order accurate.
Higher than second order convergence rates cannot be expected for limiters that preserve the local extrema of the old solution
\cite{godunov1959,zhang2010}.
Thus the local bounds of the inequality constraints for the antidiffusive fluxes need to be relaxed in smooth regions \cite{dumbser2014,guermond2018,hajduk2020b,pazner-preprint,vilar2019}.
For that purpose, we plan to equip our algorithm with suitably designed smoothness indicators in the future.

\begin{table}[ht!]
\centering
\begin{tabular}{||c||c|c||c|c||c|c||c|c||}
\hline
$1/h$ & $p=1$ & EOC & $p=2$ & EOC & $p=3$ & EOC & $p=4$ & EOC\\
\hline
24  & 1.27E-02 &      & 3.21E-04 &      & 7.38E-06 &      & 4.11E-07 & \\
32  & 6.43E-03 & 2.36 & 8.28E-05 & 4.71 & 2.13E-06 & 4.32 & 9.84E-08 & 4.97 \\
48  & 2.26E-03 & 2.58 & 1.53E-05 & 4.16 & 4.17E-07 & 4.02 & 1.27E-08 & 5.05 \\
64 & 1.01E-03 & 2.79 & 5.74E-06 & 3.41 & 1.32E-07 & 4.01 & 3.09E-09 & 4.90 \\
96 & 3.12E-04 & 2.91 & 1.62E-06 & 3.11 & 2.61E-08 & 3.99 &&\\
128 & 1.34E-04 & 2.92 & 6.86E-07 & 3.00 &&&&\\
192 & 4.17E-05 & 2.89 &&&&&&\\
\hline
\end{tabular}
\caption{The $\|\cdot\|_{L^1(\Omega)}$~errors and corresponding EOC of DG-$\mathbb Q_p$, $p \in \{1,\hdots,4\}$ solutions to the 1D advection equation with initial condition \eqref{eq:advection-smooth}.}\label{tab:dg-advection}
\end{table}

\begin{table}[ht!]
\centering
\begin{tabular}{||c||c|c||c|c||c|c||c|c||}
\hline
$1/h$ & $p=1$ & EOC & $p=2$ & EOC & $p=3$ & EOC & $p=4$ & EOC\\
\hline
24  & 9.43E-02 &      & 8.11E-02 &      & 6.73E-02 &      & 6.02E-02 & \\
32  & 7.93E-02 & 0.60 & 6.73E-02 & 0.65 & 5.51E-02 & 0.70 & 4.89E-02 & 0.72 \\
48  & 6.05E-02 & 0.67 & 5.05E-02 & 0.71 & 4.05E-02 & 0.76 & 3.56E-02 & 0.78 \\
64 & 4.92E-02 & 0.72 & 4.05E-02 & 0.77 & 3.21E-02 & 0.81 & 2.81E-02 & 0.83 \\
96 & 3.58E-02 & 0.78 & 2.91E-02 & 0.82 & 2.27E-02 & 0.85 &&\\
128 & 2.82E-02 & 0.83 & 2.27E-02 & 0.86 &&&&\\
192 & 1.98E-02 & 0.87 &&&&&&\\
\hline
\end{tabular}
\caption{The $\|\cdot\|_{L^1(\Omega)}$~errors and corresponding EOC of LO-$\mathbb Q_p$, $p \in \{1,\hdots,4\}$ solutions to the 1D advection equation with initial condition \eqref{eq:advection-smooth}.}\label{tab:lo-advection}
\end{table}

\begin{table}[ht!]
\centering
\begin{tabular}{||c||c|c||c|c||c|c||c|c||}
\hline
$1/h$ & $p=1$ & EOC & $p=2$ & EOC & $p=3$ & EOC & $p=4$ & EOC\\
\hline
24  & 1.04E-02 &      & 2.52E-03 &      & 1.27E-03 &      & 5.51E-04 & \\
32  & 5.69E-03 & 2.11 & 1.36E-03 & 2.14 & 6.60E-04 & 2.26 & 2.79E-04 & 2.37 \\
48  & 2.36E-03 & 2.17 & 5.46E-04 & 2.26 & 2.59E-04 & 2.31 & 1.07E-04 & 2.38 \\
64 & 1.27E-03 & 2.17 & 2.82E-04 & 2.30 & 1.32E-04 & 2.35 & 5.53E-05 & 2.28 \\
96 & 5.08E-04 & 2.25 & 1.08E-04 & 2.36 & 4.98E-05 & 2.40 &&\\
128 & 2.59E-04 & 2.34 & 5.58E-05 & 2.31 &&&&\\
192 & 1.01E-04 & 2.33 &&&&&&\\
\hline
\end{tabular}
\caption{The $\|\cdot\|_{L^1(\Omega)}$~errors and corresponding EOC of MCL-$\mathbb Q_p$, $p \in \{1,\hdots,4\}$ solutions to the 1D advection equation with initial condition \eqref{eq:advection-smooth}.}\label{tab:mcl-advection}
\end{table}

\subsection{Burgers equation}

In the next test of this section, we apply the methods under investigation to
a nonlinear conservation law for a
scalar quantity $u$ evolving in a domain $\Omega\subset\R^d$.
Using the constant vector $\vec v = (1,\dots,1)^T \in \R^d$
to isotropically extend the flux $f(u)=\frac{u^2}{2}$ into
multidimensions (for $d=2,3$), we consider the generalized
inviscid Burgers equation
\begin{align*}
\frac{\partial u}{\partial t} + \nabla \cdot \l \frac{\vec v u^2}{2}\r = 0 \qquad \mbox{in } \Omega \times \R_+.
\end{align*}
The flux function $\mat F(u) = \frac{\vec v u^2}{2}$ is convex and, therefore, the maximal wave speed \eqref{eq:wave-speed} is given by
$\lambda(u,v,\vec n^e) = \max \{u,v\}\left| \sum_{l=1}^d (\vec n^e)_l \right|$.

\subsubsection{One dimensional benchmark}

Once again, we consider the one dimensional case and equip $\Omega=(0,1)$ with periodic boundary conditions.
The initial data \cite{kuzmin-preprint2}
\begin{align}\label{eq:burgers1D}
u_0(x) = \sin(2\pi x)
\end{align}
remains smooth up to the time $t= \frac{1}{2\pi}$ of shock formation at $x=0.5$.
For smaller times, the classical solution can be determined using Newton's method to solve the implicit equation
$u(x,t) = \sin\l2\pi(x-u(x,t)t\r$.

We first run DG, LO and MCL up to $t=0.1$ using $\Delta t \,=\,4\cdot 10^{-4}$ and $p\in \{1,\dots,4\}$.
The $\|\cdot\|_{L^1(\Omega)}$ errors and convergence rates for the three methods are presented in \cref{tab:burgers-dg,tab:burgers-lo,tab:burgers-mcl}.
Our results confirm that the optimal convergence behavior, as observed in the previous section, carries over to the nonlinear case.
In fact, the low order method becomes first order accurate, which is an improvement compared to \cref{tab:lo-advection}.

\begin{table}[ht!]
\centering
\begin{tabular}{||c||c|c||c|c||c|c||c|c||}
\hline
$1/h$ & $p=1$ & EOC & $p=2$ & EOC & $p=3$ & EOC & $p=4$ & EOC\\
\hline
48  & 7.45E-04 &      & 1.60E-05 &      & 7.43E-07 &      & 4.96E-08 & \\
64  & 4.31E-04 & 1.90 & 7.23E-06 & 2.77 & 2.87E-07 & 3.30 & 1.14E-08 & 5.11 \\
96  & 1.98E-04 & 1.92 & 2.42E-06 & 2.70 & 6.69E-08 & 3.59 & 1.66E-09 & 4.76 \\
128 & 1.13E-04 & 1.94 & 1.09E-06 & 2.77 & 2.28E-08 & 3.73 & 4.59E-10 & 4.46 \\
192 & 5.15E-05 & 1.95 & 3.47E-07 & 2.82 & 4.89E-09 & 3.80 &&\\
256 & 2.93E-05 & 1.96 & 1.53E-07 & 2.86 &&&&\\
384 & 1.32E-05 & 1.97 &&&&&&\\
\hline
\end{tabular}
\caption{The $\|\cdot\|_{L^1(\Omega)}$~errors and corresponding EOC of DG-$\mathbb Q_p$, $p \in \{1,\hdots,4\}$ solutions to the 1D Burgers equation with initial condition \eqref{eq:burgers1D}.}
\label{tab:burgers-dg}
\end{table}

\begin{table}[ht!]
\centering
\begin{tabular}{||c||c|c||c|c||c|c||c|c||}
\hline
$1/h$ & $p=1$ & EOC & $p=2$ & EOC & $p=3$ & EOC & $p=4$ & EOC\\
\hline
48  & 1.62E-02 &      & 1.25E-02 &      & 9.10E-03 &      & 7.69E-03 & \\
64  & 1.23E-02 & 0.93 & 9.48E-03 & 0.96 & 6.95E-03 & 0.94 & 5.80E-03 & 0.98 \\
96  & 8.39E-03 & 0.95 & 6.41E-03 & 0.96 & 4.68E-03 & 0.98 & 3.90E-03 & 0.98 \\
128 & 6.37E-03 & 0.96 & 4.83E-03 & 0.98 & 3.52E-03 & 0.99 & 2.94E-03 & 0.99 \\
192 & 4.30E-03 & 0.97 & 3.25E-03 & 0.98 & 2.36E-03 & 0.99 &&\\
256 & 3.24E-03 & 0.98 & 2.45E-03 & 0.99 &&&&\\
384 & 2.17E-03 & 0.99 &&&&&&\\
\hline
\end{tabular}
\caption{The $\|\cdot\|_{L^1(\Omega)}$~errors and corresponding EOC of LO-$\mathbb Q_p$, $p \in \{1,\hdots,4\}$ solutions to the 1D Burgers equation with initial condition \eqref{eq:burgers1D}.}
\label{tab:burgers-lo}
\end{table}

\begin{table}[ht!]
\centering
\begin{tabular}{||c||c|c||c|c||c|c||c|c||}
\hline
$1/h$ & $p=1$ & EOC & $p=2$ & EOC & $p=3$ & EOC & $p=4$ & EOC\\
\hline
48  & 1.29E-03 &      & 2.03E-04 &      & 9.54E-05 &      & 4.87E-05 & \\
64  & 7.68E-04 & 1.80 & 9.98E-05 & 2.47 & 4.85E-05 & 2.36 & 2.46E-05 & 2.38 \\
96  & 3.44E-04 & 1.98 & 4.05E-05 & 2.22 & 1.95E-05 & 2.25 & 9.91E-06 & 2.24 \\
128 & 1.94E-04 & 1.98 & 2.24E-05 & 2.06 & 1.09E-05 & 2.02 & 5.07E-06 & 2.33 \\
192 & 8.41E-05 & 2.07 & 9.23E-06 & 2.19 & 4.26E-06 & 2.32 &&\\
256 & 4.69E-05 & 2.03 & 4.74E-06 & 2.32 &&&&\\
384 & 2.04E-05 & 2.05 &&&&&&\\
\hline
\end{tabular}
\caption{The $\|\cdot\|_{L^1(\Omega)}$~errors and corresponding EOC of MCL-$\mathbb Q_p$, $p \in \{1,\hdots,4\}$ solutions to the 1D Burgers equation with initial condition \eqref{eq:burgers1D}.}
\label{tab:burgers-mcl}
\end{table}

We proceed by computing approximations to this problem at final time $t=0.2$ after the shock has developed.
We use $\Delta t \,=\,2.5\cdot 10^{-3}$, $p \in \{0,1,3,7,15,31\}$ and keep \#DOF\,=\,96 constant for all simulations.
The results of LO and MCL approximations are shown in \cref{fig:burgers1D}.
Before the shock formation time, all approximations are almost indistinguishable, while at time $t=0.2$ only the highest order approximation is considerably more diffusive than other curves in diagrams depicting the LO and MCL results.
This behavior can be attributed to the fact that for $p=31$ the shock is located inside an element rather than at element boundaries, as for all other approximations.
Given the fact that MCL-$\mathbb Q_{31}$ uses just three elements, it is remarkable that the result still looks quite reasonable.
Around the common boundary of the second and third element from the left, the MCL-$\mathbb Q_{31}$ solution exhibits stronger deviations from the other curves but no violations of monotonicity are observed.
For this simple example, LO is almost as accurate as MCL with only minor peak clipping effects close to the shock, which is in accordance with the convergence rates presented in \cref{tab:burgers-lo}.

\begin{figure}[ht!]
\centering
\begin{subfigure}[b]{0.325\textwidth}
\caption{MCL solutions, $t=0.1$.}
\fbox{\includegraphics[scale=0.42]{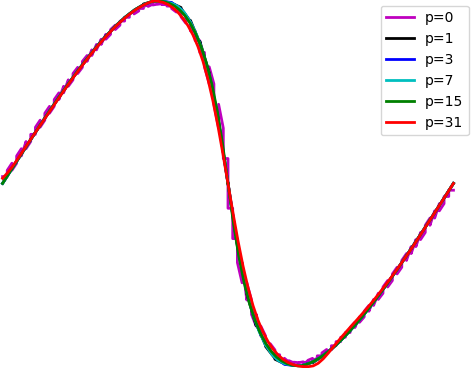}}
\end{subfigure}
\begin{subfigure}[b]{0.325\textwidth}
\caption{LO solutions, $t=0.2$.}
\fbox{\includegraphics[scale=0.42]{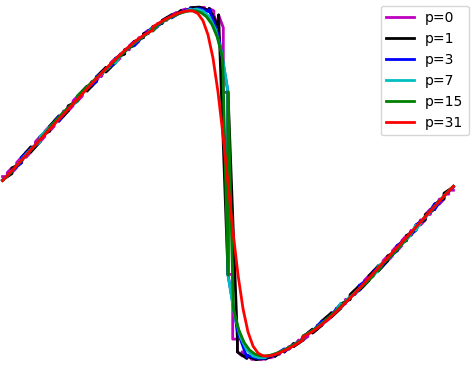}}
\end{subfigure}
\begin{subfigure}[b]{0.325\textwidth}
\caption{MCL solutions, $t=0.2$.}
\fbox{\includegraphics[scale=0.42]{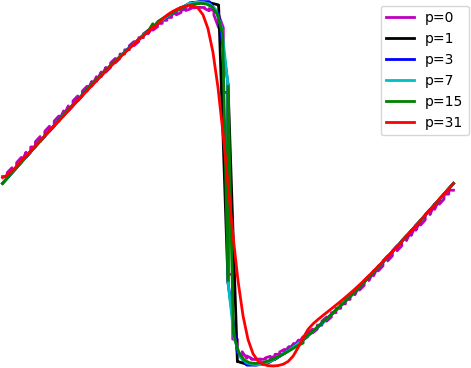}}
\end{subfigure}
\caption{Snapshots of bound preserving $p \in \{0,1,3,7,15,31\}$ approximations to the 1D Burgers equation with initial condition \eqref{eq:burgers1D} at time instants (a) before and (b),(c) after the shock development.}\label{fig:burgers1D}
\end{figure}

\subsubsection{Two dimensional benchmark}

For the 2D Burgers test, we use
$\Omega = (0,1)^2$ and the initial condition
\begin{align}\label{eq:burgers2D}
u_0(x,y) = \begin{cases}
-0.2 & \mbox{if } x < 0.5 \wedge y > 0.5, \\
  -1 & \mbox{if } x > 0.5 \wedge y > 0.5, \\
 0.5 & \mbox{if } x < 0.5 \wedge y < 0.5, \\
 0.8 & \mbox{if } x > 0.5 \wedge y < 0.5.
\end{cases}
\end{align}
The boundary type at a point $(x,y) \in \partial \Omega$ depends on the sign of $u$ which is changing in time.
At the time dependent inlet, we prescribe the analytical solution which can be found in \cite{guermond2014a}.
The invariant set of this benchmark is $\mathcal G = [-1, 0.8]$.
At the final time $t=0.5$, we compare DG-$\mathbb{Q}_p$ results for $p\in \{0,1\}$ to MCL-$\mathbb{Q}_p$ approximations using $p\in \{1,3,7,15\}$.
The snapshots presented in \cref{fig:burgers2D} were obtained using \#DOF\,=\,128$^2$ and time step $\Delta t\,=\,10^{-3}$ in all cases.

\begin{figure}[ht!]
\centering
\begin{subfigure}[b]{0.325\textwidth}
\caption{DG-$\mathbb{Q}_0,~e_1(0.5)\,=\,$1.49E-2\phantom{W}}
\includegraphics[scale=0.153]{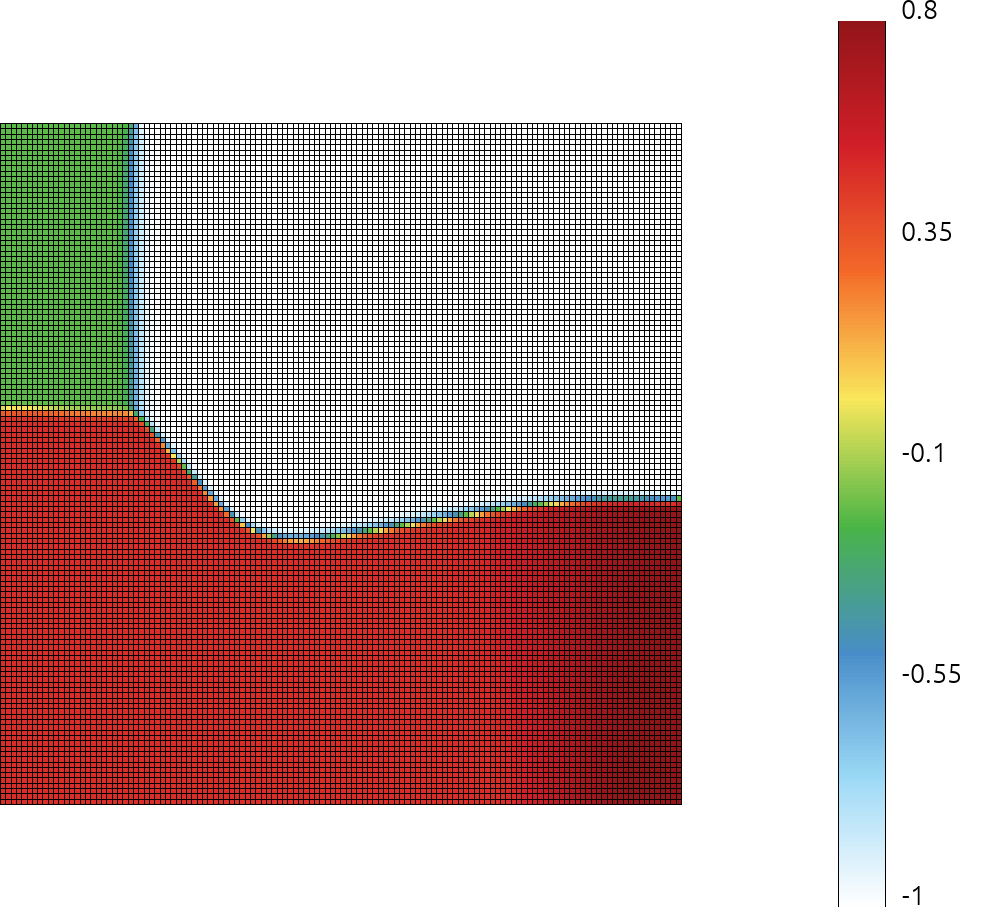}
\end{subfigure}
\begin{subfigure}[b]{0.325\textwidth}
\caption{DG-$\mathbb{Q}_1,~e_1(0.5)\,=\,$9.67E-3\phantom{W}}
\includegraphics[scale=0.153]{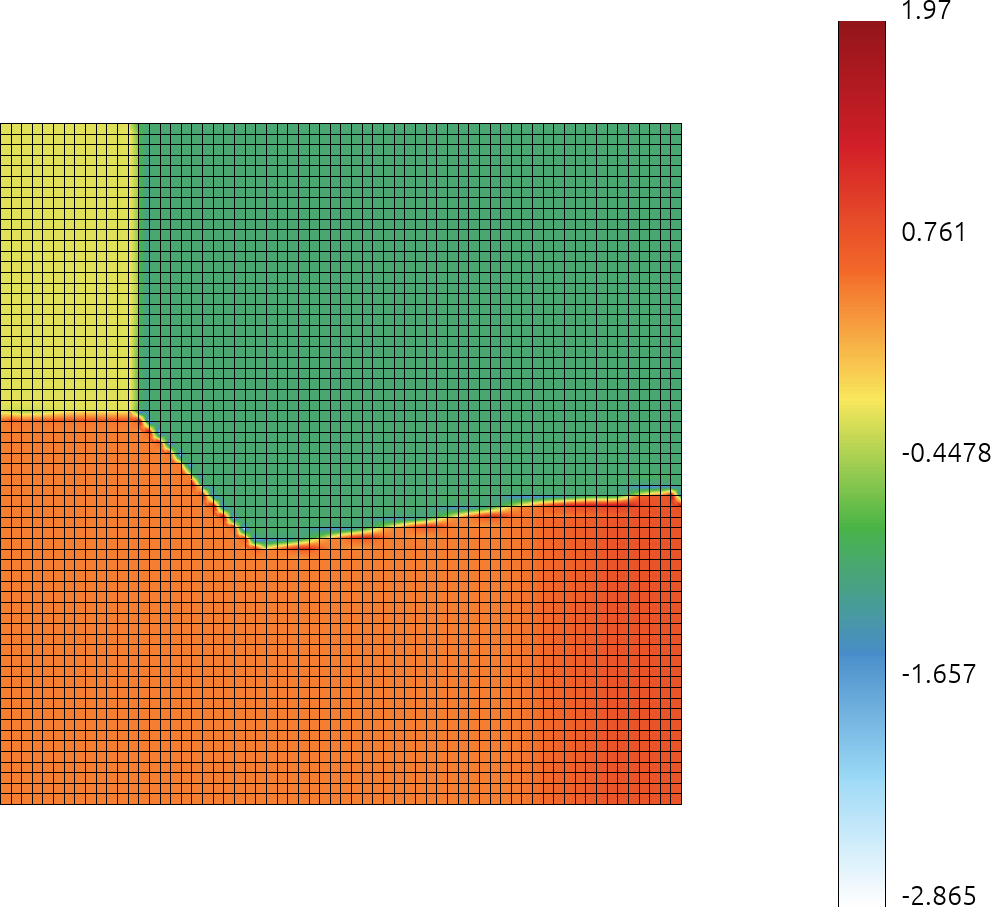}
\end{subfigure}
\begin{subfigure}[b]{0.325\textwidth}
\caption{MCL-$\mathbb{Q}_1,~e_1(0.5)\,=\,$1.09E-2\phantom{o}}
\includegraphics[scale=0.153]{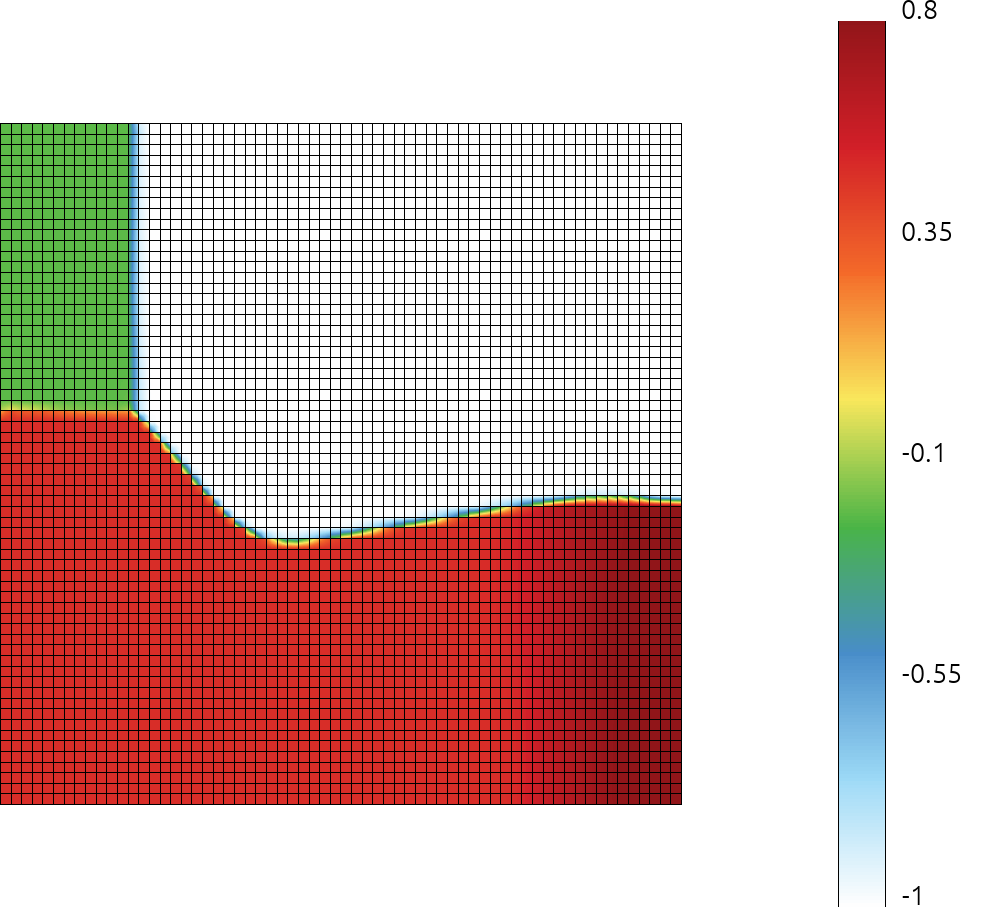}
\end{subfigure}
\vskip5mm
\begin{subfigure}[b]{0.325\textwidth}
\caption{MCL-$\mathbb{Q}_3,~e_1(0.5)\,=\,$1.28E-2\phantom{i}}
\includegraphics[scale=0.153]{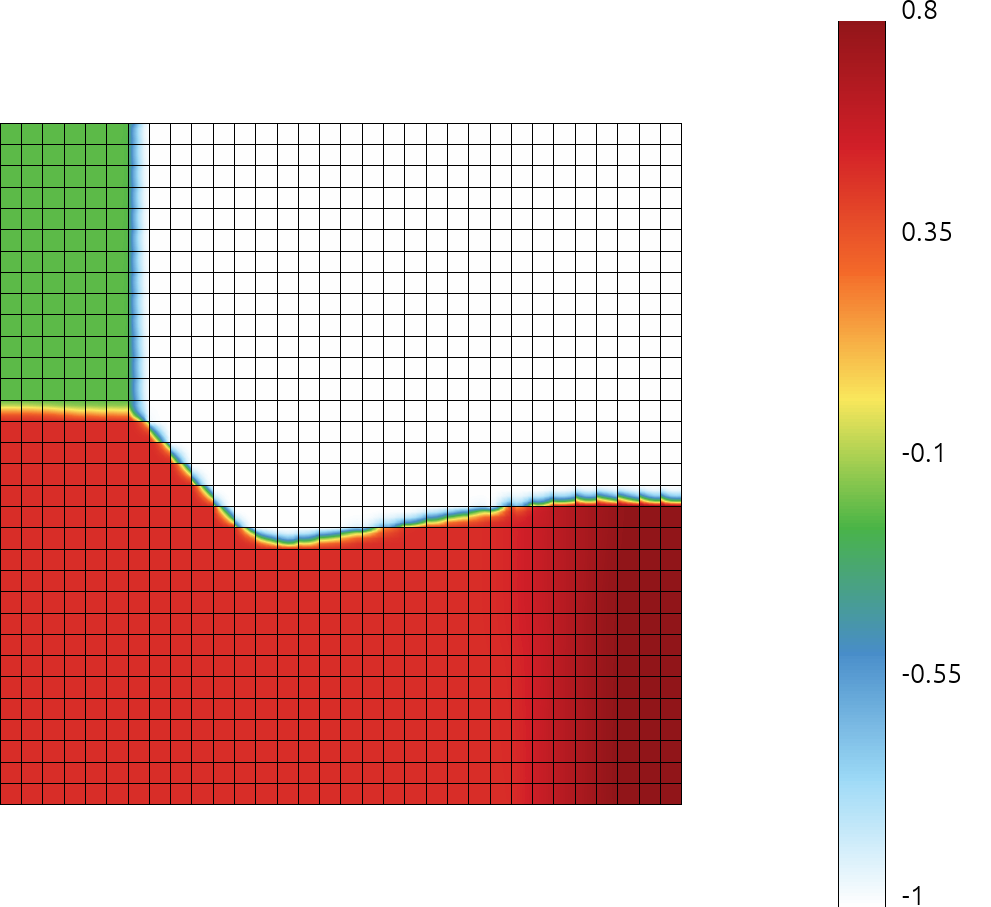}
\end{subfigure}
\begin{subfigure}[b]{0.325\textwidth}
\caption{MCL-$\mathbb{Q}_7,~e_1(0.5)\,=\,$1.73E-2\phantom{o}}
\includegraphics[scale=0.153]{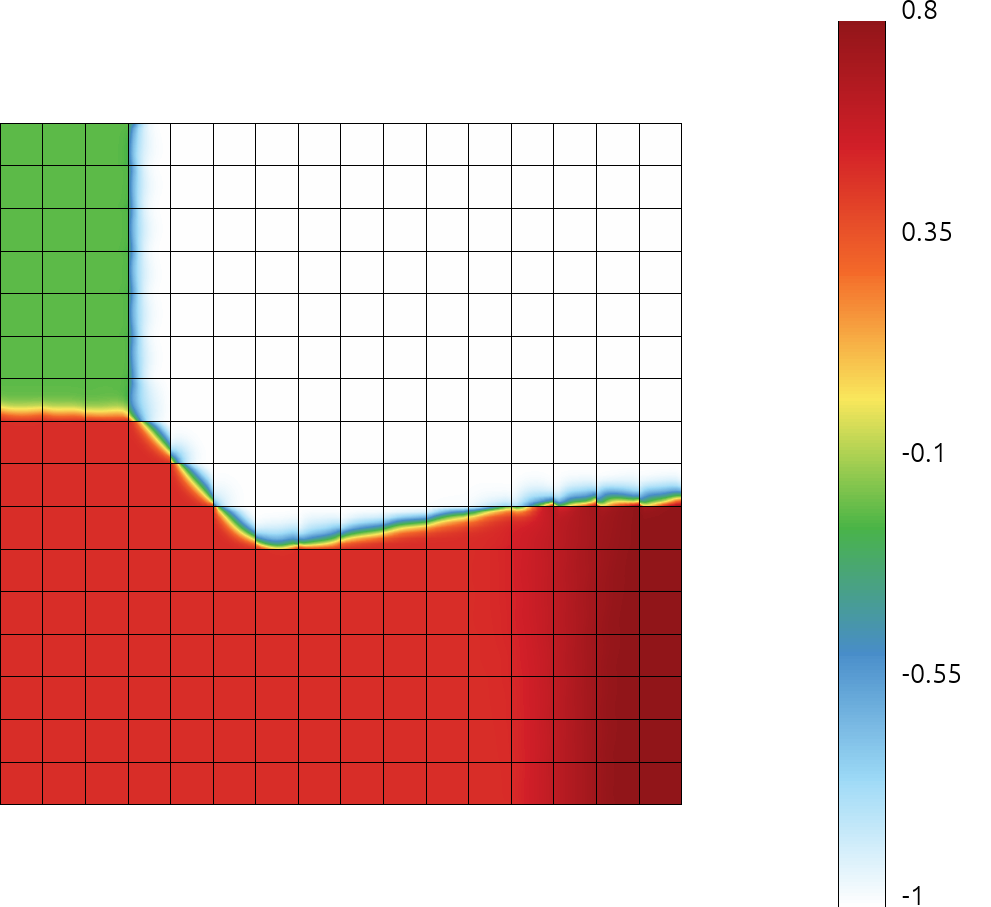}
\end{subfigure}
\begin{subfigure}[b]{0.325\textwidth}
\caption{MCL-$\mathbb{Q}_{15},~e_1(0.5)\,=\,$2.88E-2}
\includegraphics[scale=0.153]{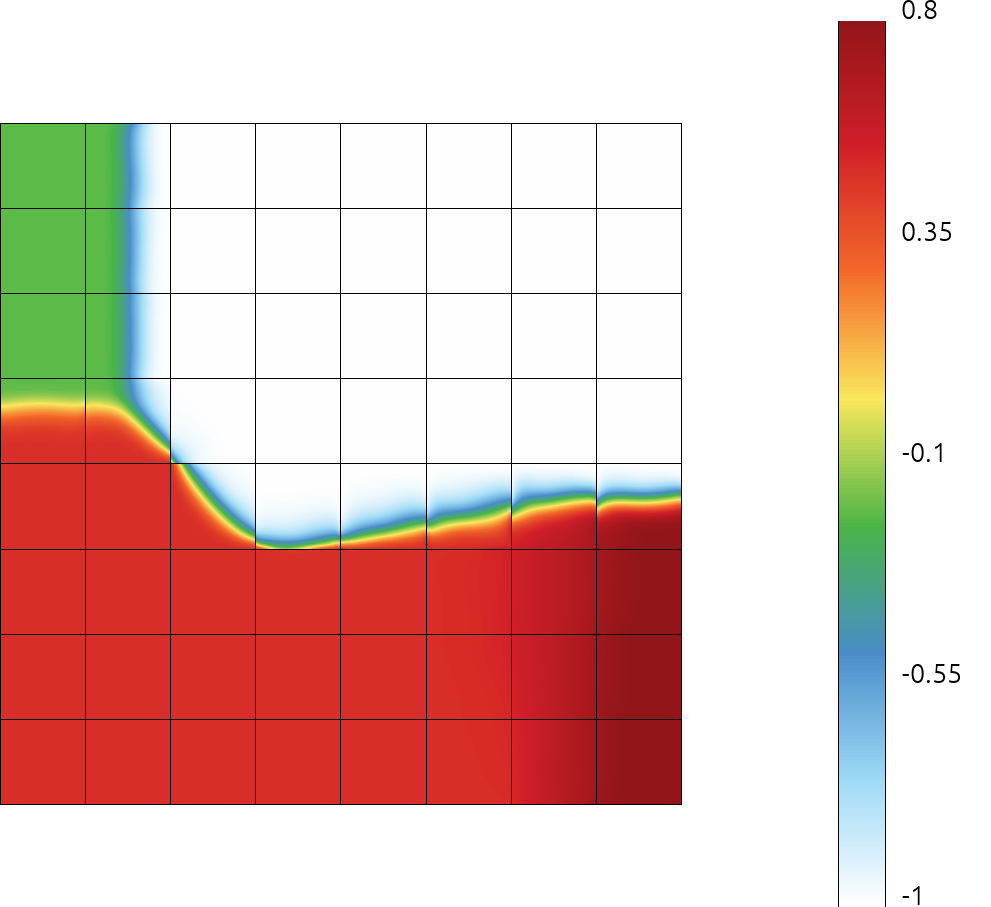}
\end{subfigure}
\caption{DG and MCL solutions of the 2D Burgers equation with initial condition \eqref{eq:burgers2D} and corresponding $e_1(t) \coloneqq \|u(t) - u_h(t) \|_{L^1(\Omega)}$~errors for
  $p \in \{0,1,3,7,15\}$ at $t=0.5$.}\label{fig:burgers2D}
\end{figure}

As expected, the DG-$\mathbb{Q}_0$ result is quite diffusive.
The DG-$\mathbb{Q}_1$ solution exhibits severe over- and undershoots close to the shocks, while all MCL approximations are nonoscillatory.
The $L^1$ error of the MCL-$\mathbb Q_1$ result is just marginally larger than that of the DG-$\mathbb Q_1$ solution.
A decline in the accuracy of MCL-$\mathbb Q_p$ is observed for higher orders on coarser meshes.
For this shock-dominated problem, the MCL-$\mathbb{Q}_1$ solution is optimal because it is monotonicity preserving and has the smallest error for a fixed \#DOF.
The MCL-$\mathbb{Q}_3$ solution is slightly more accurate than the DG-$\mathbb Q_0$ result but the $L^1$ error of the latter approximation is smaller than that of MCL-$\mathbb{Q}_p$ for $p\in \{7,15\}$.
It is still remarkable how well the MCL-$\mathbb{Q}_{15}$ approximation resolves the analytical solution, even though none of the shocks is aligned with mesh edges.

\subsection{Euler equations of gas dynamics}

An important example of \eqref{eq:pde} are the Euler equations of gas dynamics.
The solution vector and flux matrix of this hyperbolic system are given by
\begin{align*}
U = \begin{bmatrix}
\rho \\ \rho \vec v \\ \rho E
\end{bmatrix} \in \R^{d+2}, \qquad
\mat F(U) = \begin{bmatrix}
\rho \vec v \\ \rho \vec v \otimes \vec v + p\, \mat I \\ (\rho E + p) \vec v
\end{bmatrix} \in \R^{(d+2)\times d},
\end{align*}
where $\rho,\,\vec v,\,E$ denote density, velocity and specific total energy, respectively, $\mat I \in \R^{d\times d}$ is the identity matrix and $p$ is the pressure satisfying the equation of state for an ideal polytropic gas
\begin{align*}
p = (\gamma - 1)\l \rho E - \frac{\rho |\vec v|^2}{2} \r = (\gamma - 1) \rho e.
\end{align*}
Here $\rho e$ is the internal energy and $\gamma > 1$ denotes the heat capacity ratio, which is set to $1.4$ in all numerical examples considered in this work.

An upper bound for the fastest wave speed has been derived in \cite{guermond2016a}.
An invariant set of the Euler system is given by \cite{guermond2018}
\begin{align*}
\mathcal G = \lc(\rho,\rho\vec v,\rho E)^T \in \R^{d+2}:~\rho >0, \,e >0,\, s\geq s_{\min}\rc,
\end{align*}
where $s = \log\l e^{\frac{1}{\gamma-1}} \rho^{-1}\r$
is the specific entropy which should be bounded below by an arbitrary constant $s_{\min} >0$.


\subsubsection{Sod Shock tube}\label{sec:sod}

The classical shock tube example by Sod \cite{sod1978} serves as our first test for solving the Euler equations.
The 1D domain $\Omega = (0,1)$ is equipped with wall boundaries at $x\in \{0,1\}$ and the initial condition is specified as
\begin{align*}
U_0(x) = \begin{cases}
(1,0,2.5)^T & \mbox{if } x < 0.5, \\
(0.125,0,0.25)^T & \mbox{if } x > 0.5.
\end{cases}
\end{align*}
The discontinuity of the initial condition at $x=0.5$ gives rise to a shock, rarefaction wave, and contact discontinuity, which a numerical approximation should resolve reasonably well.

\begin{figure}[ht!]
\centering
\begin{subfigure}[b]{0.49\textwidth}
\caption{LO-$\mathbb Q_p$, density $\rho_h \in\, $[0.125, 1.0]
\phantom{WW}}
\fbox{\includegraphics[scale=0.6]{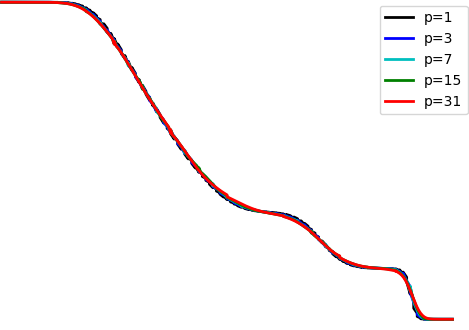}}
\end{subfigure}
\begin{subfigure}[b]{0.49\textwidth}
\caption{LO-$\mathbb Q_p$, pressure $p_h \in\, $[0.1, 1.0]\phantom{WW}}
\fbox{\includegraphics[scale=0.6]{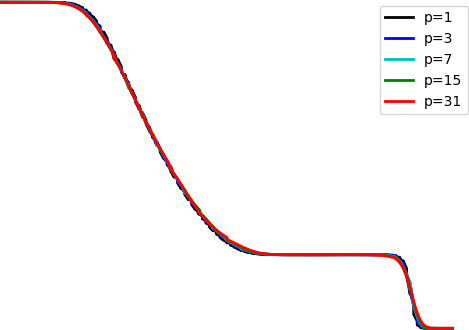}}
\end{subfigure}
\vskip5mm
\begin{subfigure}[b]{0.49\textwidth}
\caption{MCL-$\mathbb Q_p$, density $\rho_h \in\, $[0.125, 1.0]
\phantom{WW}}
\fbox{\includegraphics[scale=0.6]{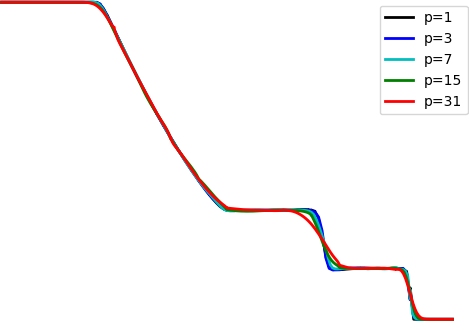}}
\end{subfigure}
\begin{subfigure}[b]{0.49\textwidth}
\caption{MCL-$\mathbb Q_p$, pressure $p_h \in\, $[0.1, 1.0]\phantom{WW}}
\fbox{\includegraphics[scale=0.6]{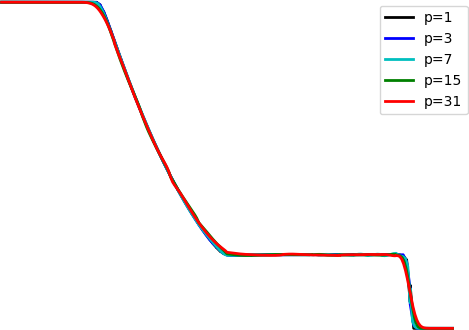}}
\end{subfigure}
\caption{Sod shock tube benchmark for the Euler equations. LO and MCL approximations at time $t=0.231$ obtained using $p \in \{1,3,7,15,31\}$.}\label{fig:sod}
\end{figure}

We evolve the LO and MCL solutions up to the final time $t=0.231$ using $\Delta t \,= \,4\cdot 10^{-4}$, \#DOF\,=\,256, and $p \in \{1,3,7,15,31\}$.
The density and pressure profiles obtained with these parameter settings are shown in \cref{fig:sod}.
They exhibit good qualitative agreement with each other and are all nonoscillatory.
Again, this is a Riemann problem with a discontinuous exact solution.
Therefore, it is quite remarkable that the solutions obtained with large $p$ on coarse meshes are almost as well resolved as the smaller-$p$-same\#DOF approximations.
Note that the MCL-$\mathbb Q_1$ solution is obtained on a mesh with 128 elements, while MCL-$\mathbb Q_{31}$ uses only 8 elements.
The monolithic limiter employed in this study represents the first high order extension of the sequential limiting procedure proposed in \cite{kuzmin2020} to the system case.
So far, only \cite{pazner-preprint} presents algebraically limited high order DG methods for systems, and the limiting strategy adopted therein is not monolithic.

A preliminary conclusion, which has been shown to be true for MCL in the scalar case, is as follows.
For shock-dominated problems and fixed \#DOF, the most accurate MCL results are obtained with low polynomial degrees on fine meshes.
This is not surprising since the potential benefit of using higher order approximations lies in the possibility of achieving improved convergence rates under certain smoothness assumptions.
If these assumptions do not hold, it makes more sense to refine the mesh than to increase the order.
For smooth solutions, the opposite may be the case.
However, to circumvent the second order accuracy barrier for local extremum diminishing approximations, the MCL scheme must be equipped with a smoothness indicator.

\subsubsection{Double Mach reflection}

A more difficult benchmark for solving the Euler equations is double Mach reflection \cite{woodward1984}.
The boundary of the computational domain $\Omega =(0,4)\times (0,1)$ 
consists of a reflecting wall \mbox{$\Gamma_w = \lc(x,0)^T \in \partial \Omega:~1/6 \leq x \leq 4\rc$}, a supersonic outlet $\Gamma_o = \lc (4,y)^T:~0 < y \leq 1\rc$, and a supersonic inlet at $\partial \Omega \setminus (\Gamma_w \cup \Gamma_o)$.
Post- and pre-shock states defined by
\begin{align*}
U_L = \begin{bmatrix}
8\\66\cos\l 30^\circ \r\\ -66 \sin\l 30^\circ \r \\ 563.5 \end{bmatrix}, 
\qquad&
U_R = \begin{bmatrix}
1.4 \\ 0 \\ 0 \\ 2.5
\end{bmatrix},
\end{align*}
are used as initial and inflow boundary conditions in the following manner
\begin{align*}
U_0(x,y) = \begin{cases}
U_L & \mbox{if } x < \frac 1 6 + \frac{y}{\sqrt 3}, \\
U_R & \mbox{otherwise},
\end{cases} \qquad&
U_{\mathrm{in}}(x,y,t) = \begin{cases}
U_L & \mbox{if } x < \frac 1 6 + \frac{y+20t}{\sqrt 3}, \\
U_R & \mbox{otherwise}.
\end{cases}
\end{align*}
The angle between the initial shock and wall boundary is $60^\circ$.
The shock-wall interaction results in a propagating Mach 10 shock.
To assess the performance of numerical schemes for this benchmark, one tries to resolve the triple point visible in the right part of $\Omega$ at time $t=0.2$ as accurately as possible.

We have already established that the most accurate approximations for fixed \#DOF are obtained with piecewise linear polynomials.
To study the impact of \#DOF on the accuracy of our scheme, we therefore compare the MCL-$\mathbb Q_1$ solutions for this benchmark on a series of refined uniform meshes.
On the coarsest level, we use \#DOF\,=\,$4\cdot96^2$, $\Delta t\,=\,5\cdot 10^{-5}$.
The mesh and time step sizes are halved at each refinement level.

\begin{figure}[ht!]
\centering
\begin{subfigure}[b]{\textwidth}
\caption{Post-shock density distribution on a mesh of $4\cdot 48^2$ elements obtained with $\Delta t\,=\,5$E-5.}
\includegraphics[scale=0.18]{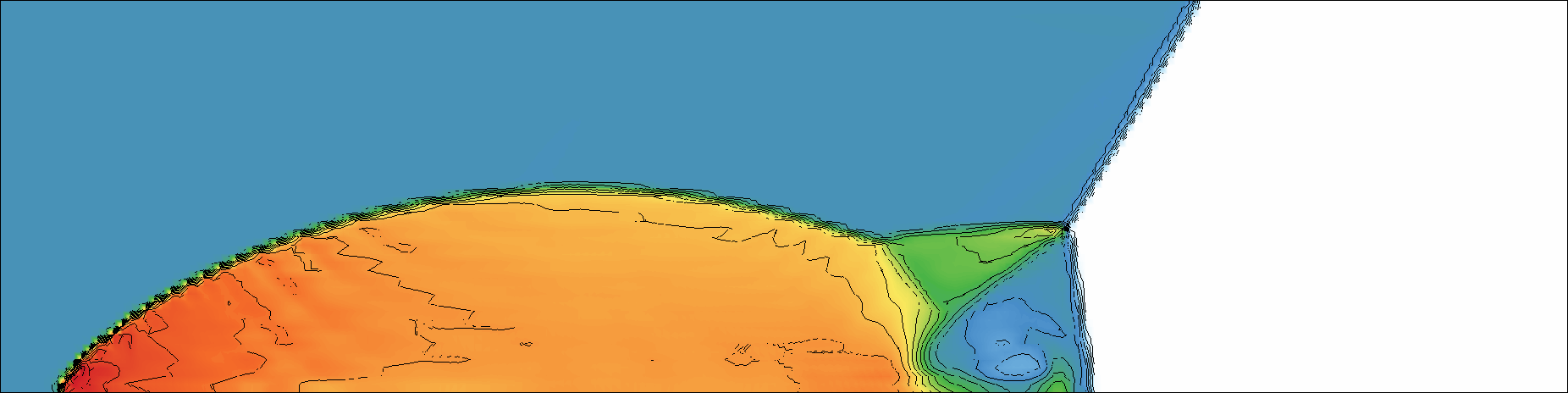}
\fbox{\includegraphics[scale=0.09]{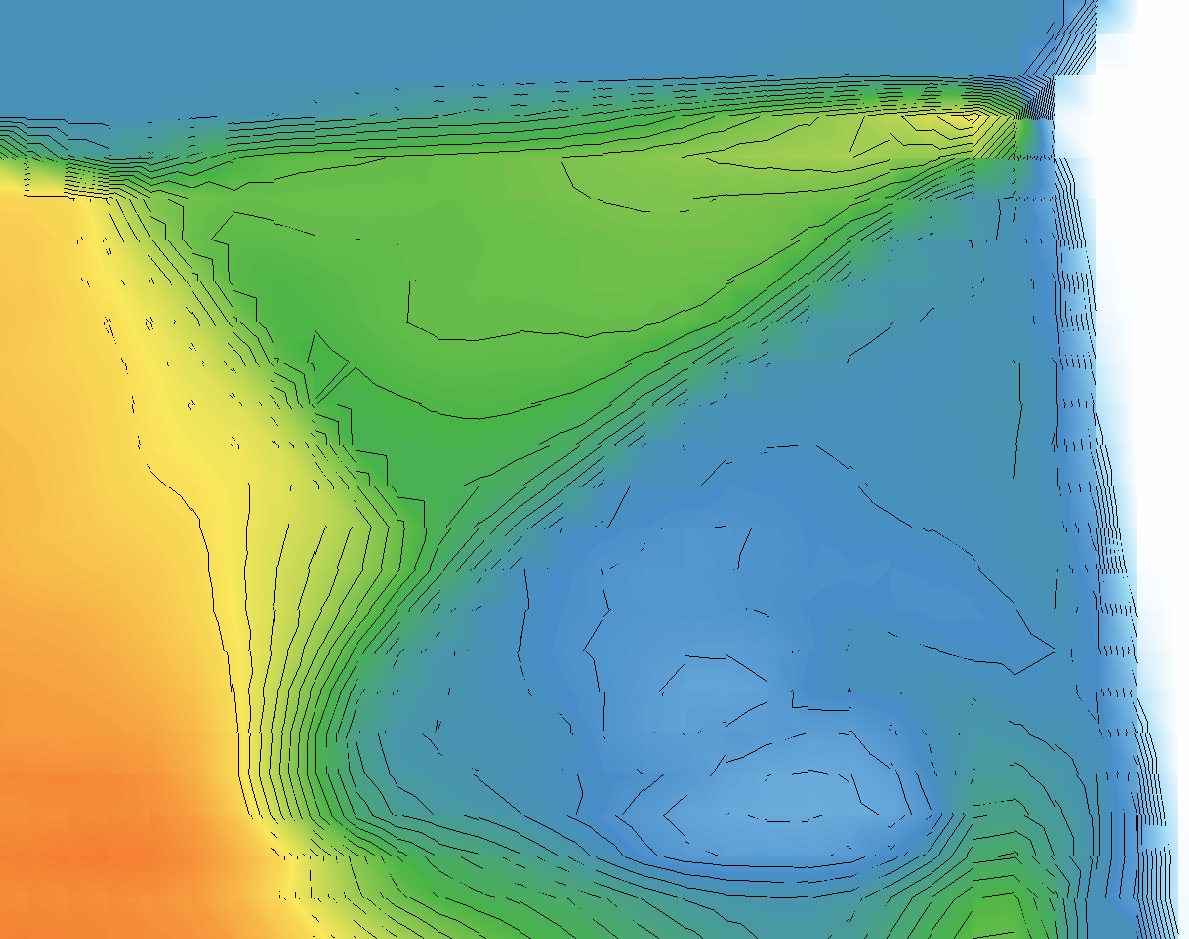}}
\end{subfigure}
\vskip2mm
\begin{subfigure}[b]{\textwidth}
\caption{Post-shock density distribution on a mesh of $4\cdot 96^2$ elements obtained with $\Delta t\,=\,2.5$E-5.}
\includegraphics[scale=0.18]{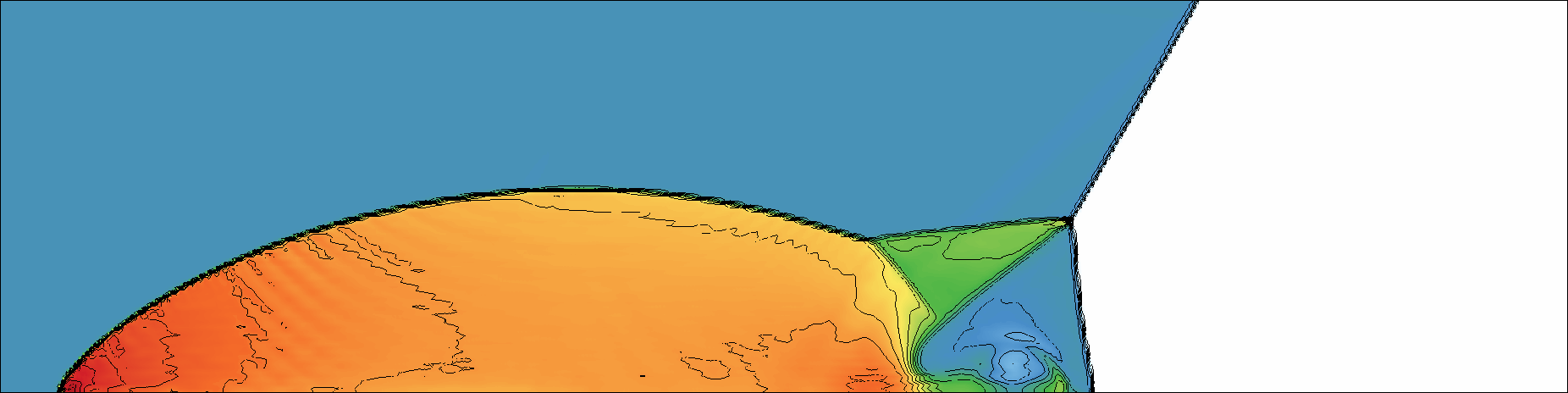}
\fbox{\includegraphics[scale=0.09]{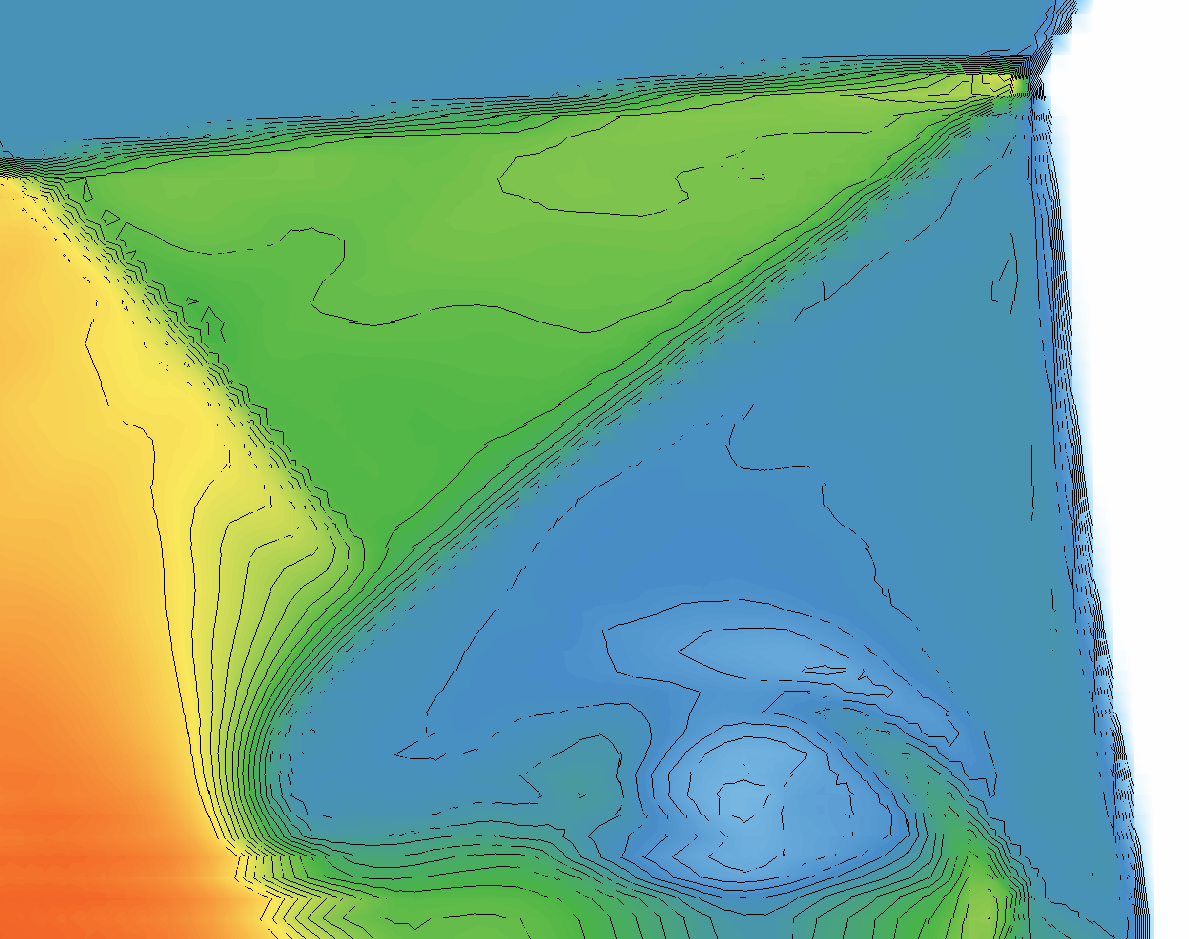}}
\end{subfigure}
\vskip2mm
\begin{subfigure}[b]{\textwidth}
\caption{Post-shock density distribution on a mesh of $4\cdot 192^2$ elements obtained with $\Delta t\,=\,1.25$E-5.}
\includegraphics[scale=0.18]{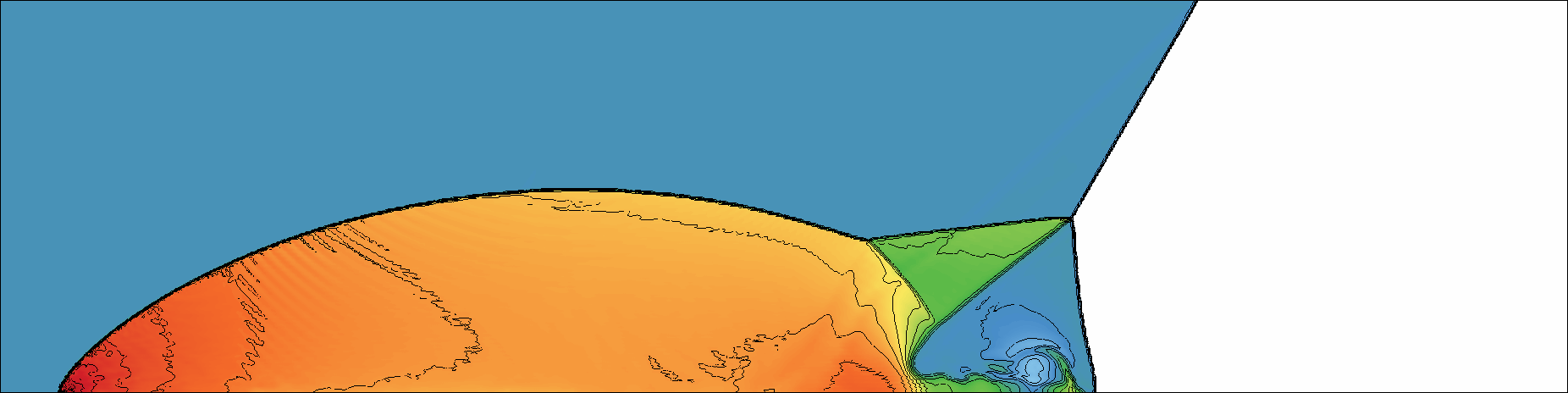}
\fbox{\includegraphics[scale=0.09]{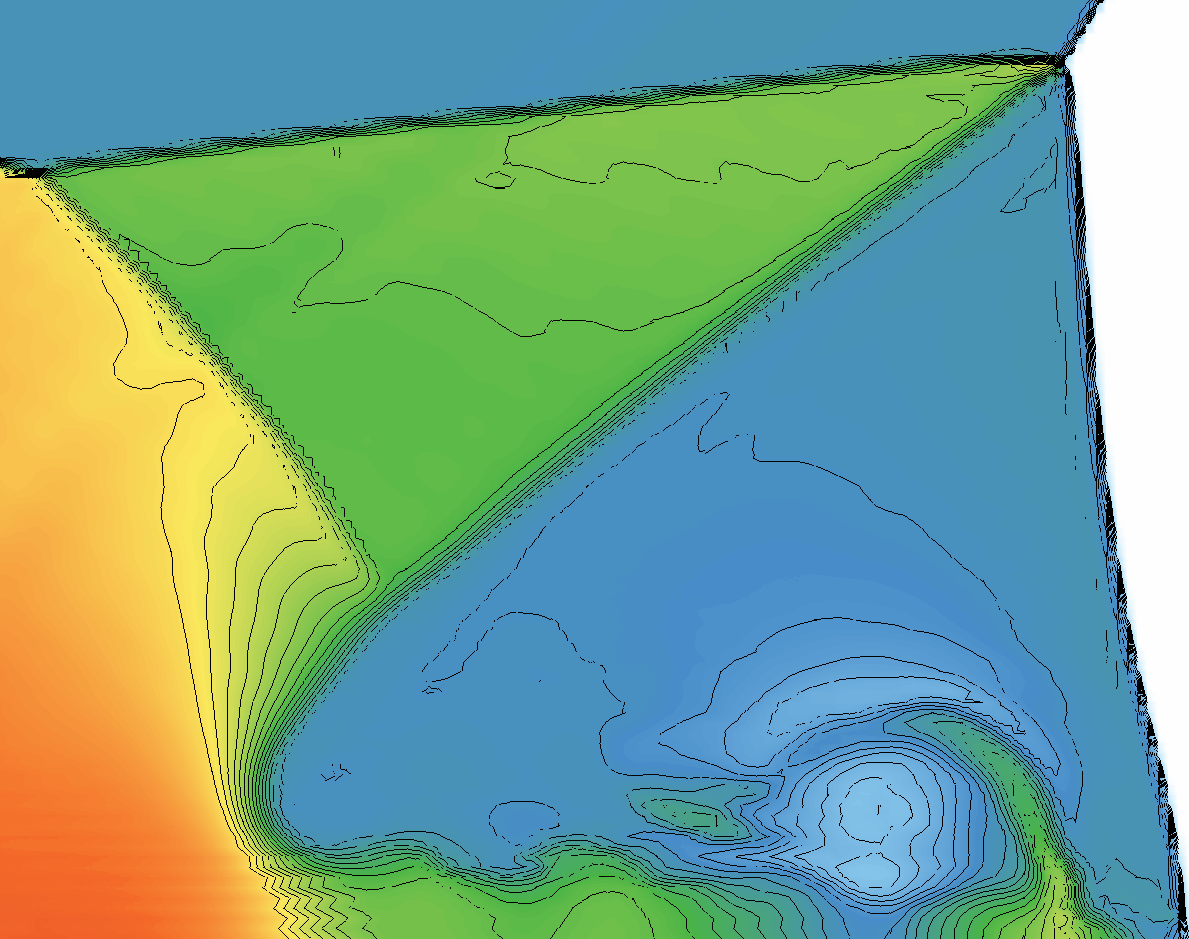}}
\end{subfigure}
\vskip2mm
\begin{subfigure}[b]{\textwidth}
\caption{Post-shock density distribution on a mesh of $4\cdot 384^2$ elements obtained with $\Delta t\,=\,6.25$E-6.}
\label{fig:double-mach-fine}
\includegraphics[scale=0.18]{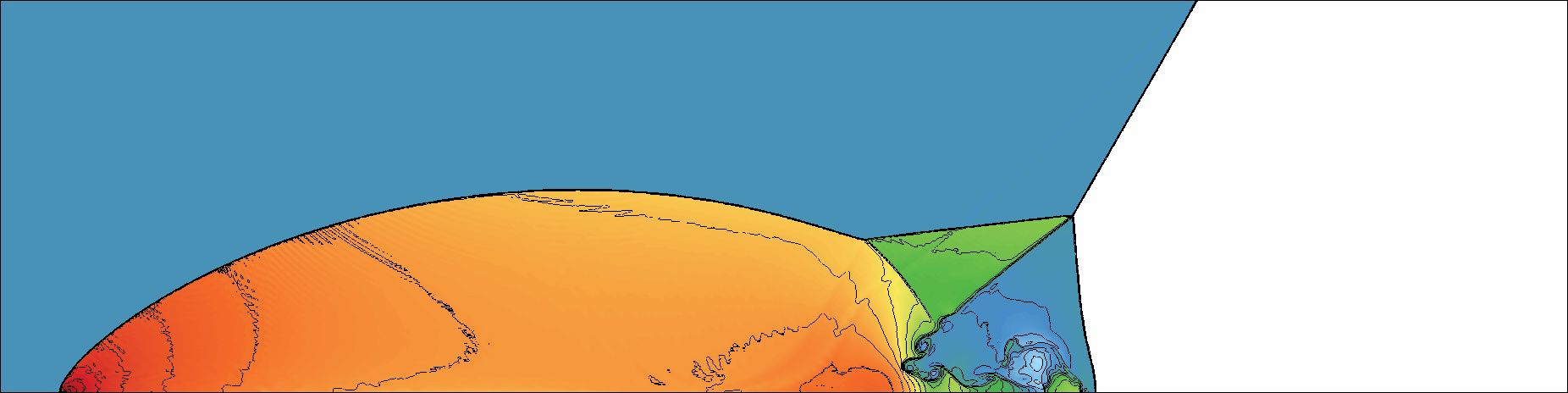}
\fbox{\includegraphics[scale=0.09]{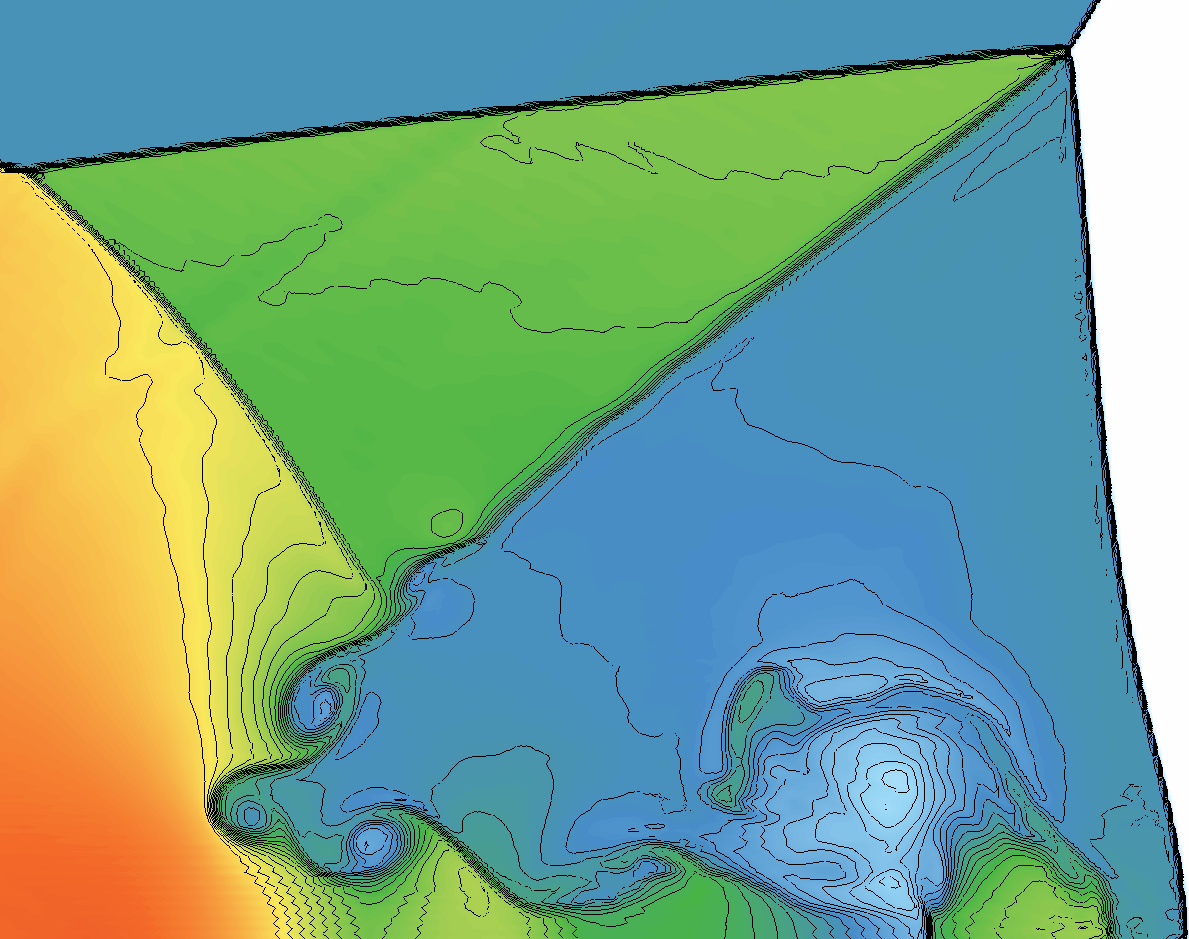}}
\end{subfigure}
\caption{Influence of the mesh size and time step on the accuracy of the
MCL-$\mathbb Q_1$ scheme for the Euler equations, as illustrated
by numerical solutions of the double Mach reflection
problem at time $t=0.2$ with zooms of the triple point region.}
\label{fig:double-mach}
\end{figure}

The snapshots in \cref{fig:double-mach} confirm the conjecture that the MCL limiter is a powerful stabilization technique not only on coarse meshes but also for higher resolution.
For instance, only the finest mesh simulation can capture the small scale eddies in the zoomed region.
On coarser meshes, even the shocks are smeared significantly.
The application of MCL to the DG-$\mathbb Q_1$ target discretization seems to introduce enough numerical dissipation in this example.
Therefore, we were able to run these simulations without the need to deal with negative pressures and perform entropy fixes.
Admittedly, solving this problem with higher order methods does produce negative pressures.
As explained in \cite{kuzmin2020}, positivity preservation can be readily enforced via additional limiting.
This feature makes it possible to solve more challenging problems, so we plan to utilize it in the future.
We expect that running MCL in a safe positivity-preserving mode will enable us to achieve better resolution of the small scale eddies in \cref{fig:double-mach-fine} using higher order spaces on coarser meshes.

\subsection{Shallow water equations}

Finally, we consider the system of shallow water equations
\begin{align*}
u = \begin{bmatrix}
H \\ H \vec v
\end{bmatrix} \in \R^{d+1}, \qquad \mat F(u) = \begin{bmatrix}
H\vec v \\
H \vec v \otimes \vec v + \frac g 2 H^2 \mat I
\end{bmatrix} \in \R^{(d+1)\times d},
\end{align*}
where $H$ denotes the total water height, $\vec v$ is the depth-averaged velocity, and $g$ is the acceleration due to gravity.
Source terms such as bathymetry gradient and friction terms are omitted in order to have a system of the form \eqref{eq:pde}.
An invariant set of the initial value problem for the shallow water equations is
\begin{align*}
\mathcal G = \lc (H,H\vec v)^T \in \R^{d+1}:~H > 0 \rc,
\end{align*}
and an upper bound for the fastest wave speed can be found in \cite{guermond2018a}.

\subsubsection{Radial dam break problem}

This Riemann problem for the shallow water equations in 2D resembles the shock tube example in \cref{sec:sod}.
The whole boundary of the domain $\Omega = (-1,1)^2$ is treated as an outlet (although reflecting wall boundary conditions could also be prescribed).
We choose the initial condition
\begin{align}\label{eq:swe-initial}
U_0(\vec x)= \l H_0, (Hu)_0, (Hv)_0 \r^T= \begin{cases}
(1,0,0)^T & \mbox{if } \left|x\right|  \leq 0.5, \\ (0.1,0,0)^T & \mbox{if } \left|x\right|  > 0.5,
\end{cases}
\end{align}
set $g= 9.81$ and solve the radially symmetric 
Riemann problem up to the time $t=0.06$.
The radial discontinuity of the initial data produces a rarefaction wave which travels towards the domain center, and a radially symmetric shock which propagates in the outward direction.
In contrast to the system of Euler equations, no contact discontinuities appear in the solution for the shallow water equations.

We approximate the solution to this problem on uniform quadrilateral meshes using \#DOF =256$^2$, $p \in \{0,1,3,7\}$ and set $\Delta t\,=\,10^{-4}$ in all tests.
The discrete initial heights $H_h(\cdot,0)$ are obtained by evaluating \eqref{eq:swe-initial} in the Bernstein nodes and using the pointwise values as Bernstein coefficients. 
While this approach is not a projection in the mathematical sense (i.e., $P^2=P$ is violated for $p>1$) and provides only second order accuracy, it is sufficient for this example and has the advantage that it does not introduce violations of maximum principles, as opposed to consistent $L^2$ projections.

\begin{figure}[ht!]
\centering
\begin{subfigure}[b]{0.325\textwidth}
\caption{DG-$\mathbb{Q}_1,H_h(\cdot,0) \in $[0.10,1.00]}
\includegraphics[scale=0.15]{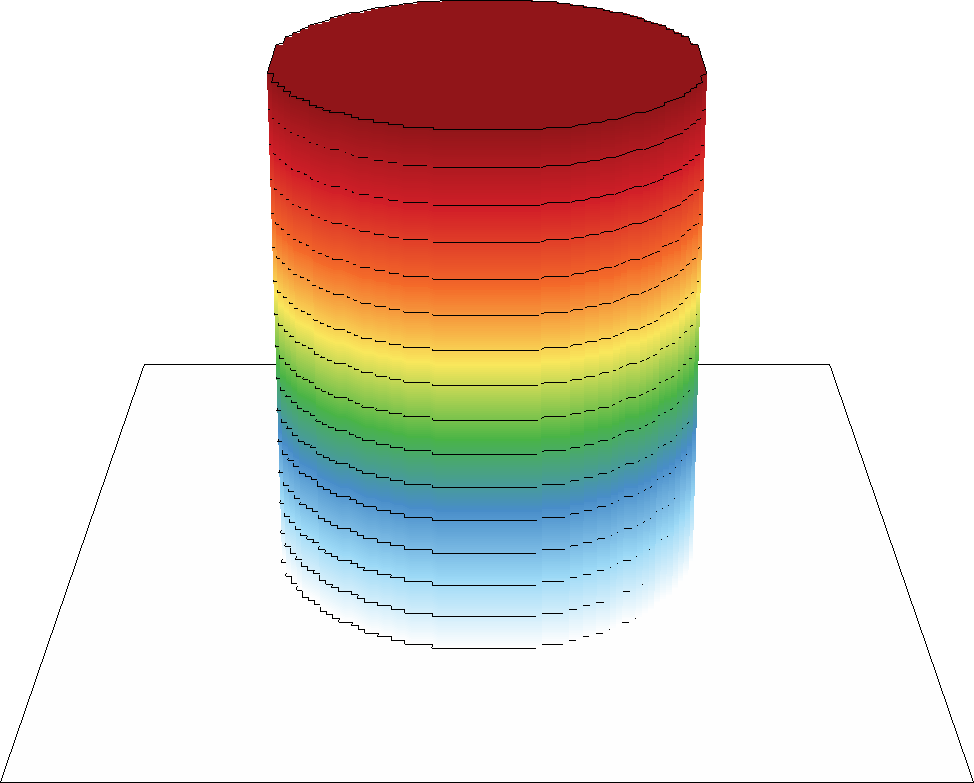}
\end{subfigure}
\begin{subfigure}[b]{0.325\textwidth}
\caption{DG-$\mathbb{Q}_1,~H_h \in\, $[7.98E-3, 1.02]}
\includegraphics[scale=0.15]{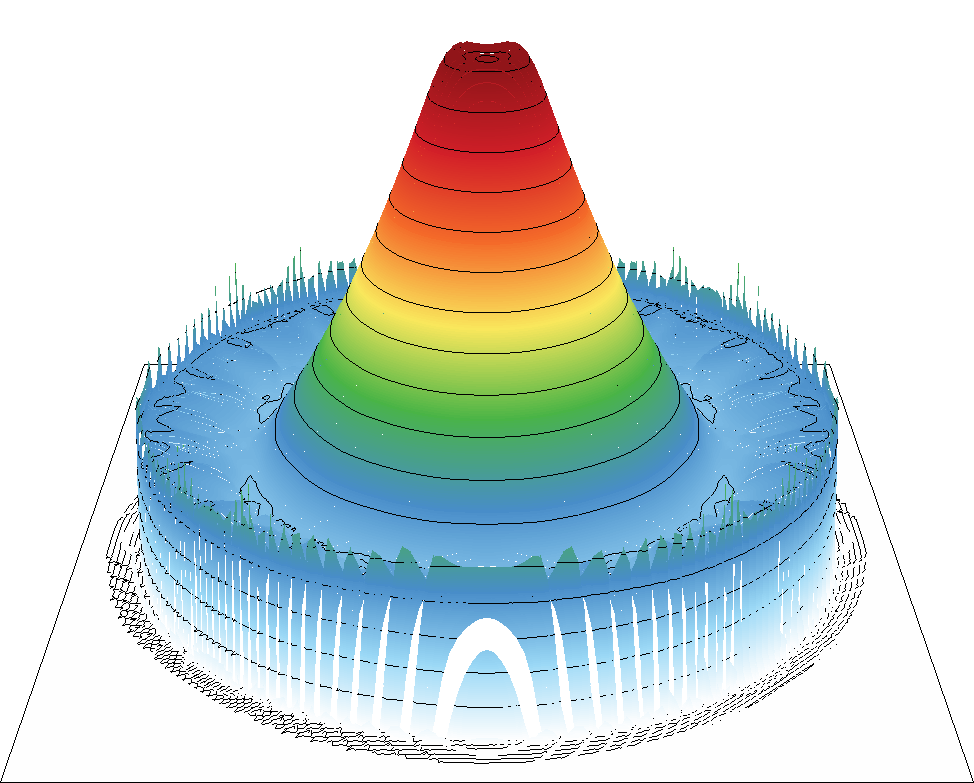}
\end{subfigure}
\begin{subfigure}[b]{0.325\textwidth}
\caption{DG-$\mathbb{Q}_0,~H_h \in\, $[0.10, 0.98]}
\includegraphics[scale=0.15]{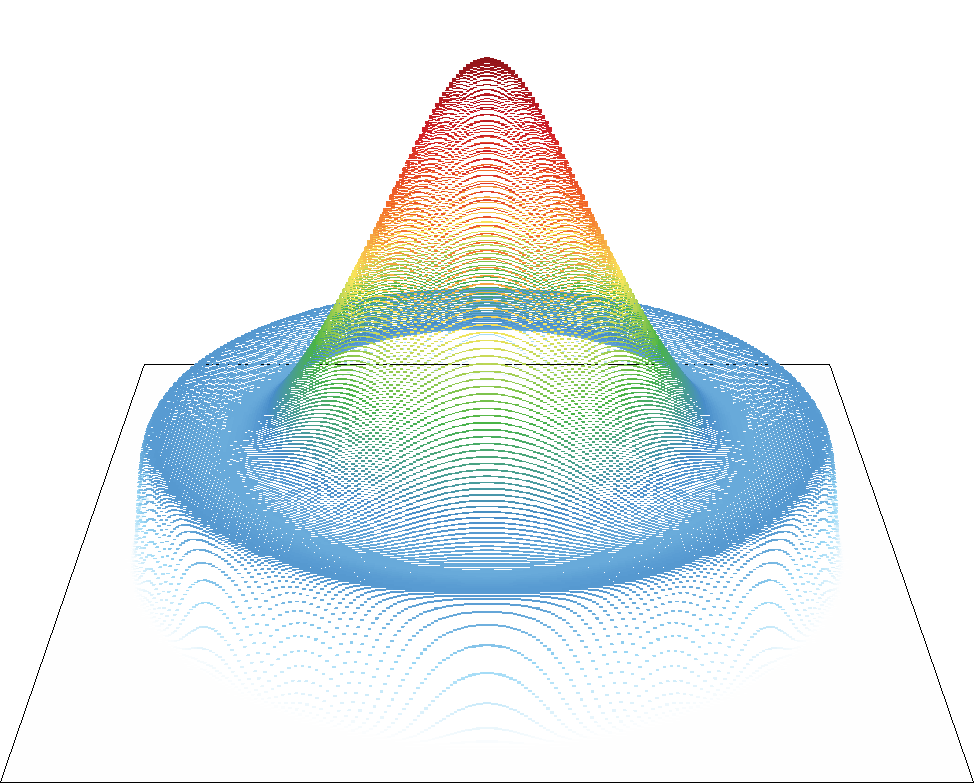}
\end{subfigure}
\vskip5mm
\begin{subfigure}[b]{0.325\textwidth}
\caption{MCL-$\mathbb{Q}_1,~H_h \in\, $[0.10, 1.00]}\label{fig:swe1}
\includegraphics[scale=0.15]{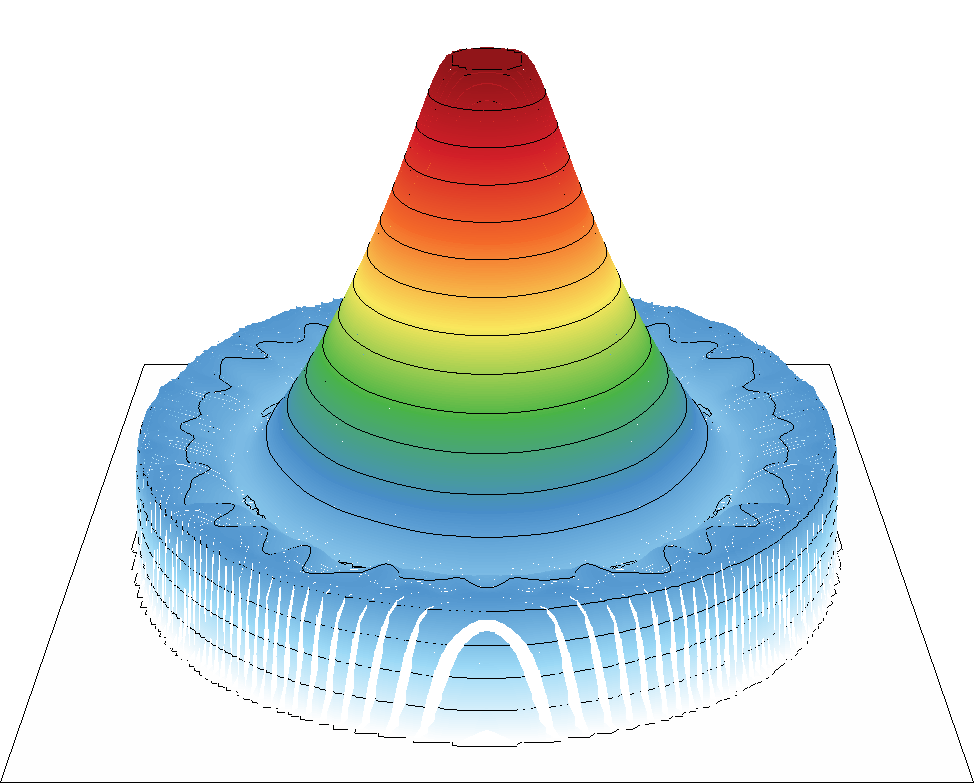}
\end{subfigure}
\begin{subfigure}[b]{0.325\textwidth}
\caption{MCL-$\mathbb{Q}_3,~H_h \in\, $[0.10, 1.00]}
\includegraphics[scale=0.15]{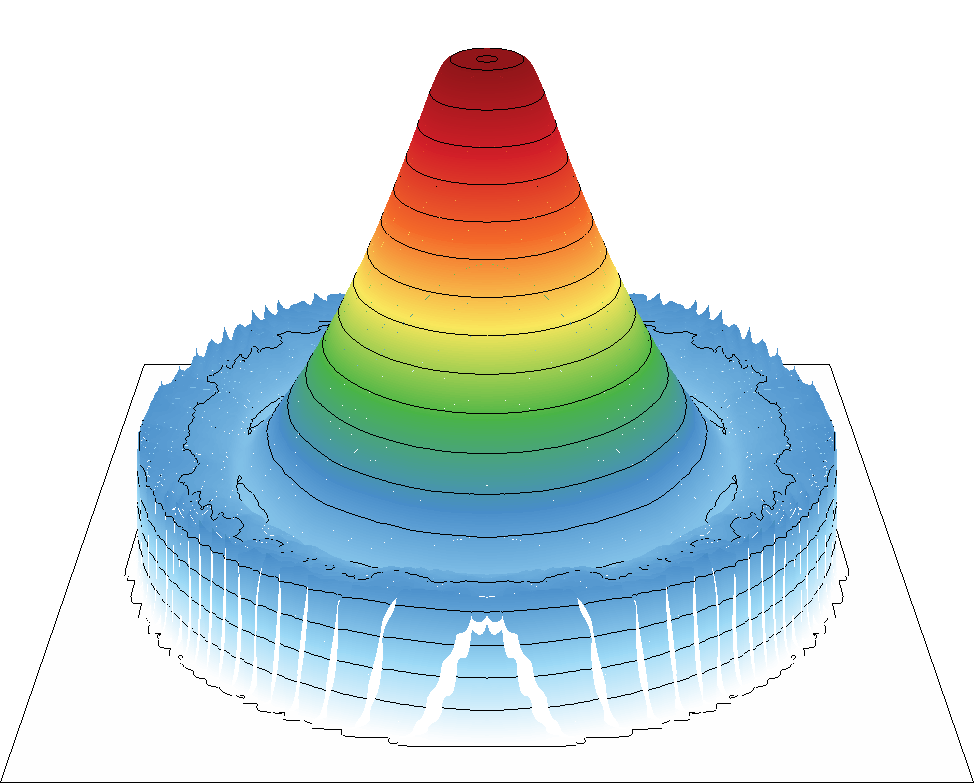}\label{fig:swe3}
\end{subfigure}
\begin{subfigure}[b]{0.325\textwidth}
\caption{MCL-$\mathbb{Q}_7,~H_h \in\, $[0.10, 1.00]}
\includegraphics[scale=0.15]{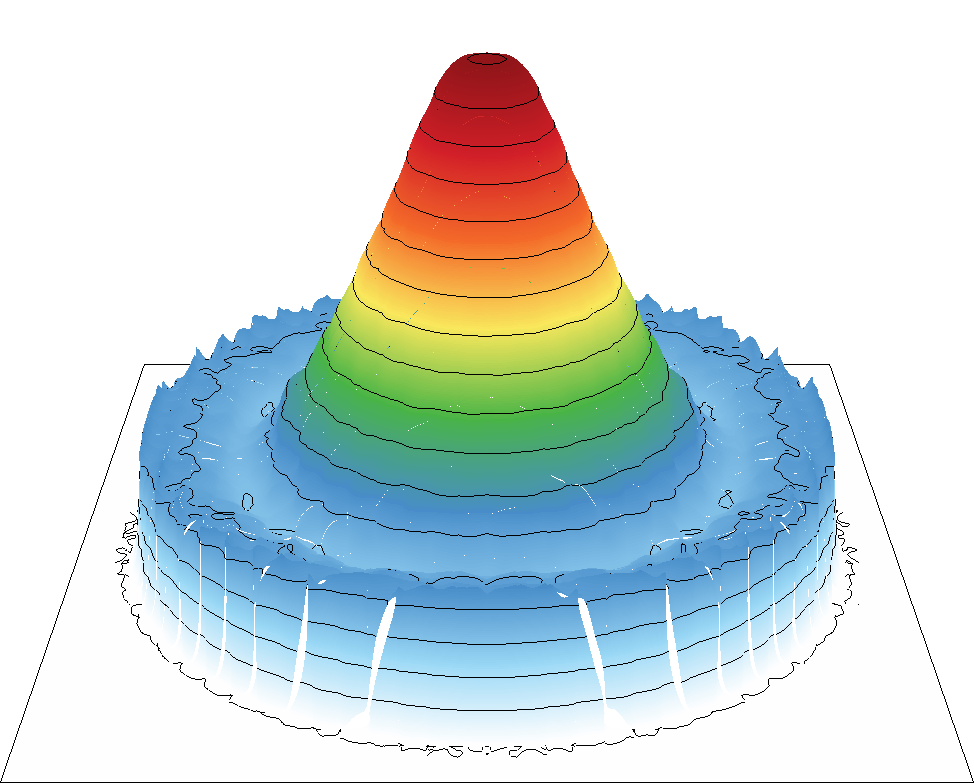}\label{fig:swe7}
\end{subfigure}
\caption{Radial dam break problem for the 2D shallow water equations. The
diagrams show (a) the discrete initial condition $H_h(\cdot,0)$ for the
water height and (b)--(f) numerical approximations
$H_h(\cdot,0.06)$ to the water height at $t=0.06$
obtained with DG and MCL schemes using
$p \in \{0,1,3,7\}$.}\label{fig:swe}
\end{figure}

All of the snapshots presented in \cref{fig:swe} reproduce the propagation of rarefaction and shock waves correctly.
However, for a marginally different setup, the DG-$\mathbb Q_1$ approximation blows up because the oscillations visible along the shock front can cause the algorithm to produce negative water heights and terminate.
Such violations of the IDP property are also visible at the peak of the rarefaction wave but these are clearly not as severe.
The DG-$\mathbb Q_0$ solution yields a rather poor approximation again, while the MCL results combine the advantages of bound preserving and sufficiently accurate approximations.
The ripples observed in the high order approximations \cref{fig:swe3,fig:swe7} are not overshoots but rather remnants of discontinuities on meshes coarser than the one employed in \cref{fig:swe1}. 
Still, due to the lack of smoothness, the most accurate results for a fixed \#DOF are obtained with MCL-$\mathbb Q_1$.

\subsubsection{Narrowing channel}\label{sec:channel}


Finally, we study another shock dominated shallow water problem which was proposed in \cite{zienkiewicz1995}.
The domain 
\begin{align*}
\Omega = \lc(x,y) \in (-10,80)\times (0,40):~ \tan\l \frac{\pi}{36} \r x < y <  40 - \tan \l\frac{\pi}{36}\r x \rc
\end{align*}
is a 40 meters wide and 90 meters long channel.
The lateral wall boundaries are horizontal for the first 10 meters of the channel, after which they are symmetrically constricted with an angle of $5^\circ$.
At the left boundary, the inflow values $U_{\text in} = (1,1,0)^T$ are prescribed.
The right boundary is an outlet, and the remainder of $\partial \Omega$ is a reflecting wall.
The gravitational constant is set to $g = 0.16$, which makes the flow regime supercritical with Froude number $2.5$.
For the initial condition $U_0 = (1,1,0)^T$, shocks start to develop at the boundary constrictions.
They then travel towards the domain center, interact with each other, and are reflected again at the opposite boundary, after which the solution approaches a symmetric steady state.
Since temporal accuracy is irrelevant for such problems, we march numerical solutions to the steady state using the forward Euler method.
Computations are terminated when the stopping criterion
\begin{align}\label{eq:convergence}
r^k \coloneqq \left|U^k - U^{k+1} \right|< 10^{-12}
\end{align}
is satisfied for solution vectors $U^k, \,U^{k+1}$ containing the degrees of freedom for all unknowns at pseudo-time steps $k$ and $k+1$, respectively. 
\begin{remark}
We define $r^k$ in this way to facilitate comparison with limiters for which the stationary nonlinear problem has no well-defined residual.
For monolithic limiters, residual-based stopping criteria are preferable to \eqref{eq:convergence} because stagnation of $U^{k}$ can lead to cancellation effects.
These may result in a termination of the simulation run before the norm of the residual becomes smaller than a prescribed tolerance. If the residual is not decreasing monotonically, one should make sure that it stays as small as desired for a certain number of consecutive steps.
\end{remark}

We solve the narrowing channel problem using pseudo-time step $\Delta t\,=\,0.025$
on an unstructured triangular mesh consisting of E\,=\,12,620 elements.
In \cref{fig:channel}, we display the distribution of the water height $H_h$ at steady state along with the evolution of the convergence indicators $r^k$ for DG-$\mathbb P_1$, LO-$\mathbb P_1$, and MCL-$\mathbb P_1$.
As in other numerical examples, the unlimited solution violates maximum principles, while the low order method smears the shocks almost beyond recognition in this test.
The MCL solution is nonoscillatory and exhibits crisp resolution of shocks, as well as good symmetry preservation properties even on this unstructured mesh which is not aligned with the shock angles.
As expected, DG-$\mathbb P_1$ and LO-$\mathbb P_1$ converge to steady state.
Remarkably, so does MCL-$\mathbb P_1$, which makes it the first DG limiter of its kind.
The fastest convergence is achieved with the unlimited DG scheme.
However, the number of pseudo-time steps for the LO and MCL calculations is still reasonable.

\begin{figure}[ht!]
\centering
\begin{subfigure}[b]{0.49\textwidth}
\caption{DG-$\mathbb{P}_1,~H_h \in\, $[0.81, 1.96]}
\includegraphics[scale=0.12]{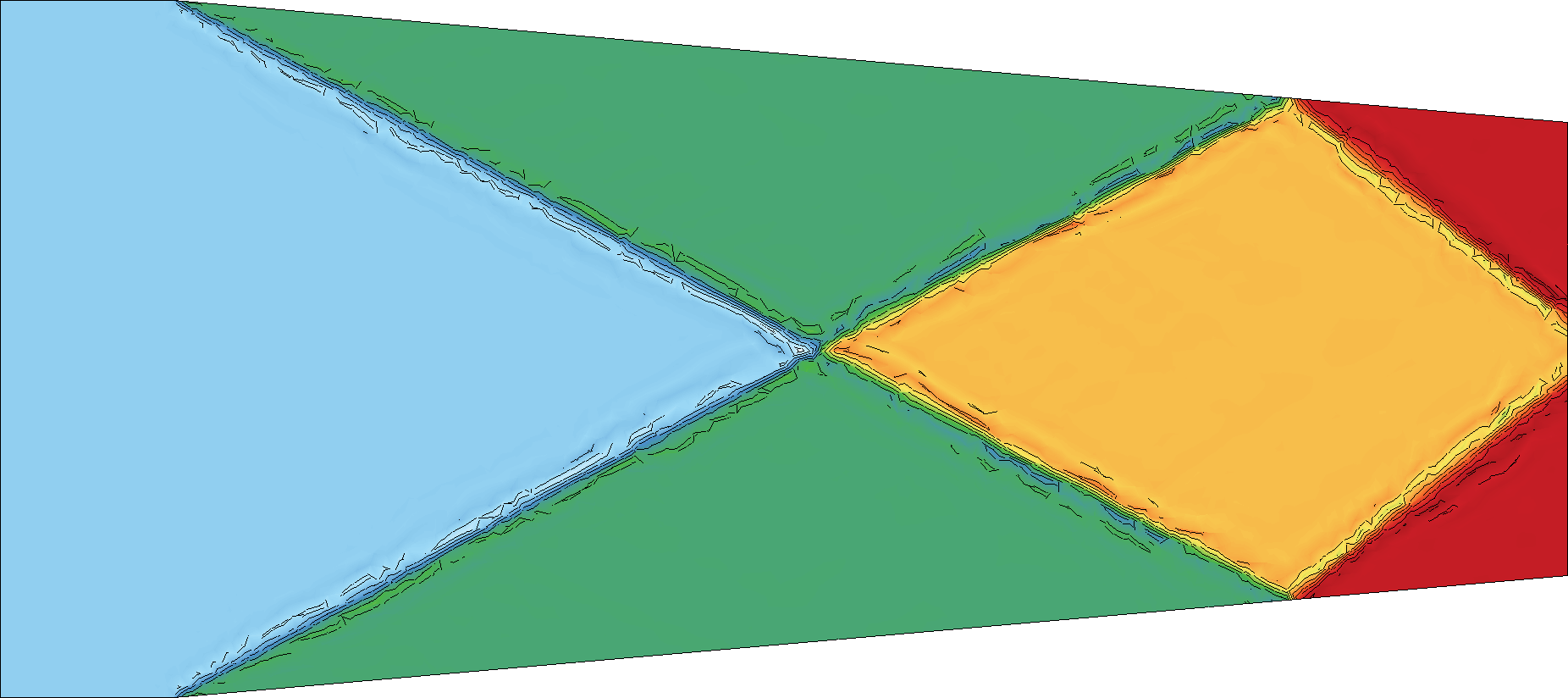}
\end{subfigure}
\begin{subfigure}[b]{0.49\textwidth}
\caption{LO-$\mathbb{P}_1,~H_h \in\, $[1.00, 1.83]}
\includegraphics[scale=0.12]{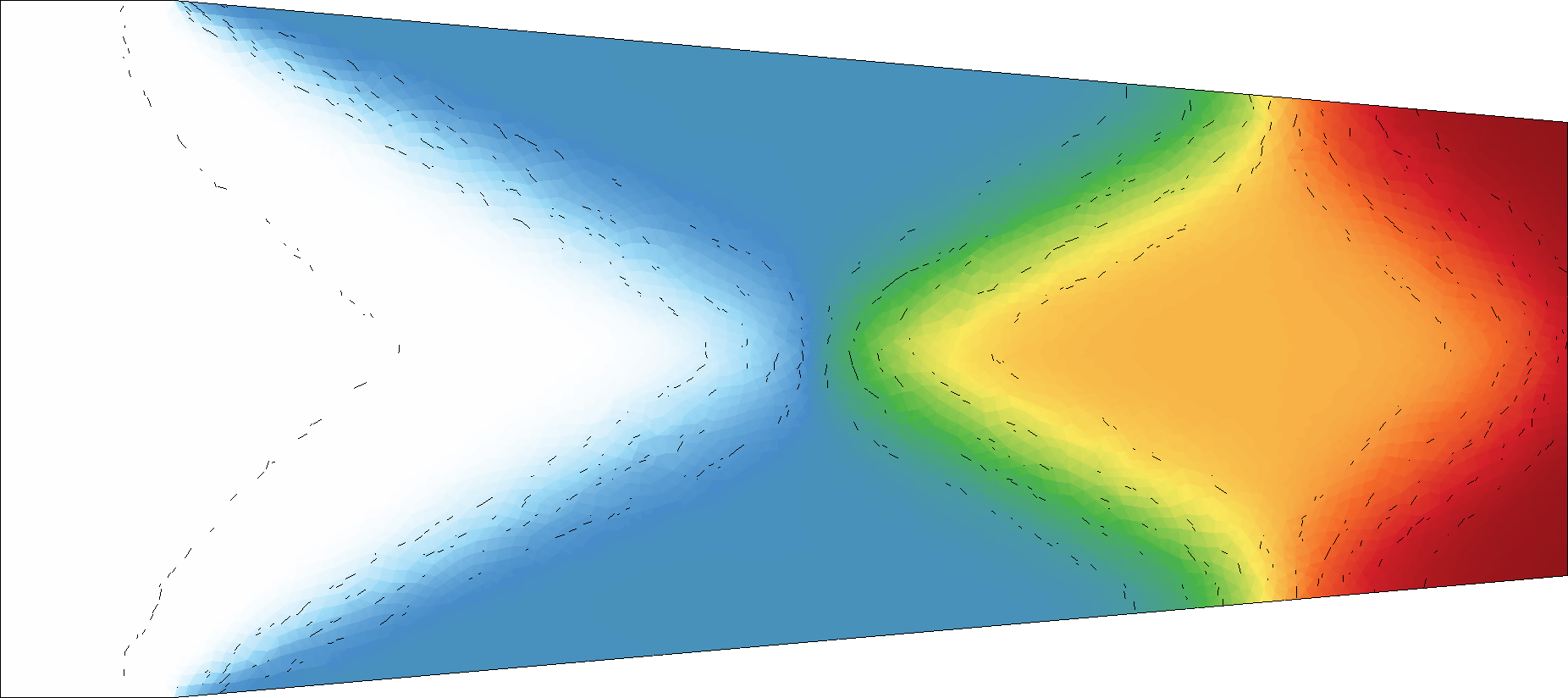}
\end{subfigure}
\vskip5mm
\begin{subfigure}[b]{0.49\textwidth}
\caption{MCL-$\mathbb{P}_1,~H_h \in\, $[1.00, 1.85]}
\includegraphics[scale=0.12]{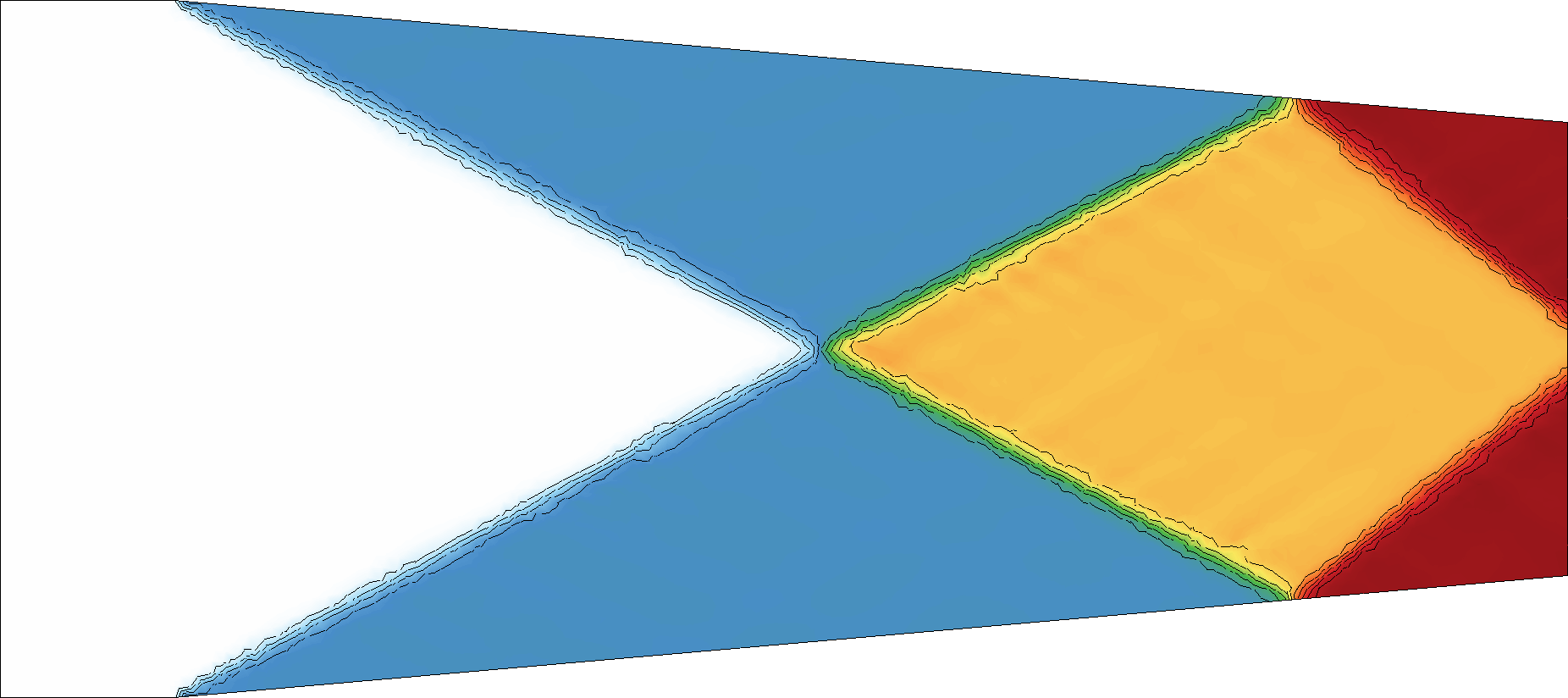}
\end{subfigure}
\begin{subfigure}[b]{0.49\textwidth}
\caption{Evolution of $r^k$ in pseudo time}\label{fig:conv}
\includegraphics[scale=0.146]{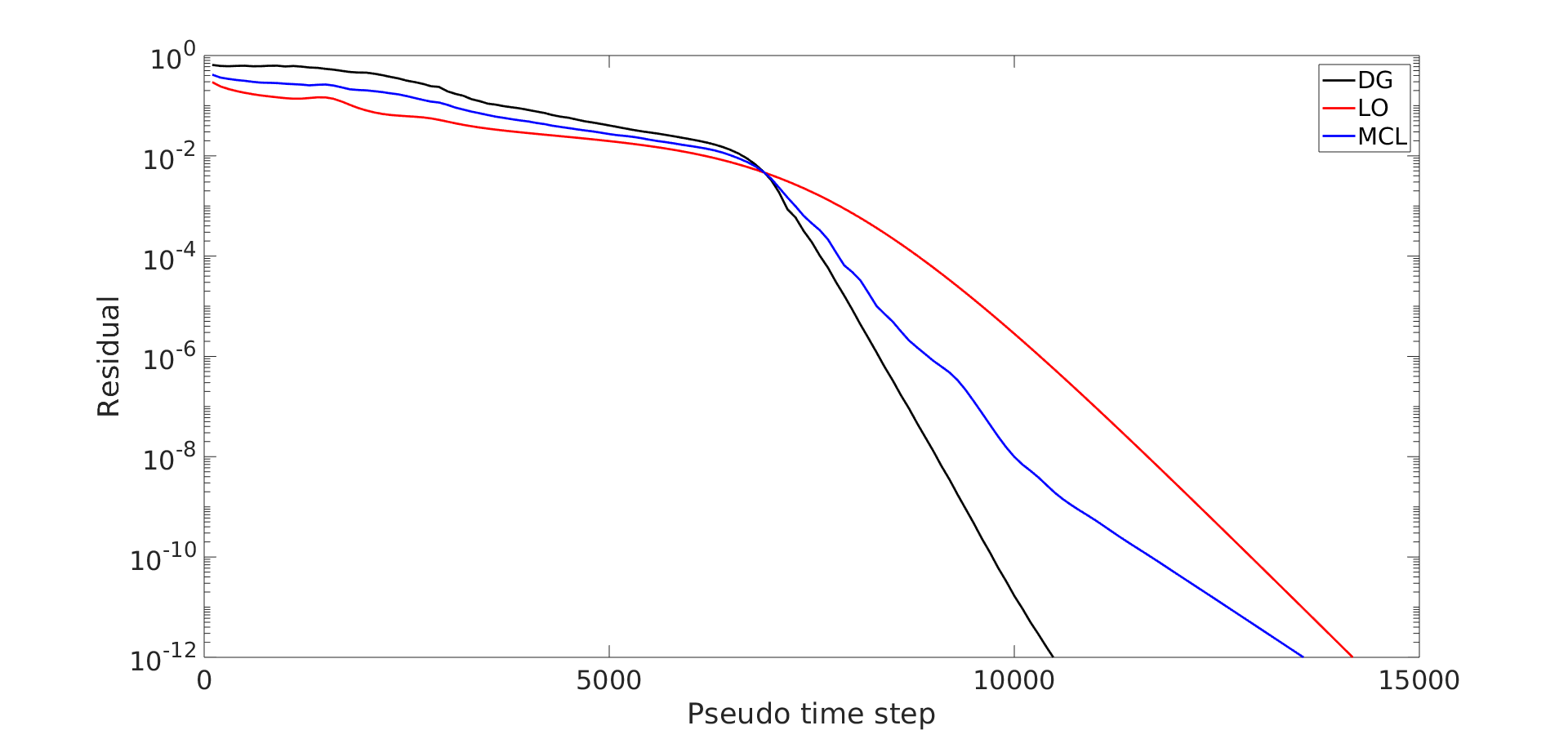}
\end{subfigure}
\caption{Constricted channel flow for the shallow water equations, water height $H_h$ at steady state obtained with (a) DG-$\mathbb P_1$, (b) LO-$\mathbb P_1$, (c) MCL-$\mathbb P_1$. Diagram (d) plots the convergence indicators $r^k$ vs. pseudo time using a logarithmic scale.}\label{fig:channel}
\end{figure}

To further motivate the use of monolithic limiters, we solved the narrowing channel problem with another code for the shallow water equations, a simulation tool based on the open source MATLAB/GNU Octave package FESTUNG \cite{festung}.
Its DG solver is equipped with a vertex-based slope limiter \cite{kuzmin2010} which can be applied (i) just to the water height, (ii) to the water height and momentum separately or (iii) to the water height and momentum sequentially (as in \cite{dobrev2017,hajduk2019}). 
In contrast to MCL, none of these limiting procedures was even close to satisfying the stopping criterion \eqref{eq:convergence}.
Only the low order method (corresponding to DG-$\mathbb P_0$) as well as the unlimited solution (corresponding to DG-$\mathbb P_1$) did converge.
This outcome is not surprising since the underlying slope limiters are used as postprocessing tools that are algebraically equivalent to localized FCT algorithms \cite{dobrev2017,lohmann2017}.
For such methods, convergence to steady-state solutions can not be expected since the residuals depend on the time step size of the pseudo-time stepping method.

We admit that we have also observed stagnation of the convergence indicators $r^k$ for MCL-$\mathbb P_p$ with $p>1$.
This disappointing result motivates further theoretical and numerical studies of monolithic limiters for high order DG methods.
The ability of the MCL-$\mathbb P_1$ scheme to produce a fully converged steady state solution on an unstructured mesh indicates that steady state convergence can also be achieved with high order extensions if the local bounds are properly chosen and forward Euler pseudo-time stepping is replaced with more advanced iterative solvers for the stationary nonlinear problem.


\section{Conclusions}\label{sec:conclusion}

The presented research indicates that monolithic convex limiting is a promising algebraic approach to designing new high order property-preserving DG schemes for hyperbolic conservation laws and systems thereof.
We have shown that some components of existing limiting techniques for continuous finite elements require careful generalizations in the DG setting.
The accuracy and robustness of our method were demonstrated in numerical studies for scalar equations and systems alike.
Thus, we believe that it is suitable for a large class of hyperbolic problems.
To the author's best knowledge, the proposed approach constitutes the first monolithic limiting tool for very high order DG methods.
The fact that we were able to obtain a fully converged steady state solution of the shallow water equations with the MCL-$\mathbb P_1$ version on an unstructured mesh is a very encouraging result as well.

Recent efforts aimed at the development of limiters for high order finite element approximations \cite{anderson2017,hajduk2020a,lohmann2017,pazner-preprint} are strongly motivated by the fact that they are ideally suited for high-performance computing.
The presented limiting methodology can be readily combined with DG discretizations of any order, as illustrated by numerical examples using polynomial degrees of up to $p=31$.
For shock-dominated benchmarks considered in this work, piecewise linear approximations were found to outperform higher order methods using the same number of unknowns.
However, local $p$-refinement may be more beneficial than local $h$-refinement in smooth regions.
Since $hp$ adaptivity is relatively easy to implement for DG methods, we envisage that the proposed methodology may be used to enforce preservation of local bounds in this context.
Employing a smoothness indicator that restricts the use of limiting to `troubled' cells (cf. \cite{dumbser2014,guermond2018,hajduk2020b,pazner-preprint,vilar2019}) would reduce the computational cost and further improve the capacity of our method to achieve optimal accuracy while preserving important qualitative properties of the exact solution.

The proposed schemes performed reasonably
well for all benchmarks considered in the
presented study. However, when it comes to solving nonlinear conservation
laws with nonunique solutions, additional difficulties can arise even if
all local maximum principles are satisfied. Nonphysical behavior of
limited high-order solutions is frequently observed, e.g., for
the KPP problem \cite{kurganov2007} or the Woodward-Colella blast
wave benchmark \cite{woodward1984}. In most cases, it can be cured
by performing
additional limiting to enforce entropy conditions (cf.
\cite{kuzmin-preprint3,kuzmin-preprint1,kuzmin-preprint2})
and/or positivity preservation (cf.
\cite{dobrev2017,hajduk2019,kuzmin2020}).

\section*{Acknowledgments}
The author would like to thank Prof. Dmitri Kuzmin (TU Dortmund University) for many insightful suggestions and supporting this effort in general.
Special thanks go also to Dr. Manuel Quezada de Luna (King Abdullah University of Science and Technology) for many fruitful discussions.

\appendix
\section{Box element dissipation along Cartesian cross stencil}
Lohmann et al. \cite{lohmann2017} proved that the application of their preconditioner to the 1D element matrix of the discretized advection operator reduces the stencil of the high order Bernstein finite element approximation to the set of nearest neighbors in the case of constant velocity.
The sparsity of $\tilde{\vec C}^e$ in 1D follows from this result.
Kuzmin and Quezada de Luna \cite{kuzmin2020a} generalized it to multiple dimensions and simplicial elements.
In the case of box elements, their analysis can be further refined by showing that the sparse matrices $\tilde C_k^e$ have compact Cartesian cross stencils, i.e., that their entries associated with pairs of diagonal nearest neighbors are zero.
This remarkable property was already mentioned in \cref{rem:cross}.
In this appendix, we prove it under the assumption that the element Jacobian $\mat J^e$ is constant or is approximated by a constant matrix.
If this assumption is not satisfied, then neither $\tilde C_k^e$ nor $D^e$ are sparse but the magnitude of their entries corresponding to pairs of far away neighbors is small compared to pairs of nearest neighbors.

Without loss of generality, we assume lexicographical ordering of degrees of freedom within the element. 
Let us first consider the reference element $\hat K$ and the associated basis functions $\hat \phi_i$ which are products of 1D Bernstein polynomials $\hat b_{i_l}$ of degree $p$ (this superscript is omitted for readability)
\begin{align*}
\hat \phi_i(\hat{\vec x}) = \prod_{l=1}^d \hat b_{i_l}(\hat x_l), \qquad
\frac{\partial \phi_i(\hat{\vec x})}{\partial \hat x_k} = \hat b_{i_k}^\prime(\hat x_k)\prod_{\genfrac{}{}{0pt}{}{l=1}{l\neq k}}^d \hat b_{i_l}(\hat x_l).
\end{align*}
The reference element matrices $\hat{\vec C} = \l \hat C_1,\dots, \hat C_d \r$ of the (consistent) discrete gradient operator are given by
\begin{align*}
\l\hat C_k\r_{ij} \coloneqq \int_{\hat K} \hat \phi_i \frac{\partial \hat \phi_j}{\partial \hat x_k} \dhx = \int_0^1 \hat b_{i_k}(\hat x_k) \hat b_{j_k}^\prime(\hat x_k) \d\hat x_k \prod_{\genfrac{}{}{0pt}{}{l=1}{l\neq k}}^d \int_0^1 \hat b_{i_l}(\hat x_l) \hat b_{j_l}(\hat x_l) \d\hat x_l.
\end{align*}
Let $\mat I \in \R^{(p+1)\times(p+1)}$ be the identity matrix.
Then $M_C^e,\,\hat C_k$ can be computed from Kronecker products, denoted by $\otimes$, of their 1D counterparts $\hat M_C^{1D},\,\hat C^{1D} \in \R^{(p+1)\times (p+1)}$
as follows:
\begin{align*}
&\l \hat M_C^{1D} \r_{ij} = \int_0^1 \hat b_i(\hat x) \hat b_j(\hat x) \d\hat x, \qquad
\l \hat C^{1D} \r_{ij} = \int_0^1 \hat b_i(\hat x) \hat b_j^\prime(\hat x) \d\hat x, \\
&M_C^e = |K^e| \underbrace{\hat M_C^{1D}\otimes \dots \otimes  \hat M_C^{1D}}_{d~\mathrm{times}}, \qquad
M_L^e = \frac{|K^e|}{(p+1)^d} \underbrace{\mat I \otimes \dots \otimes \mat I}_{d~\mathrm{times}}, \\
&\hat C_k = \underbrace{\hat M_C^{1D}\otimes \dots \otimes  \hat M_C^{1D}}_{(d-k)~\mathrm{times}} \otimes\; \hat C^{1D} \otimes
\underbrace{\hat M_C^{1D}\otimes \dots \otimes  \hat M_C^{1D}}_{(k-1)~\mathrm{times}},\qquad k \in \{1,\dots,d\}.
\end{align*}
Next, we utilize the associativity of $\otimes$, as well as the identities
\begin{align*}
(A \otimes B)^{-1} = A^{-1} \otimes B^{-1}, \qquad (A \otimes B) (C\otimes D) = (AC) \otimes (BD)
\end{align*}
for matrices $A,B,C,D$ of appropriate size.
Invoking integral transformation rules and applying the preconditioner $P^e = M_L^e (M_C^e)^{-1}$ to $\vec C^e = |K^e| (\mat J^e)^{-T} \hat{\vec C} =\, $adj$(\mat J^e)^T \hat{\vec C}$, where adj$(\cdot)$ denotes the adjugate of a matrix, we observe that
\begin{align}\label{eq:outerprod}
&\tilde C_k^e = \frac{|K^e|}{(p+1)^d} \l |K^e| \hat M_C^{1D}\otimes \dots \otimes  \hat M_C^{1D} \r^{-1}
\l \sum_{l=1}^d \text{adj}(\mat J^e)_{lk} \hat C_l \r \notag\\
&= \sum_{l=1}^d \frac{\text{adj}(\mat J^e)_{lk}}{(p+1)^d} \l \hat M_C^{1D}\otimes \dots \otimes  \hat M_C^{1D} \r^{-1} 
\l \hat M_C^{1D}\otimes \dots \otimes \hat C^{1D} \otimes\dots\otimes  \hat M_C^{1D} \r \notag\\
&= \sum_{l=1}^d \frac{\text{adj}(\mat J^e)_{lk}}{(p+1)^d} \l \mat I \otimes\dots\otimes\mat I \otimes \l\l \hat M_C^{1D}\r^{-1} \hat C^{1D} \r \otimes \mat I \otimes\dots\otimes\mat I  \r.
\end{align}
The fact that $\tilde C_k^e$  has zero entries for all pairs of diagonal neighbors follows from the observation that \eqref{eq:outerprod} admits a decomposition into matrices corresponding to each spatial dimension. 
Due to \eqref{eq:dMat}, the dissipative matrices $D^e$ have the same sparsity pattern as $\tilde C_k^e.\hfill \square$
\begin{example}
For $p=1$ we have
\begin{align*}
\hat M_C^{1D} = \frac 1 6\begin{bmatrix}
2 & 1 \\ 1 & 2
\end{bmatrix}, \qquad
2\,\hat C^{1D} = \l \hat M_C^{1D}  \r^{-1} \hat C^{1D} = \begin{bmatrix}
-1 & 1 \\ -1 & 1
\end{bmatrix}
\end{align*}
and, therefore, for a (quadrilateral) box element in 2D, it follows that
\begin{align*}
\tilde C_1^e = \frac{\text{adj}(\mat J^e)_{11}}{4} \begin{bmatrix*}[r]
-1 & \phantom{-}1 & 0 & \phantom{-}0 \\
-1 & 1 & 0 & 0 \\
0 &0 & -1 & 1 \\
0 &0 & -1 & 1
\end{bmatrix*}
+ \frac{\text{adj}(\mat J^e)_{21}}{4} \begin{bmatrix*}[r]
-1 & 0 & \phantom{-}1 & \phantom{-}0 \\
0 & -1 & 0 & 1 \\
-1 & 0 & 1 & 0 \\
0 & -1 & 0 & 1
\end{bmatrix*}.
\end{align*}
The matrix $\tilde C_2^e$ is obtained similarly.
Note that the entries corresponding to pairs of diagonal neighbors $(1,4),\,(2,3),\,(3,2),\,(4,1)$ are zero.
\end{example}

\section{Calculation of preconditioned gradient matrices}
Formula \eqref{eq:outerprod} not only illustrates the proposition that dissipation is only added in Cartesian directions but also presents a practical way of computing $\tilde C_k^e$ similarly to the 1D case. 
For box elements, $\tilde C_k^e$ should be calculated via \eqref{eq:outerprod} because computation/application of the preconditioner $P^e$ requires inversion of $\hat M_C^{1D}$, which may give rise to significant errors in the case of high order Bernstein polynomials.
This issue arises due to the ill-conditioning of Bernstein mass matrices, which can destroy the sparsity of $\tilde C_k^e$ for orders as small as $p=4$.
Fortunately, we do not need to invert $\hat M_C^{1D}$ since we can apply a formula derived by Lohmann et al. \cite{lohmann2017} to our setting thus:
\begin{align*}
\l \l \hat M_C^{1D}  \r^{-1} \hat C^{1D} \r_{ij} = \begin{cases}
p+2-j & \mbox{if } i=j-1, \\
2(j-1)-p & \mbox{if } i=j, \\
-j & \mbox{if } i=j+1, \\
0 & \mbox{otherwise}. \\
\end{cases}
\end{align*}

Matters become more involved if simplices are considered because a decomposition like \eqref{eq:outerprod} is impossible.
Using degree elevation formulas for Bernstein polynomials \cite{kirby2011}, Kuzmin and Quezada de Luna \cite{kuzmin2020a} derived the identity
\begin{align}\label{eq:simplical}
\nabla b_{\vec \alpha}^p = \sum_{k,l=1}^{d+1}\max\{0, \min\{p, (\alpha_l - \delta_{kl} +1)\}\} \nabla a_k b_{\vec \alpha - \vec e_k + \vec e_l}^p
= \sum_{|\vec \beta| = p} \vec \mu_{\vec \beta} b_{\vec \beta}^p,
\end{align}
where $a_k$ are the barycentric coordinates.
From \eqref{eq:simplical}, one can see that the coefficients $\vec \mu_{\vec \beta}$ are nonzero only for nearest neighbors (for simplices one has also to consider diagonal neighbors).
For a given space dimension $d$ and polynomial degree $p$, the second identity in \eqref{eq:simplical} can be used to compute $\vec \mu_{\vec \beta}$, which was shown in \cite{kuzmin2020a} to coincide with the entries of $\l M_L^e\r^{-1}\tilde{\vec C}^e$ corresponding to pairs of nodes associated with multi-indices $\vec \alpha,\,\vec \beta$.
Below we present MATLAB routines that compute $\tilde{\hat C}_k$ on the reference simplex for any~$p$. These codes are available on GitHub \cite{hajduk2020c} as well.
\cref{code:simplex} relies on \cref{code:tri1,code:tri2}, which assume lexicographical numbering of nodes and only cover the case of triangular elements ($d=2$).
\cref{code:simplex} remains valid for tetrahedra. However, one has to implement analogous mappings.

\begin{lstlisting}[caption={Calculation of preconditioned gradient matrices on simplicial elements.}\label{code:simplex}]
function CTilde = computePrecMatSimplex(dim, p)
N = nchoosek(p+dim,p);    % number of local degrees of freedom
CTilde = zeros(N,N,dim);  % initialization
i2Alpha = mapI2Alpha(p);  % array mapping local indices to multiindices
volRefSimplex = 1 / factorial(dim); % |\hat K|
M_L = volRefSimplex / N;  % diagonal entries of lumped mass matrix \hat M_L
gradBary = [-ones(1,dim); eye(dim)]; % gradients of barycentric coordinates
for m = 1:dim             % compute \tilde{ \hat C}_m
  for j = 1:N             % compute (\tilde{ \hat C}_m)_{:,j}
    alpha = i2Alpha(j,:); % multiindex for j-th local index
    for k = 1:dim+1
      for l = 1:dim+1
        beta = alpha;
        beta(k) = beta(k) - 1;
        beta(l) = beta(l) + 1;
        if (beta(l) > p || beta(k) < 0)
          continue;
        end
        i = getIfromAlpha(p,beta); % local index corresponding to beta
        CTilde(i,j,m) = CTilde(i,j,m) + gradBary(k,m) ...          % cf.
                                        * (alpha(l) - (l==k) + 1); % (B.1)
      end
    end
  end
end
CTilde = M_L * CTilde; % multiplication with lumped mass matrix
end
\end{lstlisting}

\begin{lstlisting}[caption={Mapping of local numbers to multiindices on triangular elements.}\label{code:tri1}]
function i2alpha = mapI2Alpha(p)
dim = 2;                      % spacial dimension
N = nchoosek(p+dim,p);        % number of local degrees of freedom
i2alpha = zeros(N,dim+1);     % initialization
a = [p 0 0];                  % first multiindex
ctr = 1;                      % counter variable
for j = 0:p                   % step through DOF in y-direction
  for i = 0:p-j               % step through DOF in x-direction
    i2alpha(ctr,:) = a;       % fill map entry
    ctr = ctr+1;              % update counter
    a(1:2) = a(1:2) + [-1 1]; % go to next multiindex in x-direction
  end
  a(1) = p-1-j;               % go to next multiindex in y-direction
  a(2) = 0;                   % go to first DOF in x-direction
  a(3) = a(3)+1;              % go to next DOF in y-direction
end
end
\end{lstlisting}

\begin{lstlisting}[caption={Mapping of multiindices to local numbers on triangular elements.}\label{code:tri2}]
function i = getIfromAlpha(p,alpha)
i = (p+1)*alpha(3) - alpha(3)*(alpha(3)-1)/2 + alpha(2);
i = i+1; % MATLAB counts indices starting at 1.
end
\end{lstlisting}

\bibliographystyle{abbrvnat}
\bibliography{\pathLaTeX/references}

\begin{thebibliography}{50}
\providecommand{\natexlab}[1]{#1}
\providecommand{\url}[1]{\texttt{#1}}
\expandafter\ifx\csname urlstyle\endcsname\relax
  \providecommand{\doi}[1]{doi: #1}\else
  \providecommand{\doi}{doi: \begingroup \urlstyle{rm}\Url}\fi

\bibitem[mfe()]{mfem}
{MFEM}: Modular finite element methods library.
\newblock \url{https://mfem.org}.

\bibitem[Abgrall(2006)]{abgrall2006}
R.~Abgrall.
\newblock Essentially non-oscillatory residual distribution schemes for
  hyperbolic problems.
\newblock \emph{J. Comput. Phys.}, 214:\penalty0 773--808, 2006.
\newblock \doi{10.1016/j.jcp.2005.10.034}.

\bibitem[Abgrall and Trefil{\'i}k(2010)]{abgrall2010}
R.~Abgrall and J.~Trefil{\'i}k.
\newblock An example of high order residual distribution scheme using
  non-lagrange elements.
\newblock \emph{J. Sci. Comput.}, 45:\penalty0 3--25, 2010.
\newblock \doi{10.1007/s10915-010-9405-y}.

\bibitem[Ainsworth et~al.(2011)Ainsworth, Andriamaro, and
  Davydov]{ainsworth2011}
M.~Ainsworth, G.~Andriamaro, and O.~Davydov.
\newblock Bernstein--{B}{\'e}zier finite elements of arbitrary order and
  optimal assembly procedures.
\newblock \emph{SIAM J. Sci. Comput.}, 33:\penalty0 3087--3109, 2011.
\newblock \doi{10.1137/11082539X}.

\bibitem[Anderson et~al.(2017)Anderson, Dobrev, Kolev, Kuzmin, Quezada~de Luna,
  Rieben, and Tomov]{anderson2017}
R.~Anderson, V.~Dobrev, T.~Kolev, D.~Kuzmin, M.~Quezada~de Luna, R.~Rieben, and
  V.~Tomov.
\newblock High-order local maximum principle preserving ({MPP}) discontinuous
  {G}alerkin finite element method for the transport equation.
\newblock \emph{J. Comput. Phys.}, 334:\penalty0 102--124, 2017.
\newblock \doi{10.1016/j.jcp.2016.12.031}.

\bibitem[Anderson et~al.(2019)Anderson, Andrej, Barker, Bramwell, Camier,
  Cerveny, Dobrev, Dudouit, Fisher, Kolev, et~al.]{anderson2019}
R.~Anderson, J.~Andrej, A.~Barker, J.~Bramwell, J.-S. Camier, J.~Cerveny,
  V.~Dobrev, Y.~Dudouit, A.~Fisher, T.~Kolev, et~al.
\newblock {MFEM}: a modular finite element methods library.
\newblock \emph{arXiv preprint arXiv:1911.09220}, 2019.

\bibitem[Badia et~al.(2017)Badia, Bonilla, and Hierro]{badia2017}
S.~Badia, J.~Bonilla, and A.~Hierro.
\newblock Differentiable monotonicity-preserving schemes for discontinuous
  {G}alerkin methods on arbitrary meshes.
\newblock \emph{Comput. Method. Appl. M.}, 320:\penalty0 582--605, 2017.
\newblock \doi{10.1016/j.cma.2017.03.032}.

\bibitem[Barrenechea et~al.(2016)Barrenechea, John, and
  Knobloch]{barrenechea2016}
G.~R. Barrenechea, V.~John, and P.~Knobloch.
\newblock Analysis of algebraic flux correction schemes.
\newblock \emph{SIAM J. Numer. Anal.}, 54:\penalty0 2427--2451, 2016.
\newblock \doi{10.1137/15M1018216}.

\bibitem[Barth and Jespersen(1989)]{barth1989}
T.~Barth and D.~Jespersen.
\newblock The design and application of upwind schemes on unstructured meshes.
\newblock In \emph{Proc. AIAA 27th Aerospace Sciences Meeting, Reno}, page
  0366, 1989.
\newblock \doi{10.2514/6.1989-366}.

\bibitem[Berthon(2005)]{berthon2005}
C.~Berthon.
\newblock Stability of the {MUSCL} schemes for the {E}uler equations.
\newblock \emph{Commun. Math. Sci.}, 3:\penalty0 133--157, 2005.
\newblock URL \url{projecteuclid.org/euclid.cms/1118778272}.

\bibitem[Boris and Book(1973)]{boris1973}
J.~P. Boris and D.~L. Book.
\newblock Flux-corrected transport. {I. SHASTA}, a fluid transport algorithm
  that works.
\newblock \emph{J. Comput. Phys.}, 11:\penalty0 38--69, 1973.
\newblock \doi{10.1016/0021-9991(73)90147-2}.

\bibitem[Dobrev et~al.(2018)Dobrev, Kolev, Kuzmin, Rieben, and
  Tomov]{dobrev2017}
V.~Dobrev, T.~Kolev, D.~Kuzmin, R.~Rieben, and V.~Tomov.
\newblock Sequential limiting in continuous and discontinuous {G}alerkin
  methods for the {E}uler equations.
\newblock \emph{J. Comput. Phys.}, 356:\penalty0 372 -- 390, 2018.
\newblock \doi{10.1016/j.jcp.2017.12.012}.

\bibitem[Dumbser et~al.(2014)Dumbser, Zanotti, Loub{\`e}re, and
  Diot]{dumbser2014}
M.~Dumbser, O.~Zanotti, R.~Loub{\`e}re, and S.~Diot.
\newblock A posteriori subcell limiting of the discontinuous {G}alerkin finite
  element method for hyperbolic conservation laws.
\newblock \emph{J. Comput. Phys.}, 278:\penalty0 47--75, 2014.
\newblock \doi{10.1016/j.jcp.2014.08.009}.

\bibitem[Frank et~al.(2020)Frank, Reuter, Aizinger, Hajduk, and Rupp]{festung}
F.~Frank, B.~Reuter, V.~Aizinger, H.~Hajduk, and A.~Rupp.
\newblock {FESTUNG}: {T}he {F}inite {E}lement {S}imulation {T}oolbox for
  {UN}structured {G}rids, 2020.
\newblock URL \url{https://github.com/FESTUNG}.
\newblock Version 1.0.

\bibitem[Godunov(1959)]{godunov1959}
S.~K. Godunov.
\newblock A difference method for numerical calculation of discontinuous
  solutions of the equations of hydrodynamics.
\newblock \emph{Mat. Sb.}, 89:\penalty0 271--306, 1959.
\newblock URL \url{http://mi.mathnet.ru/eng/msb4873}.

\bibitem[Gottlieb et~al.(2001)Gottlieb, Shu, and Tadmor]{gottlieb2001}
S.~Gottlieb, C.-W. Shu, and E.~Tadmor.
\newblock {S}trong {S}tability-{P}reserving {H}igh-{O}rder {T}ime
  {D}iscretization {M}ethods.
\newblock \emph{SIAM Rev.}, 43:\penalty0 89--112, 2001.
\newblock \doi{10.1137/S003614450036757X}.

\bibitem[Guermond and Nazarov(2014)]{guermond2014a}
J.-L. Guermond and M.~Nazarov.
\newblock A maximum-principle preserving {C}0 finite element method for scalar
  conservation equations.
\newblock \emph{Comput. Method. Appl. M.}, 272:\penalty0 198--213, 2014.
\newblock \doi{10.1016/j.cma.2013.12.015}.

\bibitem[Guermond and Popov(2016{\natexlab{a}})]{guermond2016}
J.-L. Guermond and B.~Popov.
\newblock Invariant domains and first-order continuous finite element
  approximation for hyperbolic systems.
\newblock \emph{SIAM J. Numer. Anal.}, 54:\penalty0 2466--2489,
  2016{\natexlab{a}}.
\newblock \doi{10.1137/16M1074291}.

\bibitem[Guermond and Popov(2016{\natexlab{b}})]{guermond2016a}
J.-L. Guermond and B.~Popov.
\newblock Fast estimation from above of the maximum wave speed in the {R}iemann
  problem for the {E}uler equations.
\newblock \emph{J. Comput. Phys.}, 321:\penalty0 908--926, 2016{\natexlab{b}}.
\newblock \doi{10.1016/j.jcp.2016.05.054}.

\bibitem[Guermond and Popov(2017)]{guermond2017}
J.-L. Guermond and B.~Popov.
\newblock Invariant domains and second-order continuous finite element
  approximation for scalar conservation equations.
\newblock \emph{SIAM J. Numer. Anal.}, 55:\penalty0 3120--3146, 2017.
\newblock \doi{10.1137/16M1106560}.

\bibitem[Guermond et~al.(2014)Guermond, Nazarov, Popov, and Yang]{guermond2014}
J.-L. Guermond, M.~Nazarov, B.~Popov, and Y.~Yang.
\newblock A second-order maximum principle preserving {L}agrange finite element
  technique for nonlinear scalar conservation equations.
\newblock \emph{SIAM J. Numer. Anal.}, 52:\penalty0 2163--2182, 2014.
\newblock \doi{10.1137/130950240}.

\bibitem[Guermond et~al.(2018{\natexlab{a}})Guermond, de~Luna, Popov, Kees, and
  Farthing]{guermond2018a}
J.-L. Guermond, M.~Q. de~Luna, B.~Popov, C.~E. Kees, and M.~W. Farthing.
\newblock Well-balanced second-order finite element approximation of the
  shallow water equations with friction.
\newblock \emph{SIAM J. Sci. Comput.}, 40:\penalty0 A3873--A3901,
  2018{\natexlab{a}}.
\newblock \doi{10.1137/17M1156162}.

\bibitem[Guermond et~al.(2018{\natexlab{b}})Guermond, Nazarov, Popov, and
  Tomas]{guermond2018}
J.-L. Guermond, M.~Nazarov, B.~Popov, and I.~Tomas.
\newblock Second-order invariant domain preserving approximation of the {E}uler
  equations using convex limiting.
\newblock \emph{SIAM J. Sci. Comput.}, 40:\penalty0 A3211--A3239,
  2018{\natexlab{b}}.
\newblock \doi{10.1137/17M1149961}.

\bibitem[Guermond et~al.(2019)Guermond, Popov, and Tomas]{guermond2019}
J.-L. Guermond, B.~Popov, and I.~Tomas.
\newblock Invariant domain preserving discretization-independent schemes and
  convex limiting for hyperbolic systems.
\newblock \emph{Comput. Method. Appl. M.}, 347:\penalty0 143--175, 2019.
\newblock \doi{10.1016/j.cma.2018.11.036}.

\bibitem[Hajduk()]{hajduk2020c}
H.~Hajduk.
\newblock \url{https://github.com/HennesHajduk/PrecMatSimplex}.

\bibitem[Hajduk et~al.(2019)Hajduk, Kuzmin, and Aizinger]{hajduk2019}
H.~Hajduk, D.~Kuzmin, and V.~Aizinger.
\newblock New directional vector limiters for discontinuous {G}alerkin methods.
\newblock \emph{J. Comput. Phys.}, 384:\penalty0 308--325, 2019.
\newblock \doi{10.1016/j.envsoft.2018.01.003}.

\bibitem[Hajduk et~al.(2020{\natexlab{a}})Hajduk, Kuzmin, Kolev, and
  Abgrall]{hajduk2020a}
H.~Hajduk, D.~Kuzmin, T.~Kolev, and R.~Abgrall.
\newblock Matrix-free subcell residual distribution for {B}ernstein finite
  element discretizations of linear advection equations.
\newblock \emph{Comput. Method. Appl. M.}, 359:\penalty0 112658,
  2020{\natexlab{a}}.
\newblock \doi{10.1016/j.cma.2019.112658}.

\bibitem[Hajduk et~al.(2020{\natexlab{b}})Hajduk, Kuzmin, Kolev, Tomov, Tomas,
  and Shadid]{hajduk2020b}
H.~Hajduk, D.~Kuzmin, T.~Kolev, V.~Tomov, I.~Tomas, and J.~N. Shadid.
\newblock Matrix-free subcell residual distribution for {B}ernstein finite
  elements: {M}onolithic limiting.
\newblock \emph{Comput. Fluids}, 200:\penalty0 104451, 2020{\natexlab{b}}.
\newblock \doi{10.1016/j.compfluid.2020.104451}.

\bibitem[Kirby(2011)]{kirby2011}
R.~C. Kirby.
\newblock Fast simplicial finite element algorithms using {B}ernstein
  polynomials.
\newblock \emph{Numer. Math.}, 117:\penalty0 631--652, 2011.
\newblock \doi{10.1007/s00211-010-0327-2}.

\bibitem[Krivodonova(2007)]{krivodonova2007}
L.~Krivodonova.
\newblock Limiters for high-order discontinuous {G}alerkin methods.
\newblock \emph{J. Comput. Phys.}, 226:\penalty0 879--896, 2007.
\newblock \doi{10.1016/j.jcp.2007.05.011}.

\bibitem[Kurganov et~al.(2007)Kurganov, Petrova, and Popov]{kurganov2007}
A.~Kurganov, G.~Petrova, and B.~Popov.
\newblock Adaptive semidiscrete central-upwind schemes for nonconvex hyperbolic
  conservation laws.
\newblock \emph{SIAM J. Sci. Comput.}, 29:\penalty0 2381--2401, 2007.
\newblock \doi{10.1137/040614189}.

\bibitem[Kuzmin(2001)]{kuzmin2001}
D.~Kuzmin.
\newblock Positive finite element schemes based on the flux-corrected transport
  procedure.
\newblock In \emph{Computational Fluid and Solid Mechanics}, pages 887--888.
  Citeseer, 2001.

\bibitem[Kuzmin(2010)]{kuzmin2010}
D.~Kuzmin.
\newblock A vertex-based hierarchical slope limiter for p-adaptive
  discontinuous {G}alerkin methods.
\newblock \emph{J. Comput. Appl. Math.}, 233:\penalty0 3077--3085, 2010.
\newblock \doi{10.1016/j.cam.2009.05.028}.

\bibitem[Kuzmin(2012)]{kuzmin2012a}
D.~Kuzmin.
\newblock Algebraic flux correction {I}. {S}calar conservation laws.
\newblock In \emph{Flux-Corrected Transport: Principles, Algorithms, and
  Applications}, pages 145--192. Springer, 2012.
\newblock \doi{10.1007/978-94-007-4038-9_6}.

\bibitem[Kuzmin(2020{\natexlab{a}})]{kuzmin-preprint3}
D.~Kuzmin.
\newblock Entropy stabilization and property-preserving limiters for
  discontinuous {G}alerkin discretizations of nonlinear hyperbolic equations.
\newblock \emph{arXiv preprint arXiv:2004.03521}, 2020{\natexlab{a}}.

\bibitem[Kuzmin(2020{\natexlab{b}})]{kuzmin2020}
D.~Kuzmin.
\newblock Monolithic convex limiting for continuous finite element
  discretizations of hyperbolic conservation laws.
\newblock \emph{Comput. Method. Appl. M.}, 361:\penalty0 112804,
  2020{\natexlab{b}}.
\newblock \doi{10.1016/j.cma.2019.112804}.

\bibitem[Kuzmin and de~Luna(2020{\natexlab{a}})]{kuzmin-preprint1}
D.~Kuzmin and M.~Q. de~Luna.
\newblock Algebraic entropy fixes and convex limiting for continuous finite
  element discretizations of scalar hyperbolic conservation laws.
\newblock \emph{arXiv preprint arXiv:2003.12007}, 2020{\natexlab{a}}.

\bibitem[Kuzmin and de~Luna(2020{\natexlab{b}})]{kuzmin-preprint2}
D.~Kuzmin and M.~Q. de~Luna.
\newblock Entropy conservation property and entropy stabilization of high-order
  continuous galerkin approximations to scalar conservation laws.
\newblock \emph{arXiv preprint arXiv:2005.08788}, 2020{\natexlab{b}}.

\bibitem[Kuzmin and Quezada~de Luna(2020)]{kuzmin2020a}
D.~Kuzmin and M.~Quezada~de Luna.
\newblock Subcell flux limiting for high-order {B}ernstein finite element
  discretizations of scalar hyperbolic conservation laws.
\newblock \emph{J. Comput. Phys.}, 411:\penalty0 109411, 2020.
\newblock \doi{10.1016/j.jcp.2020.109411}.

\bibitem[Kuzmin and Turek(2002)]{kuzmin2002}
D.~Kuzmin and S.~Turek.
\newblock Flux correction tools for finite elements.
\newblock \emph{J. Comput. Phys.}, 175:\penalty0 525--558, 2002.
\newblock \doi{10.1006/jcph.2001.6955}.

\bibitem[Kuzmin et~al.(2012)Kuzmin, L{\"o}hner, and Turek]{kuzmin2012}
D.~Kuzmin, R.~L{\"o}hner, and S.~Turek.
\newblock \emph{Flux-corrected transport: principles, algorithms, and
  applications}.
\newblock Springer, 2012.

\bibitem[Lohmann(2019)]{lohmann2019}
C.~Lohmann.
\newblock \emph{Physics-Compatible Finite Element Methods for Scalar and
  Tensorial Advection Problems}.
\newblock Springer Spektrum, 2019.
\newblock \doi{10.1007/978-3-658-27737-6}.

\bibitem[Lohmann et~al.(2017)Lohmann, Kuzmin, Shadid, and Mabuza]{lohmann2017}
C.~Lohmann, D.~Kuzmin, J.~N. Shadid, and S.~Mabuza.
\newblock Flux-corrected transport algorithms for continuous {G}alerkin methods
  based on high order {B}ernstein finite elements.
\newblock \emph{J. Comput. Phys.}, 344:\penalty0 151--186, 2017.
\newblock \doi{10.1016/j.jcp.2017.04.059}.

\bibitem[Pazner(2020)]{pazner-preprint}
W.~Pazner.
\newblock Sparse invariant domain preserving discontinuous {G}alerkin methods
  with subcell convex limiting.
\newblock \emph{arXiv preprint arXiv:2004.08503}, 2020.

\bibitem[Sod(1978)]{sod1978}
G.~A. Sod.
\newblock A survey of several finite difference methods for systems of
  nonlinear hyperbolic conservation laws.
\newblock \emph{J. Comput. Phys.}, 27:\penalty0 1--31, 1978.
\newblock \doi{10.1016/0021-9991(78)90023-2}.

\bibitem[Vilar(2019)]{vilar2019}
F.~Vilar.
\newblock A posteriori correction of high-order discontinuous {G}alerkin scheme
  through subcell finite volume formulation and flux reconstruction.
\newblock \emph{J. Comput. Phys.}, 387:\penalty0 245--279, 2019.
\newblock \doi{j.jcp.2018.10.050}.

\bibitem[Woodward and Colella(1984)]{woodward1984}
P.~Woodward and P.~Colella.
\newblock The numerical simulation of two-dimensional fluid flow with strong
  shocks.
\newblock \emph{J. Comput. Phys.}, 54:\penalty0 115--173, 1984.
\newblock \doi{10.1016/0021-9991(84)90142-6}.

\bibitem[Zalesak(1979)]{zalesak1979}
S.~T. Zalesak.
\newblock Fully multidimensional flux-corrected transport algorithms for
  fluids.
\newblock \emph{J. Comput. Phys.}, 31:\penalty0 335--362, 1979.
\newblock \doi{10.1016/0021-9991(79)90051-2}.

\bibitem[Zhang and Shu(2010)]{zhang2010}
X.~Zhang and C.-W. Shu.
\newblock On maximum-principle-satisfying high order schemes for scalar
  conservation laws.
\newblock \emph{J. Comput. Phys.}, 229:\penalty0 3091--3120, 2010.
\newblock \doi{10.1016/j.jcp.2009.12.030}.

\bibitem[Zienkiewicz and Ortiz(1995)]{zienkiewicz1995}
O.~C. Zienkiewicz and P.~Ortiz.
\newblock A split-characteristic based finite element model for the shallow
  water equations.
\newblock \emph{Int. J. Numer. Meth. Fl.}, 20:\penalty0 1061--1080, 1995.
\newblock \doi{10.1002/fld.1650200823}.

\end{thebibliography}
\addcontentsline{toc}{section}{Bibliography}
\end{document}